\pdfoutput=1
\documentclass[10pt,a4paper]{article}

\usepackage[utf8]{inputenc}
\usepackage[english]{babel}
\usepackage{amsmath}
\usepackage{amsfonts}
\usepackage{amssymb}
\usepackage{mathrsfs}
\usepackage{amsthm}
\usepackage{latexsym,amsmath,amstext,amssymb,mathtools}
\usepackage{caption}
\usepackage{commath} 
\usepackage{hhline}
\usepackage{listings}
\usepackage{enumerate}
\usepackage{color}
\usepackage{xfrac}
\usepackage{faktor}
\usepackage{colonequals}
\usepackage{xr} 
\usepackage[toc,page]{appendix} 
\usepackage{hyphenat}
\usepackage{stmaryrd} 
\usepackage{csquotes}
\usepackage{array} 
\usepackage{indentfirst}
\usepackage{float}
\usepackage{capt-of}   

\usepackage[hidelinks]{hyperref}
\usepackage{cleveref}

\usepackage{tikz}
\usepackage{standalone}
\usepackage{forest}
\usepackage{tikz-cd}
\tikzset{
    lablrot/.style={anchor=north, rotate=-90, inner sep=1.5mm},
    lablrotsouth/.style={anchor=south, rotate=-90, inner sep=1.5mm}
}

\numberwithin{equation}{section}

\theoremstyle{plain}
\newtheorem{theorem}[equation]{Theorem}
\newtheorem*{theorem*}{Theorem}

\newtheorem{proposition}[equation]{Proposition}
\newtheorem{lemma}[equation]{Lemma}
\newtheorem*{lemma*}{Lemma}
\newtheorem*{proposition*}{Proposition}
\newtheorem*{corollary*}{Corollary}

\newtheorem{corollary}[equation]{Corollary}

\newtheorem{question}[equation]{Question}

\newtheorem{maintheorem}{Theorem}

\newtheorem{introexample}{Example}
\newtheorem{introcorollary}{Corollary}
\newtheorem{introproposition}{Proposition}

\theoremstyle{definition}
\newtheorem{definition}[equation]{Definition}

\theoremstyle{remark}
\newtheorem*{remark}{Remark}

\newcommand{\dd}{\mathrm{d}}
\newcommand{\BO}{\mathcal{O}}

\newcommand\HH{\mathrm{H}}

\newcommand{\hh}{\mathrm{h}}

\newcommand{\CC}{\mathbb{C}}
\newcommand{\ZZ}{\mathbb{Z}}
\newcommand{\NN}{\mathbb{N}}
\newcommand{\QQ}{\mathbb{Q}}
\newcommand{\PP}{\mathbb{P}}
\newcommand{\dv}{\partial}

\newcommand\restr[1]{\raisebox{-.5ex}{$|$}_{#1}}
\newcommand{\sing}{\mathrm{sing}}
\newcommand{\Sing}{\mathrm{Sing}}

\newcommand{\GG}{\mathcal{G}}

\newcommand{\EE}{\mathscr{E}}

\DeclareMathOperator{\Spec}{Spec}

\DeclareMathOperator{\divv}{div}

\DeclareMathOperator{\Ext}{Ext}
\newcommand{\sheafhom}{\mathscr{H}\kern -.5pt om}

\newcommand{\Nm}{\mathrm{Nm}}
\newcommand{\Pic}{\mathrm{Pic}}

\DeclareMathOperator{\Prym}{Prym}
\newcommand{\cC}{\mathcal{C}}
\newcommand{\eE}{\mathcal{E}}

\newcommand{\dec}{\mathrm{dec}}

\newcommand{\PPdelta}{\PP_{\delta',\delta''}}
\newcommand{\F}{\mathcal{F}}

\newcommand{\ud}{{\underline{d}}}

\newcommand{\Nd}{{N_{\underline{d}}}}

\newcommand{\pP}{\mathcal{P}}
\newcommand{\SE}{\mathscr{S}}
\newcommand{\BE}{\mathscr{BE}}

\newcommand{\pt}{\mathrm{pt}}
\newcommand{\Ker}{\mathrm{Ker}}

\newcommand{\mult}{\mathrm{mult}}
\newcommand{\Res}{\mathrm{Res}}

\newcommand{\rank}{\mathrm{rank}}

\usepackage{subfiles}

\usepackage{microtype} 
\usepackage[backend=biber, style=alphabetic, giveninits=true, maxalphanames=99, doi=false,isbn=false,url=false]{biblatex}
\newcommand{\Rel}{\mathrm{Rel}}
\newcommand{\Card}{\mathrm{Card}}
\newcommand{\ns}{N_\mathrm{ad}}
\newcommand{\tns}{\ns^\mathrm{is}}
\newcommand{\nss}{N}

\newcommand{\Xiasing}{\Xi_{\sing,\mathrm{ad}}}
\newcommand{\Xiasingis}{\Xi^\mathrm{is}_{\sing,\mathrm{ad}}}

\newcommand{\Wsingis}{W_\sing^\mathrm{is}}

\newcommand{\ue}{{\underline{e}}}

\newcommand{\Part}{\mathscr{P}}

\title{The Gauss map for bielliptic Prym varieties}
\author{Constantin Podelski}

\begin{document}
\maketitle
\begin{abstract}
    We completely describe the degree of the Gauss map of the theta divisor of bielliptic Prym varieties. We characterize bielliptic Prym varieties whose Gauss degree is the same as Jacobians. We also construct bielliptic Prym varieties with a very low Gauss degree. In dimension $5$, we obtain a complete description of the Gauss degree on the Andreotti-Mayer locus.
\end{abstract}

\tableofcontents
\section{Introduction}

Let $\mathcal{A}_g$ denote the moduli space of principally polarized abelian varieties (ppav's for short) over the complex numbers. For $(A,\Theta)\in \mathcal{A}_g$, the Gauss map
\begin{align*}
    \GG: \Theta &\dashrightarrow \PP(T^\vee_0 A) \\
     x &\mapsto \tau_{x}^\ast T_x \Theta
\end{align*}
attaches to a smooth point $x\in \Theta_{\mathrm{sm}}$ its tangent space translated to the origin, viewed as a hyperplane in $T_0 A$. Here $\tau_x:A\to A$ is the translation by $x$. The Gauss map has many fascinating properties: For instance, it is generically finite if and only if $(A,\Theta)$ is  indecomposable as a ppav \cite[Sec. 4.4]{Birkenhake2004}. Thus, it is interesting to compute its generic degree. This degree was unknown until now beyond a few cases: 
\begin{itemize}
    \item For smooth theta divisors of $g$-dimensional ppav's the degree is $g!$.
        \item For non-hyperelliptic (resp. hyperelliptic) Jacobians of smooth genus $g$ curves the degree is $b_{g-1}\coloneqq \binom{2g-2}{g-1}$ (resp. $2^{g-1}$) \cite[247]{arbarello}.
        \item For a general Prym variety the degree is $D(g+1)+2^{g-2}$, where $D(g)$ is the degree of the variety of all quadrics of rank $\leq 3$ in $\PP^{g-1}$ \cite{Verra98}.
    \item For the intermediate Jacobian of a cubic threefold the degree is $72$ \cite{ClemensGriffiths1972IJCubicThree}.
    \item Recently, we computed the Gauss degree on a general ppav of some of the irreducible components of the Andreotti-Mayer loci \cite{podelski2023GaussArt} (see \ref{Equ: Intro: Gauss degree on Agt}).
\end{itemize}
Denote by $\BE_g\subset \mathcal{A}_g$ the locus of \emph{bielliptic Prym varieties}, i.e. the closure in $\mathcal{A}_g$ of the set of Prym varieties arising from étale double covers of bielliptic curves. Our main result is the complete description of the Gauss degree on $\BE_g$. For $g\geq 5$, the locus $\BE_g$ has $\lfloor g/2 \rfloor+1$ irreducible components denoted by $\EE_{g,0},\dots,\EE_{g,\lfloor g/2 \rfloor}$. For $0\leq t \leq g/2$, we define open subsets of $\EE_{g,t}'\subset \EE_{g,t}$ (Definition \ref{Def of Egt prime}), such that we have the following:
\begin{maintheorem}[\ref{Thm: degree Gauss Map on Egt, in general} and \ref{Prop: Gauss degree on Egt'}]\label{maintheorem: degree Gauss Map on Egt, in general}
For $0\leq t \leq g/2$ and $(P,\Xi)=\mathrm{Prym}(\tilde{C}/C)\in \EE'_{g,t}$, we have
\[ \deg \GG_\Xi= \begin{cases}
    b_{g-1}-2\ns(\tilde{C}/C)\,, &\text{if $t=0$,}\\
    b_{t-1}b_{g-t}+b_tb_{g-t-1}-2^{g-1}-2\ns(\tilde{C}/C)\,, &\text{if $t>0$,}
\end{cases}\]
where $b_n=\binom{2n}{n}$ denotes the middle binomial coefficient and $\ns(\tilde{C}/C)$ is a combinatorial invariant of double covers of bielliptic curves (see \ref{Def: eta}).
\end{maintheorem}
\begin{remark}
For a general $(P,\Xi)\in \EE_{g,t}'$ we have $\ns(\tilde{C}/C)=0$. In particular, a general ppav in $\EE_{g,0}$ has the same Gauss degree as non-hyperelliptic Jacobians. The number $\ns(\tilde{C}/C)$ corresponds to the number of additional singularities of $\Xi$, compared to a general member of $\EE_{g,t}'$. Under some mild assumptions, these additional singularities are isolated singularities of maximal rank (Proposition \ref{Prop: singularities are quadratic of maximal rank}), and are at $2$-torsion points. This provides an interesting example of RST-exceptional singularities in the sense of \cite{casalainamartin2009}\footnote{The tangent cone of $\tilde{\Theta}$ is contained in the tangent space to $P$ at these singularities.} where we can compute the tangent cone to the Prym theta divisor. This is a difficult question in general (see \cite[Question 2.9]{casalainamartin2012singularities}). Also, the rank of these quadratic the singularity is maximal, in contrast to quadratic RST-stable singularities whose rank is at most $4$.  
\end{remark}
Recall the Andreotti-Mayer loci defined by
\[ \mathcal{N}^{(g)}_{k} \coloneqq \{(A,\Theta)\in\mathcal{A}_g \,|\, \dim \Sing(\Theta) \geq k \} \subset \mathcal{A}_g \,, \quad \text{for $k\geq 0$.}\]
We will drop the supscript when it is clear from the context. Andreotti and Mayer \cite{Andreotti1967} show that the Jacobian locus $\mathcal{J}_g$ is an irreducible component of $\mathcal{N}_{g-4}$. But $\mathcal{N}_{g-4}$ has more components: For $g\geq 5$, the loci $\EE_{g,0}$ and $\EE_{g,1}$ as well as the loci $\mathcal{A}^2_{t,g-t}$ for $2\leq t\leq g/2$ are irreducible components of $\mathcal{N}_{g-4}$, distinct from $\mathcal{J}_g$ and the components of decomposable ppav's \cite{Donagi1981tetragonal} \cite{Debarre1988} (we refer to Debarre's paper for a definition of $\mathcal{A}^2_{t,g-t}$). In \cite{podelski2023GaussArt} we show that for $2\leq t \leq g/2$ and a general $(A,\Theta)\in\mathcal{A}^2_{t,g-t}$, we have
\begin{equation}\label{Equ: Intro: Gauss degree on Agt}
     \deg \GG_\Theta= t! (g-t)! g \,, 
\end{equation}
thus the degree is distinct from the Gauss degree on Jacobians. We thus obtain the following immediate corollary to Theorem \ref{maintheorem: degree Gauss Map on Egt, in general}:
\begin{introcorollary}
    Amongst the known irreducible components of the Andreotti-Mayer locus $\mathcal{N}_{g-4}$, the only component besides $\mathcal{J}_g$ whose general ppav has the same Gauss degree as Jacobians is $\EE_{g,0}$.
\end{introcorollary}
In analogy to the Andreotti-Mayer loci, Codogni, Grushevski and Sernesi \cite{Gru17} define the \emph{Gauss loci} by
\[ \mathcal{G}^{(g)}_d \coloneqq \{ (A,\Theta)\in \mathcal{A}_g\,|\, \deg \GG_\Theta \leq d\} \subset \mathcal{A}_g \,, \quad \text{for $d\geq 0$.} \]
By Codogni and Krämer \cite{KraemerCodogni}, the Gauss loci are closed in $\mathcal{A}_g$ and the Jacobian locus $\mathcal{J}_g$ is an irreducible component of $\mathcal{G}_{b_{g-1}}$. We obtain the following corollary to Theorem \ref{maintheorem: degree Gauss Map on Egt, in general}:
\begin{introcorollary}[\ref{Cor: EEg0 irreducible component of Gauss locus}]\label{IntroCor: EEg0 irreducible component of Gauss locus}
    For $g\geq 4$, the locus $\EE_{g,0}$ is an irreducible component of $\GG_{b_{g-1}}$, distinct from $\mathcal{J}_g$.
\end{introcorollary}
We also study degenerations in $\EE_{g,t}$. For $\ud\in \mathscr{P}_g$ a partition of $g$, we define a locus $\SE_\ud\subset \BE_g$ (Definition \ref{Def: SE ud}), corresponding to degenerations of bielliptic Prym varieties where the elliptic curve becomes singular. Together with the loci $\EE'_{g,t}$ these cover all of $\BE_g$ apart from the intersection with Jacobians and decomposable ppav's (Lemma \ref{Lemma: decomposition of Egt}). For a reduced partition $\ud=(d_1,\dots,d_n)\in \Part_g$ (i.e. we assume with $d_i>0$ for $1\leq i \leq n$), we define
\begin{equation}\label{Equ: Intro: definition of mu ud}
    \mu_\ud \coloneqq \frac{1}{2} b_{d_1-1}\cdots b_{d_n-1} \sum_{\epsilon\in \{0,1\}^{n} } \frac{(-2)^{l(\epsilon)}(n-l(\epsilon)) b_{n-l(\epsilon)}}{d_1^{\epsilon_1}\cdots d_n^{\epsilon_n}} \,. 
\end{equation}
where $l(\epsilon)\coloneqq \#\{i\,|\,\epsilon_i\neq 0\}$ denotes the length of the partition (see Section \ref{Sec: App: Computation of the coefficient mu} for alternate forms of $\mu_\ud$). We then make the following computation:
\begin{maintheorem}[\ref{Thm: Degree Egt with E cycle of P1's}]\label{maintheorem: Degree Egt with E cycle of P1's}
For $g\geq 0$, $\ud\in \Part_g$ be a partition of $g$, and $(P,\Xi)=\Prym(\tilde{C}/C)\in \SE_\ud$, we have
\[ \deg(\GG_\Xi)= \mu_\ud-2\ns(\tilde{C}/C) \,.\]
\end{maintheorem}
\begin{introexample}
In some situations, the expression of $\mu_\ud$ becomes much simpler. 
\begin{itemize}
\item If $l(\ud)=1$, we have
\[  \mu_g=b_{g-1} \,. \]
In particular, for a general $(P,\Xi)\in \SE_{{g}}$, the degree of the Gauss map is the same as for non-hyperelliptic Jacobians.
\item If $l(\ud)=2$, we have
\[ \mu_{(d_1,d_2)}=\left(6-\frac{2(d_1+d_2)}{d_1 d_2}\right)b_{d_1-1} b_{d_2-1}  \,. \]
In particular, if $d_1=1$ and $d_2=g-1$, we have $\mu_{1,g-1}=b_{g-1}$. Thus, for a general $(P,\Xi)\in \SE_{{1,g-1}}$, the degree of the Gauss map is the same as for non-hyperelliptic Jacobians.
\item At the other extreme, if $l(\ud)=g$ and $\underline{d}=(1^g)=(1,\dots,1)$, we have
\[ \mu_\ud=g\binom{g-1}{\lfloor g/2 \rfloor} \,. \]
This is the lowest possible value of $\mu_\ud$ for $\ud\in \Part_g$.
\end{itemize}
\end{introexample}
When investigating how to maximize $\ns(\tilde{C}/C)$ in the last example, we discover the following ppav which gives the lowest Gauss degree that we know of besides hyperelliptic Jacobians:
\begin{introproposition}[\ref{Sec: App: A special ppav in arbitrary dimension}]\label{Introprop: Generalization of Varleys fourfold}
    There is a Prym $(P,\Xi)\in \SE_{(1^g)}$ with 
    \[ \deg \GG_\Xi = \begin{cases}
        g \binom{g-1}{g/2}-2^{g-1} & \text{if $g$ is even,} \\
        g \binom{g-1}{\lfloor g/2 \rfloor }-\frac{2^g-2}{3} &\text{ if $g$ is odd.}
    \end{cases}\]
\end{introproposition}
If $g=4$, the Gauss degree in the above proposition is $4$ and we thus recover Varley's fourfold. For $g\geq 5$, the ppav in the above proposition has $2^{g-2}$ (resp. $(2^{g-1}-1)/3$) isolated singularities at $2$-torsion points, in the even (resp. odd) case.
\par 
For $k\geq 0$, we make the two following notations (in analogy to the $k$-th theta-null locus)
    \begin{align*} 
    \EE_{g,t }^{\prime k} &\coloneqq \{ \Prym(\tilde{C}/C)\in \EE'_{g,t}\,|\, \ns(\tilde{C}/C)=k \}\,, \\
    \SE_\ud^k &\coloneqq \{ \Prym(\tilde{C}/C)\in \SE_\ud \,|\, \ns(\tilde{C}/C)=k \} \,. 
    \end{align*} 
According to the remark above, for $g\geq 5$, a general member of these loci has exactly $k$ isolated singularities, and these singularities are at $2$-torsion points. In dimensions $4$ and $5$, the locus of bielliptic Prym varieties completely covers the Andreotti-Mayer locus $\mathcal{N}^{(g)}_{g-4}$ \cite{Beauville1977} \cite{Debarre1988}. We thus have the following immediate corollaries to Theorem \ref{maintheorem: degree Gauss Map on Egt, in general} and \ref{maintheorem: Degree Egt with E cycle of P1's}:
\begin{introcorollary}
    The locus of ppav's of dimension $4$ having the same Gauss degree as non-hyperelliptic Jacobians is
    \[\mathcal{J}_4^\mathrm{nh} \cup \EE^{\prime 0}_{4,0} \cup \SE^0_4 \cup \SE^0_{1,3} \,.\]
\end{introcorollary}

\begin{introcorollary}
    The subset of $\mathcal{N}^{(5)}_1$ of ppav's with the same Gauss degree as non-hyperelliptic Jacobians is
    \[ \mathcal{J}_5^\mathrm{nh} \cup \EE^{\prime 0}_{5,0} \cup \SE^0_5 \cup \SE^0_{1,4}\,.\]
\end{introcorollary}
In general, Theorem \ref{maintheorem: degree Gauss Map on Egt, in general} and \ref{maintheorem: Degree Egt with E cycle of P1's} along with \ref{Cor: Gauss degree on EEg,1 greater than Jacobians} imply the following:
\begin{introcorollary}
    For $g\geq 5$, the locus of bielliptic Prym varieties having the same Gauss degree as non-hyperelliptic Jacobians is
    \[ \left(\BE_g\cap \mathcal{J}_g^\mathrm{nh}\right)\cup \EE^{\prime 0}_{g,0} \cup \SE^0_g \cup \SE^0_{1,g-1} \,.\]
\end{introcorollary}

This text is organized as follows: In Section \ref{Sec: The family Egt}, we recall the construction of the loci $\EE_{g,t}$, and prove the main part of Theorem \ref{maintheorem: degree Gauss Map on Egt, in general}. In Section \ref{sec: Pryms arising from admissible covers}, we study degenerations in $\EE_{g,t}$ and define the loci $\SE_\ud$. We then prove Theorem \ref{maintheorem: Degree Egt with E cycle of P1's}. Finally, in Section \ref{Sec: additional isolated Singularities}, we study the possible values of $\ns(\tilde{C}/C)$, and obtain a complete description of this invariant in dimension $5$. We then prove Proposition \ref{Introprop: Generalization of Varleys fourfold}. \par 
We work over the field of complex numbers.
\par 
\textbf{Acknowledgements.} The author would like to thank his supervisor Thomas Krämer for his continuous support and the many discussions that led to this paper. 

\subsection{Summary of the results in dimension \texorpdfstring{$5$}{5}}
We obtain a complete description of the degree of the Gauss map on $\mathcal{N}^{(5)}_1$. For an irreducible locus $Z\subset \mathcal{A}_g$, let $\deg \GG(Z)$ be the degree of the Gauss map on a general ppav $(A,\Theta)\in Z$. We have by Theorem \ref{maintheorem: degree Gauss Map on Egt, in general} and Corollary \ref{Cor: possible values of eta for Egt and SEud}:
\renewcommand{\arraystretch}{1.3}
\[ \begin{array}{|c|c|c|c|c|c|}
\hline 
Z & \mathcal{J}_5 & \mathcal{A}_5^{\mathrm{dec}} & \EE_{5,0}' & \EE_{5,1}' & \EE_{5,2}' \\
\hline 
\deg \GG(Z) & 70 & 0 & 70 & 94 & 60\\
\hline 
\max_Z \ns & & &11 & 9 & 10 \\
\hline 
\end{array}\,.\]
By Theorem \ref{maintheorem: Degree Egt with E cycle of P1's}, the degree on the boundary sets $\SE_\ud$ is:
\[
\begin{array}{|c|c|c|c|c|c|c|c|}
\hline 
    \ud &(1^5) & (1^3,2) &(1,1,3) &(1,2,2)&(1,4)&(2,3)&(5) \\
    \hline
    \deg \GG(\SE_\ud) & 30 & 36 & 52 & 44 & 70 & 52 & 70 \\
    \hline 
    \max_{\SE_\ud} \ns &  5 & 6&7 &8 &9 &10 &10\\
    \hline 
\end{array}\,.
\]
\renewcommand{\arraystretch}{1}
In both tables $\max_Z\ns$ denotes the highest value of $\ns(\tilde{C}/C)$ for $(P,\Xi)=\mathrm{Prym}(\tilde{C}/C)\in Z$. For every locus $Z$ in the tables above and every $0\leq k\leq \max_Z \ns$, there is a ppav $(P,\Xi)\in Z$ with $\ns(\tilde{C}/C)=k$ (\ref{Cor: possible values of eta for Egt and SEud}). In particular, on $\mathcal{N}^{(5)}_1$ the Gauss map can have the degrees
\[ \{0,16\}\cup\{20,22,\dots,92\}\setminus \{72,74\} \,.\]
On $\mathcal{A}_5\setminus \mathcal{N}^{(5)}_0$, the $\Theta$ divisors are smooth thus 
\[ \deg \GG = g! \,. \] 
On $\mathcal{N}^{(5)}_0\setminus \mathcal{N}^{(5)}_1$, i.e. on ppav's $(A,\Theta)\in \mathcal{A}_5$ such that $\Theta$ has isolated singularities, it is well-known \cite[rem. 2.8]{Gru17} that
\[ \deg \GG=g!-\sum_{z\in \Sing(\Theta)} \mult_z \Theta \,. \]
    where $\mult_z \Theta$ is the Samuel multiplicity as defined in \cite[Sec. 4.3]{Fulton1998}. A particular example is the locus of intermediate Jacobians of cubic threefolds $IJ_5$ whose theta divisor has a unique cubic isolated singularity and whose Gauss degree is $72$ \cite{ClemensGriffiths1972IJCubicThree}. We do not know the complete list of possible $0$-dimensional singularities are for $(A,\Theta)\in \mathcal{A}_5$. This is the last unknown piece to completely describe the degree of the Gauss map on $\mathcal{A}_5$:  
\begin{question}
    For $(A,\Theta)\in \mathcal{A}_5$, with $\Theta$ having isolated singularities, what are the possible values of
    \[ \sum_{z\in \Sing(\Theta)} \mult_z \Theta \,, \]
    where $\mult_z \Theta$ is the Samuel multiplicity.
\end{question}
    The graph on the following page summarizes the results. An arrow from $Z$ to $W$ means that $W\subset \overline{Z}$. The thick red lines denote all the loci $\EE_{g,t }^{\prime k}$ (resp. $\SE_\ud^{k}$) for $0\leq k \leq \max_Z \ns $. We have omitted certain inclusions in the closure, like 
\[ \SE_{\underline{d}}^k \subset \overline{\SE_{\underline{e}}^k} \]
for all subdivision $\ud$ of the partition $\underline{e}$ and $k>0$. There might also be other inclusions we do not know about. $IJ_5$ is in parentheses as it is in $\mathcal{N}^{(5)}_0$ and not $\mathcal{N}^{(5)}_1$.

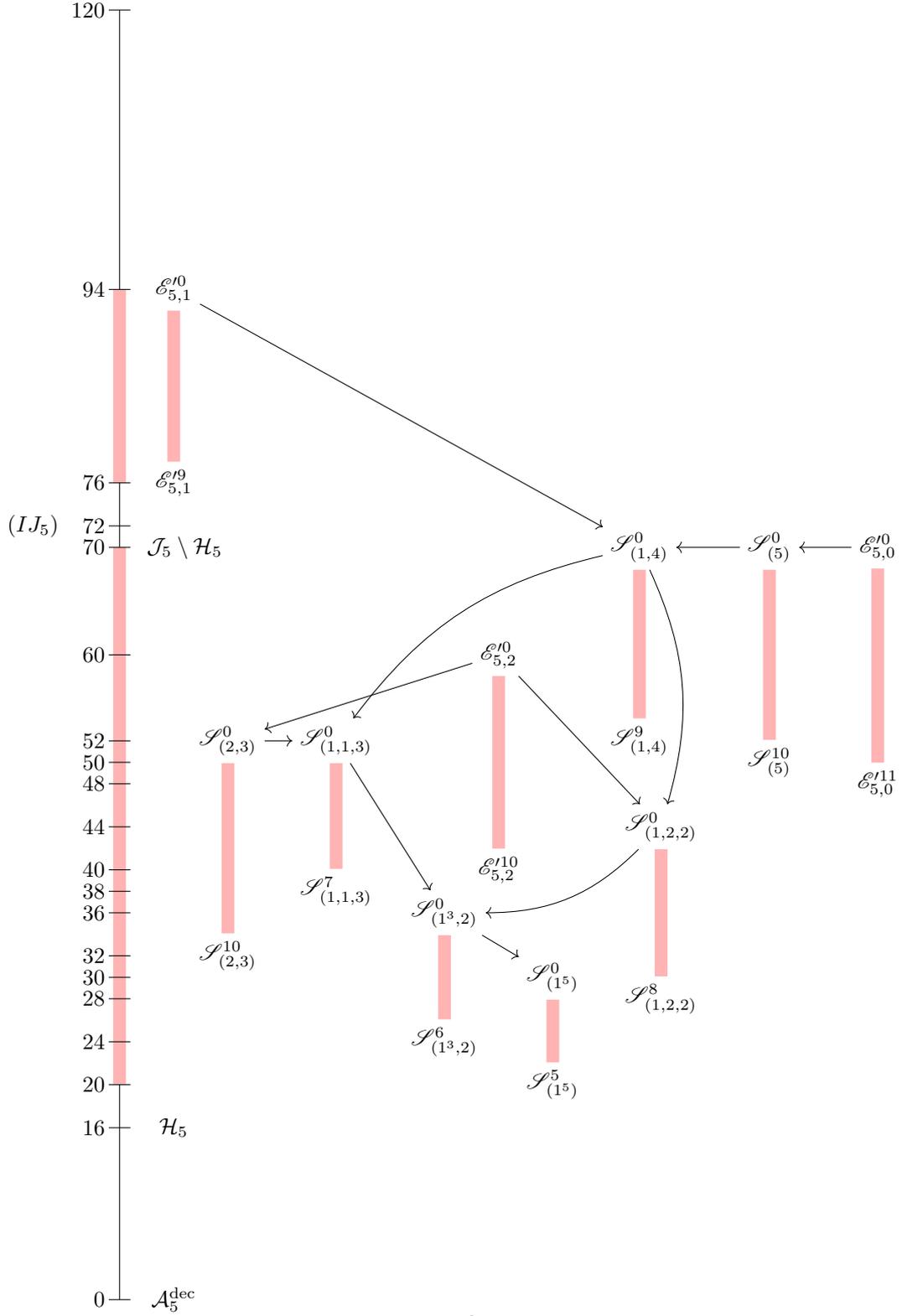
\begin{figure}
\caption{Possible Gauss degrees on $\mathcal{N}^{(5)}_1$}
\hspace*{-4em}{
\begin{tikzpicture}[scale=0.17,
deg/.style={circle, draw=black, right=7,scale=.7},
line/.style={line width=0.2cm, red!30},
arr/.style={->}]

 \draw (0,0) -- (0,120);

\node (Adec) at (5,0) {$\mathcal{A}_5^{\mathrm{dec}}$};
\node (H5) at (5,16) {$\mathcal{H}_5$};
\node(S1112) at (30,36) {$\SE_{(1^3,2)}^0$};
\node(S1112e) at (30,24) {$\SE_{(1^3,2)}^{6}$};
\node(S113e) at (20,38) {$\SE_{(1,1,3)}^{7}$};
\node(S113) at (20,52) {$\SE_{(1,1,3)}^{0}$};
\node (S23) at (10,52) {$\SE_{(2,3)}^0$};
\node (S23e) at (10,32) {$\SE_{(2,3)}^{10}$};
\node (S122e) at (50,28) {$\SE_{(1,2,2)}^{8}$};
\node (S122) at (50,44)  {$\SE_{(1,2,2)}^0$};
\node(S14e) at (48,52) {$\SE_{(1,4)}^{9}$};
\node(S14) at (48,70) {$\SE_{(1,4)}^0$};
\node(S5) at (60,70)  {$\SE_{(5)}^0$};
\node(S5e) at (60,50) {$\SE_{(5)}^{10}$};
\node(E5) at (70,70)  {$\EE_{5,0}^{\prime 0}$};
\node(E5e) at (70,48) {$\EE_{5,0}^{\prime 11}$};
\node(J5) at (6,70) {$\mathcal{J}_5\setminus \mathcal{H}_5$};
\node(E51) at (5,94) {$\EE_{5,1}^{\prime 0}$};
\node (E51e) at (5,76) {$\EE_{5,1}^{\prime 9}$};
\node (E52) at (35,60) {$\EE_{5,2}^{\prime 0}$};
\node (E52e) at (35,40) {$\EE_{5,2}^{\prime 10 }$};
\node (IJ) at (-8,72) {$\left( IJ_5\right)$};
\node (S1e) at (40,20){$\SE_{(1^5)}^{5}$}; 
\node (S1) at (40,30) {$\SE_{(1^5)}^0$};

\draw [line] (S5) -- (S5e);
\draw [line] (S1112) -- (S1112e);
\draw [line] (S113) -- (S113e);
\draw[line] (S14) -- (S14e);
\draw [line] (S122) -- (S122e);
\draw [line] (E52) -- (E52e);
\draw[line] (E51) -- (E51e);
\draw [line] (E5) -- (E5e);
\draw [line] (S23) -- (S23e);
\draw [line] (0,94) -- (0,76);
\draw [line] (0,70) -- (0,20);
\draw [line] (S1) -- (S1e);

\draw[arr] (E52) -- (S23);
\draw[arr] (S23) -- (S113);
\draw[arr] (S113) -- (S1112);

\draw[->, out=225, in=0] (S122) to (S1112.east);
\draw [arr] (E52) -- (S122);
\draw[->, bend right=20] (S14) to (S113);
\draw [->, bend left=20] (S14) to (S122);
\draw[arr] (S5) -- (S14);
\draw[arr] (E5) -- (S5);
\draw[arr] (E51) -- (S14);
\draw[arr] (S1112) -- (S1);

  \foreach \i in {0,16,20,24,28,30,32,36,38,40,44,48,50,52,60,70,72,76,94,120}{
	\draw (0,\i) node[left=3pt] {\i} ;
    \draw (1,\i) -- (-1,\i); 
} Lemma decomp in intro
\end{tikzpicture}
}
\vspace{-40pt}
\end{figure}

\newpage 
\section{The family \texorpdfstring{$\EE_{g,t}^\ast$}{Egt}}\label{Sec: The family Egt}


\subsection{The construction}\label{Sec: Construction of the Families Eg,t}

We recall the construction of the bielliptic Prym locus, following \cite{Debarre1988}. Other references are \cite{Beauville1977}, \cite{Donagi1981tetragonal} and \cite{Naranjo1992BiellipticPryms}. Let $C$ be a smooth irreducible projective curve of genus $g+1$ that is bielliptic, i.e. admits a double cover $p:C\to E$ to an elliptic curve. Let $\pi:\tilde{C}\to C$ be an étale double cover. Let
\begin{align*} 
P&\coloneqq \{ L\in \Pic^{2g-2}(\tilde{C})\,|\, \Nm_{\pi}(L)=\omega_C\,, \,\hh^0(C,L) \text{ even} \}\subset \Pic^{2g-2}(\tilde{C})\,, \\
\Xi &\coloneqq \{L\in P\,|\, \hh^0(C,L)>0 \} \subset P \,,
\end{align*}    
be the associated Prym variety, with its principal polarization. Suppose that the Galois group of the composite $p\circ \pi$ is $(\ZZ/2\ZZ)^2$, then the two other intermediate quotients induce a tower of smooth projective curves
\begin{equation}\label{Diagramm: E_g,t tower}
    \begin{tikzcd}
        & \tilde{C} \arrow[dl,"\pi"'] \arrow[d,"\pi'"] \arrow[dr,"\pi''"] \arrow[dd,bend right, "f"'] & \\
        C \arrow[dr,"p"'] & C'\arrow[d,"p'"] & C'' \arrow[dl,"p''"] \\
        & E & 
    \end{tikzcd}\,.
\end{equation}
We can assume $g(C')=t+1$ and $g(C'')=g-t+1$ for some $0\leq t \leq g/2$. For $i \in \{\emptyset,\,',\,''\}$, the morphism $p^i$ corresponds to an involution $\tau^i$ on $C^i$, a branch divisor $\Delta^i$, a ramification divisor $R^i$, and a line bundle $\delta^i$ such that $\Delta^i \in |2\delta^i|$.
\begin{align*}
    \Delta'=Q_{1}+\cdots+Q_{2t}\,, \quad \text{and} \quad  \Delta''&=Q_{2t+1}+\cdots+Q_{2g}\,.
\end{align*}
We have $\Delta=\Delta'+\Delta''$ and $\delta=\delta'+\delta''$. Let $P_1,\dots,P_{2g}$ be the corresponding ramification points of $p$ on $C$. The étale covering $\pi$ is associated to the following $2$-torsion point of $\Pic^0(C)$
\[ \eta= P_1+\cdots+P_{2t}-p^\ast \Delta'=P_{2t+1}+\cdots+P_{2g}-p^\ast \Delta'' \,. \]
Denote by $\sigma$ the involution of $\tilde{C}$ induced by $\pi$. The morphism $\pi'$ (resp. $\pi''$) is branched on $p'^\ast(\Delta'')$ (resp. $p''^\ast(\Delta')$), and corresponds to an involution $\sigma'$ (resp. $\sigma''=\sigma \sigma'$) on $\tilde{C}$.\\
 Note that the above data is equivalent to the choice of an elliptic curve $E$, a line bundle $\delta\in \Pic^g(E)$, a reduced divisor $\Delta\in |2\delta|$, the choice of a $2t$ subset $\Delta'\leq \Delta$, and a line bundle $\delta'$ such that $\Delta'\in |2\delta'|$.
\begin{definition}[The Family $\EE_{g,t}$]\label{Def: Definition of Egt}
For $0\leq t \leq g/2$, let $\EE_{g,t}^\ast \subset \mathcal{A}_g$ be the set of Prym varieties $\Prym(\tilde{C}/C)$ obtained in this way. If $t=2$, we require additionally that $C'$ is non-hyperelliptic. We denote by $\EE_{g,2}^{\mathrm{h} \ast}$ the cases where $t=2$ and $C'$ is hyperelliptic. \\ 
For $0\leq t \leq g/2$, let $\EE_{g,t}\coloneqq \overline{\EE_{g,t}^\ast}\subset \mathcal{A}_g$ denote the closure. We will define in the next section the loci $\EE'_{g,t}$, which corresponds to the case where $C'$ and $C''$ are smooth, but $C$ might be singular, see Definition \ref{Def of Egt prime}. We have inclusions $\EE_{g,t}^\ast \subset \EE'_{g,t} \subset \EE_{g,t}$, and $\EE_{g,2}^{\mathrm{h}\ast} \subset \EE_{g,2}$.
\end{definition}
By \cite[Prop. 5.4.2]{Debarre1988}, if $g\geq 5$, the loci $\EE_{g,t}$ for $0\leq t \leq g/2$ are the $\lfloor g/2 \rfloor +1$ irreducible components of $\BE_g$. The component $\EE_{g,0}$ is of dimension $2g$, and $\EE_{g,t}$ is of dimension $2g-1$ if $t\geq 1$. Moreover, $\EE_{g,0}$ and $\EE_{g,1}$ (and $\EE_{5,2}=\mathcal{A}^2_{2,3}$ if $g=5$) are irreducible components of the Andreotti-Mayer locus $\mathcal{N}_{g-4}$. In dimension $4$ we have $\BE_4=\EE_{4,1}$, and $\overline{\theta^2_{\mathrm{null},4}}=\EE_{4,0}$, where $\theta^k_{\mathrm{null},g}\subset \mathcal{A}_g$ is the locus of ppav's with exactly $k$ vanishing theta nulls.
\begin{remark}
    Note that we have not treated the case where the Galois group of the composition $p\circ \pi$ is $\ZZ/4\ZZ$. It turns out that these Pryms are tetragonally related (and thus isomorphic) to the Pryms in $\EE_{g,0}$ by \cite[Prop. 3.8]{donagi1992fibers}.
\end{remark}
As we will use it time and time again, we recall here the following proposition from \cite[338]{Mumford1974}, \cite[Prop. 5.27]{Debarre1988}:
\begin{proposition}\label{Prop: covering pullback line bundle sequence}
Let $\pi:\tilde{C} \to C$ be a double covering of smooth curves associated to $(\delta$, $\Delta)$ and $L$ a line bundle on $\tilde{C}$. We say divisor $D$ on $\tilde{C}$ is \emph{$\pi$-simple} if we cannot write it as
\[ D=F+\pi^\ast E \,, \quad \text{with $E,F$ effective and $E\neq0$.} \]
Let $D$ be an effective $\pi$-simple divisor, and $M=\pi^\ast L \otimes \BO_{\tilde{C}}(D)$. Then there is a short exact sequence of $\BO_{C}$-modules:
\[ 0 \to L \to \pi_\ast M \to L \otimes \BO_{C}(\pi_\ast D - \delta ) \to 0 \,. \]
\end{proposition}

\subsubsection*{A relation between Pryms}
We denote the norm maps associated to $p'$ and $p''$ by $\Nm'$ and $\Nm''$ respectively. If $t>0$, let 
\begin{align*}
    P'&\coloneqq  \{L \in\Pic^t(C') \,|\, \Nm'L=\delta' \} \,, & \Xi' &\coloneqq \Theta'\cap P' \,,  \\ 
    P''& \coloneqq \{L\in \Pic^{g-t}(C'')\,|\, \Nm'' L=\delta'' \}\,, & \Xi''&\coloneqq \Theta'' \cap P''\,, 
\end{align*}
denote the Prym varieties associated to $p'$ and $p''$ respectively. If $t=0$ we let $P'=0$ and
\begin{align*}
    P''& \coloneqq \{L\in \Pic^{g-t}(C'')\,|\, \Nm'' L=\delta\}\,, & \Xi''&\coloneqq \Theta'' \cap P''\,.
\end{align*}
The polarization of $P''$ (resp. $P'$ if $t\neq 0$) induced by $\Xi''$ (resp. $\Xi'$) is of type $(1,\dots,1,2)$. By \cite[Prop. 5.5.1]{Debarre1988}, the restriction of the pullback map defines an isogeny of polarized abelian varieties
\begin{equation}\label{Prop: pullback of L_P to P'xP''}
 (\pi'^\ast+\pi''^\ast)\restr{P'\times P''}:P'\times P'' \to P
\end{equation}
If $t\geq 1$, the pullback $\pi'^\ast: P'\to P$ and $\pi''^\ast:P''\to P$ are injective, thus $P',P''$ are complementary abelian varieties in the sense of \cite[125]{Birkenhake2004}.

\subsection{The evaluation bundle}\label{Sec: The evaluation bundle}
Before proceeding with the computation of the Gauss, we recall the definition of the evaluation bundle following \cite[339]{arbarello}. Let $C$ be a smooth curve of genus $g$, and $L$ be a line bundle of degree $n$ on $C$ and $0\leq d \leq n$. The \emph{evaluation bundle} associated to $L$ on $C_d$ is defined by
\[ E_L \coloneqq p_{2,\ast}(\BO_\Delta \otimes p_1^\ast L)\,, \]
where $\Delta\subset C\times C_d$ is the universal divisor and $p_1,p_2$ are the projections from $C\times C_d$ to the first and second factor, respectively. $E_L$ is a vector bundle of rank $d$ on $C_d$. At a reduced divisor $D=P_1+\cdots+P_d\in C_d$ the fiber is identified with $E_{L,D}=\oplus_{i=1}^d L_{P_i}$. Moreover, there is a map of vector bundles called the \emph{evaluation map}
\[ \mathrm{ev_L}:\HH^0(C,L)\otimes_\CC \BO_{C_d} \to E_L\,. \]
For $s\in \HH^0(C,L)$ and $D\in C_d$ we have 
\[ s\in \Ker_D \mathrm{ev}_L \iff D\leq \divv s \,. \]
We denote by $\alpha:C_d\to JC$ the Abel-Jacobi map. In what follows, denote by $\HH_\ast(-)=\HH_{2\ast}(-,\QQ)$ (resp. $\HH^\ast(-)$) the even-dimensional homology (resp. cohomology) with $\QQ$-coefficients. Let $x=[C_{d-1}]\in \HH^1(C_d)$ and $\theta\in \HH^1(JC)$ denote the polarization. By abuse of notation, we denote by $\theta$ also the pullback to $C_d$ by $\alpha$. We have the following (see \cite[342]{arbarello}):
\begin{lemma}\label{Lem: Chern classes of the evaluation bundle}
Let $0\leq d\leq n$, $C$ be a smooth curve of genus $g$, $L$ a line bundle of degree $n$ on $C$ and $E_L$ the corresponding evaluation bundle on $C_d$. The $r$-th Chern class of the evaluation bundle is given by
\[ c_r(E_L)=\sum_{k=0}^r \binom{n-g-d+r}{k} \frac{x^k \theta^{r-k}}{(r-k)!}\in \HH^r(C_d) \,. \]   
\end{lemma}
As a corollary we have the following
\begin{corollary}\label{Cor: pushforward chern class of evaluation bundle}
Let $0\leq m \leq d \leq n$, $C$ be a smooth curve of genus $g$, $L$ a line bundle of degree $n$ on $C$, and $E_L$ the corresponding evaluation bundle on $C_d$. We have
\[ \alpha_\ast(x^m c(E_L))=\sum_{k=0}^g \binom{2k+n-2g-m}{k+d-g-m} \frac{\theta^k}{k!}\in \HH^\ast(JC)\]
\end{corollary}
\begin{proof}
    With Poincaré's formula \cite[25]{arbarello} and the Chu-Vandermonde binomial identity we have
    \begin{align*}
        \alpha_\ast(x^m c_r(E_L)) &\coloneqq \sum_k \binom{n-g-d+r}{k} \frac{x^{k+m} \theta^{r-k}}{(r-k)!} \\
        &= \frac{\theta^{g+r-d+m}}{(g+r-d+m)!}\sum_k \binom{n-g-d+r}{k} \binom{g+r-d+m}{r-k} \\
        &= \frac{\theta^{g+r-d+m}}{(g+r-d+m)!} \binom{n-2d+2r+m}{r} \,.
    \end{align*}
    The corollary follows after a change of variables.
\end{proof}

\subsection{The Gauss degree on \texorpdfstring{$\EE^\ast_{g,t}$}{Egt}}\label{Sec: Gauss degree Egt}
In this section, we compute the degree of the Gauss map on bielliptic Pryms arising from coverings of smooth curves:
\begin{theorem}\label{Thm: degree Gauss Map on Egt, in general}
Let $g\geq 4$, $0\leq t \leq g/2$, and $(P,\Xi)=\mathrm{Prym}(\tilde{C}/C)\in \EE^\ast_{g,t}$. The degree of the Gauss map $\GG:\Xi \dashrightarrow \PP^{g-1}$ is given by
\[ \deg \GG= \begin{cases}
    b_{g-1}-2\ns(\tilde{C}/C)\,, &\text{it $t=0$,}\\
    b_{t-1}b_{g-t}+b_tb_{g-t-1}-2^{g-1}-2\ns(\tilde{C}/C)\,, &\text{if $t>0$.}
\end{cases}\]
If $g\geq 5$ (or $g=4$ and $C''$ not hyperelliptic) and $(P,\Xi)\in \EE_{g,2}^{\hh \ast }$, the degree of the Gauss map is given by
\[ \deg \GG =2b_{g-2}+4b_{g-3}-2^{g-2}-2\tns(\tilde{C}/C) \,. \]
If $(P,\Xi)\in \EE_{4,2}^{\hh \ast}$ and $C''$ is hyperelliptic, then
\[\deg \GG=14-2\tns(\tilde{C}/C)\,. \]
\end{theorem}
Throughout the proof we will have to state many properties for both $C'$ and $C''$. Thus, to avoid repeating ourselves we make the following convention: Unless specified otherwise, when we index something with $i$ we mean for $i\in \{',''\}$. For example, with ``let $Q^i\in C^i$ be a point" we will mean ``let $Q'\in C'$ and $Q''\in C''$ be points". \par 
\par  \vspace{.5cm} \textbf{Step one: Setting the scene.} We keep the notations of the previous section. Define $\tilde{\Xi}\subset \Theta'\times \Theta''$ and $W\subset C'_t\times C''_{g-t}$ by the following two cartesian diagrams:

\begin{center}
\begin{tikzcd}
W \arrow[d,hook] \arrow[r,two heads,"\alpha\restr{W}"] & \tilde{\Xi} \arrow[d,hook] \arrow[r,two heads,"\Nm"] & \{\delta\} \arrow[d,hook] \\
C'_t\times C''_{g-t} \arrow[r,two heads, "\alpha"']& \Theta'\times \Theta'' \arrow[r,two heads,"\Nm"'] & \Pic^g(E) \,,
\end{tikzcd}
\end{center}
where $\alpha=\alpha'\times \alpha''$ is the product of the Abel-Jacobi maps, and we make the abuse of notation to write $\Nm$ for the sum $(\Nm'\oplus \Nm''):JC'\times JC''\to JE$. If $t=0$, what we mean with this is that
\[ \tilde{\Xi}=\Nm''^{-1}(\delta)\subset \Theta''\,, \quad \text{and} \quad W=\alpha''^{-1}(\tilde{\Xi})\subset C''_{g} \,. \]
Clearly $\tilde{\Xi}$ is not contained in $\Theta'_\sing\times \Theta''_\sing$ thus $\alpha\restr{W}$ is birational. Let
\[ g \coloneqq \pi'^\ast+\pi''^\ast : \Pic(C')\times \Pic(C'')\to \Pic(\tilde{C}) \,. \]
By \cite[Prop. 5.2.1]{Debarre1988}, we have a surjection
\[g\restr{\tilde{\Xi}}: \tilde{\Xi}\to \Xi \,. \]
Define $\tilde{\mathcal{G}}\coloneqq \mathcal{G}\circ g\restr{\tilde{\Xi}} \circ \alpha\restr{W}:W\to \PP^{g-1}$. We will compute the degree of $\mathcal{G}$ by computing the degree of $\tilde{\mathcal{G}}$. By the following proposition we have $\deg \tilde{\GG}=2\deg \GG$:
\begin{proposition}\label{Prop: g is of degree 2}
The pullback map induces a generically finite morphism of degree $2$
\[
    g\restr{\tilde{\Xi}}: \tilde{\Xi} \to \Xi\,,
    \]
corresponding on the finite locus to the quotient by the involution $\iota\restr{\tilde{\Xi}}$ where
\begin{align*} 
\iota: \Pic^{t}(C')\times \Pic^{g-t}(C'') &\to \Pic^{t}(C')\times \Pic^{g-t}(C'') \\
(L',L'')&\mapsto (\omega_{C'}-\tau'L',\omega_{C''}-\tau''L'') \,. 
\end{align*}
If $t=0$, the morphism $g\restr{\Xi''}:\Xi''\to \Xi$ is étale of degree $2$.
\end{proposition}
\begin{proof}
Let $L=\pi'^\ast(\BO_{C'}(D')\otimes \pi''^\ast(\BO_{C''}(D'')\in \tilde{\Xi}$. We have
\begin{align*}
    \Nm(L)&= \BO_C(\pi_\ast(\pi'^\ast D'+\pi''^\ast D'')) \\
    &= \BO_C( p^\ast (p'_\ast D'+p''_\ast D'')) \\
    &= p^\ast \delta \\
    &= \omega_C\,.
\end{align*}
Thus $\sigma L = \omega_{\tilde{C}}-L$, and
\[ \pi'^\ast(\omega_{C'}-\tau'D')+\pi''^\ast(\omega_{C''}-\tau''D'')= \omega_{\tilde{C}}-\sigma(\pi'^\ast D'+\pi''^\ast D'')\sim \pi'^\ast D'+\pi''^\ast D'' \,. \]
Thus $g\restr{\tilde{\Xi}}$ factors through the quotient by the involution $\iota$. For a general $L\in \tilde{\Xi}$, we can assume $D'$ (resp. $D''$) to be $p'$ (resp. $p''$)-simple. Thus by \ref{Prop: covering pullback line bundle sequence} we have a left exact sequence
\[ 0\to \HH^0(C'',D'') \to \HH^0(\tilde{C},L) \to \HH^0(C'',\BO_{C''}(D''+\pi''_\ast \pi'^\ast D'-\tilde{\delta}''))\,,  \]
where $\tilde{\delta}''$ is the line bundle associated to the cover $\tilde{C}\to C''$. We have $p''^\ast \delta=\BO_{C''}(D''+\tau'' D''+ p''^\ast p'_\ast D')$ and $p''^\ast \delta=\omega_{C''}\otimes \tilde{\delta''}$, thus the left exact sequence above becomes
\[ 0 \to \HH^0(C'',D'') \overset{\alpha}{\longrightarrow} \HH^0(\tilde{C}, L) \overset{\beta}{\longrightarrow} \HH^0(C'',\omega_{C''}(-\tau'' D'')) \,. \]
The two maps $\alpha$ and $\beta$ are defined as follows:
\begin{align*}
 \forall v\in \HH^0(C'',D'')\,, \quad \alpha(v)=v\cdot p'^\ast u_{D'} \,, \quad \text{where } \divv u_{D'}=D' \,, \\
\forall u\in \HH^0(\tilde{C}, L)\,, \quad u \cdot \sigma''(p'^\ast u_{D'})-\sigma''(u) \cdot p'^\ast(u_{D'})= s \cdot \pi''^\ast(\beta(u)) \,, \end{align*}
where $\divv s = \tilde{R''}$ is the ramification of $\pi''$. Suppose $u=\pi'^\ast (u')\cdot \pi''^\ast(u'')\in \HH^0(\tilde{C}, L)$ for some $u'\in \HH^0(C',D'_1)$ and $u''\in \HH^0(C'',D''_1)$. Then
\[ s\cdot \pi''^\ast(\beta(u))=\pi''^\ast(u'')\cdot \pi'^\ast(u'\cdot \tau'(u_{D'})-\tau'(u')\cdot u_{D'} ) \,. \]
Thus either $\beta(u)=0$ and then $u''\in \HH^0(C'',D'')$ and $u'=u_{D'}$ or $u''\in \HH^0(C'',\omega_{C''}(-\tau'' D''))$. The same reasoning can be applied to $u'$, thus the preimage of $L$ by $g\restr{\tilde{\Xi}}$ is $\{(\BO_{C'}(D'),\BO_{C''}(D'')),(\omega_{C'}(-\tau' D'),\omega_{C''}(-\tau''(D''))\}$.\\ Finally, we have shown that a general $L\in g(\tilde{\Xi})$ has a $2$-dimensional space of sections, we are thus in $P$ and not in $P^{-}$. If $t=0$, then $g\restr{\tilde{\Xi}}$ is the restriction of the degree $2$ isogeny $P''\to P$, it is thus étale.
\end{proof}
We have the following:
\begin{proposition}\label{Prop: Fibers of g restr Xi tilde}
The locus where $g\restr{\tilde{\Xi}}:\tilde{\Xi}\to \Xi$ fails to be finite is 
\[ \left((p'^\ast \Pic^1(E)+W^0_{t-2}(C'))\times (p''^\ast \Pic^1(E)+W^0_{g-t-2}(C'')\right)\times_{\Pic^g(E)}\{\delta\} \,. \]
On this locus the fibers of $g$ are the orbits under action of the antidiagonal embedding $(p'^\ast,-p''^\ast):JE\to JC'\times JC''$.
\end{proposition}
\begin{proof}
In the above proof, the only assumption made to prove the finiteness of $g$ at $(\BO_{C"}(D'),\BO_{C''}(D''))\in \tilde{\Xi}$ is that $D'$ (resp. $D''$) is $p'$-simple (resp. $p''$-simple). The result follows.
\end{proof}
We have the following description of $W$'s singularities:
\begin{proposition}\label{Prop: smoothness and singularities of W}
We have
\[ \Sing(W)=\{(D',D'')\in W\,|\, D'\leq R'\,,D''\leq R''\} \,. \]
At a point $w\in\Sing(W)$, $W$ has a quadratic hypersurface singularity of maximal rank, i.e. local analytically we have
\[ (W,w)\simeq V(z_1^2+\cdots+z_g^2)\subset (\CC^g,0) \,. \]
\end{proposition}
\begin{proof}
Let $(D',D'')\in W$. Suppose that $D^i\leq R^i$. The usual coordinate $z$ on $\CC$ descends locally to a coordinate on $E=\CC/\Lambda_E$. $p^i:C^i\to E$ is ramified at the points $P_k^i\in D^i\leq R^i$ thus there are local coordinates $z_k$ on $C^i$ at $P_k^i$ such that in the coordinate $z$
\[ p^i(z_k)=p(P^i_k)+z_k^2 \,. \]
Since the points $P^i_k$ are all distinct, $(z_k)_{1\leq k\leq g}$ defines a local coordinate system for $C'_t\times C''_{g-t}$ at $(D',D'')$ and locally $W$ is defined by
\begin{align*}
    0&= \Nm(z_1,\dots,z_g)-\delta \\
    &= z_1^2+\cdots+z_g^2+\Nm(D',D'')-\delta \\
    &= z_1^2+\cdots+z_g^2 \,.
\end{align*}
Now suppose $D^i\nleq R^i$ for some $i\in\{',''\}$. For simplicity, assume $D''=P_0+F$ with $P_0$ not a ramification point of $p''$. Consider the embedding
\begin{align*}
    \iota : C'' &\hookrightarrow C'_t\times C''_{g-t} \\
    P &\mapsto(D',P+F) \,.
\end{align*}
Then the composite $\Nm\circ \iota: P\mapsto p''(P)+\Nm'(D')+\Nm''(F)$ has nonzero differential at $P_0$. In particular, $\Nm$ must have nonzero differential at $(D',D'')$ and thus $W$ is smooth at this point.
\end{proof}
We make the following definition:
\begin{definition}\label{Def: Xiasing in the smooth case}
    For $0\leq t \leq g/2$ and $(P,\Xi)\in \EE_{g,t}^\ast$, Let
    \[ \Wsingis \coloneqq \{ (D',D'')\in W_\sing \,|\, \Nm'(D')\neq \delta \} \,. \] 
    We define the set of \emph{additional singularities} and \emph{isolated additional singularities} of $\Xi$ respectively, as $0$-cycles, by
    \begin{align*}
        \Xiasing &\coloneqq  g\circ \alpha [W_\sing]/2  \,, \\
         \Xiasingis &\coloneqq g\circ \alpha [\Wsingis]/2 \,.
    \end{align*}
     We define the \emph{number of additional singularities} as $\ns(\tilde{C}/C)\coloneqq \deg(\Xiasing)$ and the number of \emph{isolated} additional singularities by $\tns(\tilde{C}/C)\coloneqq \deg(\Xiasingis)$. This is an invariant of the double cover $\tilde{C}\to C$ when $g\geq 4$ \footnote{If $g\geq 5$, the bielliptic involution on $C$ is unique \cite[Ex. VIII C-2]{arbarello}. If $g=4$, the bielliptic involution is not unique anymore, but if $C$ has two bielliptic involutions inducing two towers as in (\ref{Diagramm: E_g,t tower}), then one can investigate the group acting on $\tilde{C}$ and deduce that genus of the other intermediate curves in the two towers ($C'_1$, $C''_1$, $C'_2$, $C''_2$) are all equal to $3$, and none is hyperelliptic. Thus, by Theorem \ref{Thm: degree Gauss Map on Egt, in general} $\ns(\tilde{C}/C)$ is an invariant of the double cover $\tilde{C}\to C$.}
\end{definition}
By definition, we thus have
    \begin{equation}\label{Def: eta} 
    \begin{aligned}
        \ns(\tilde{C}/C)&=\frac{1}{2} \# \{ (D',D'')\in E_t\times E_{g-t}\,|\, D'\leq \Delta'\,, D''\leq \Delta''\,, D'+D'' \sim  \delta \}\,, \\
        \tns(\tilde{C}/C)&=\ns(\tilde{C}/C)-\frac{1}{2}\#\{ (D',D'')\in E_t\times E_{g-t}\,|\, \forall i\in \{',''\}\,, D^i\leq \Delta^i\,, D^i \sim \delta^i \}\,.
    \end{aligned}
    \end{equation}
The terminology is justified by the following:
\begin{proposition}\label{Prop: singularities are quadratic of maximal rank}
At a point in $\Xiasingis$, $\Xi$ has an isolated quadratic singularity of maximal rank. These are two-torsion points of $P$.
\end{proposition}
\begin{remark}
    The points in $\Xiasing\setminus \Xiasingis$ are singular points of $\Xi$ as well, but they are not isolated anymore when $g\geq 7$, as they are in the set $V$ of \cite[Prop. 5.2.2]{Debarre1988}. They are limits of isolated singularities in a suitable deformation of $(P,\Xi)$ in $\EE_{g,t}$.
\end{remark}
\begin{proof}
Let $L=\BO_{\tilde{C}}(\pi'^\ast D'+\pi''^\ast D'')\in \Xiasingis$. By \ref{Prop: g is of degree 2}, $g$ is étale at $\alpha(D',D'')$. Because $D^i\not \sim \delta^i$ we have $\hh^0(C',D')=\hh^0(C'',D'')=1$ thus $\alpha:C'_t\times C''_{g-t}\to \Theta'\times \Theta''$ is locally an isomorphism at $(D',D'')$. In particular, 
\[ (g\circ \alpha) \restr{W}:W\to \Xi \]
is a local isomorphism at $(D',D'')$ and thus by \ref{Prop: smoothness and singularities of W}, $\Xi$ has an isolated quadratic singularity of maximal rank at $L$. We have 
\[ L^{\otimes 2}=f^\ast \delta = K_{\tilde{C}} \,, \]
thus $L$ corresponds to a $2$-torsion point.
\end{proof}

\par  \vspace{.5cm} \textbf{Step two: A configuration of Gauss Maps.} For a general $(D',D'')\in W$ we have
\begin{align*}
\tilde{\GG}(D',D'')&=T_{g\circ \alpha (D',D'')} \Xi   \\
&= \dd g \left( T_{\alpha(D',D'')} \tilde{\Xi} \right) \\
&= \dd g \left((T_{\alpha'(D')}\Theta' \oplus T_{\alpha''(D'')} \Theta'' )\cap \Ker(\dd \Nm) \right)\,.
\end{align*}
As $g$ and $\Nm$ are morphism of abelian varieties, their differential is constant. Thus we define the following rational map
\begin{align*}
    \F: \PP(T_0 JC')^\ast \times \PP(T_0JC'')^\ast &\dashrightarrow \PP(T_0P)^\ast \\
    (H',H'')&\mapsto dg\left( (H'\oplus H'')\cap \Ker(d\Nm) \right)\,.
\end{align*}
Thus we have $\tilde{\GG}=\F \circ (\tilde{\GG'}\times \tilde{\GG''})$, where $\tilde{\GG^i}=\GG^i\circ \alpha^i:C^i_t\dashrightarrow \PP(T_0 JC^i)^\ast$ are the Gauss maps corresponding to $C^i$. We now give a more precise description of $\F$. We have the identification
\[ TJC'= TP'\oplus TJE \,, \quad \text{and} \quad TJC''=TJE\oplus TP'' \,. \]
Under this identification we have
\[ \Ker(\dd \Nm)= TP'\oplus \Delta_{TJE} \oplus TP''\subset TJC'\oplus TJC'' \,, \]
where $\Delta_{TJE}\subset TJE\oplus TJE$ is the antidiagonal embedding. By \ref{Prop: pullback of L_P to P'xP''} there is a canonical identification $TP=TP'\oplus TP''$ and
\[ \dd(g\restr{\Ker \Nm}): TP'\oplus \Delta_{TJE} \oplus TP'' \to TP'\oplus TP''  \]
is the natural projection. Moreover, for $i\in\{',''\}$ there are canonical identifications
\begin{equation}\label{Egt: Equ: Identifications tangent P and tangent JC}
\begin{aligned} 
   T_0^\vee JC^i &=\HH^0(C^i,\omega_{C^i})\,,\\
    T_0^\vee P^i &= \HH^0(E,\delta^i)\,, \\
    T_0^\vee P&=T_0^\vee P'\oplus T_0^\vee P''=\HH^0(E,\delta')\oplus \HH^0(E,\delta'')\,.
\end{aligned}
\end{equation}
We thus define $\PPdelta\coloneqq \PP(\HH^0(E,\delta')\oplus\HH^0(E,\delta''))=\PP(T_0P)^\ast$. Let $\dd z$ be the constant differential form on $E$. Then $R^i=\divv(p^{i,\ast} \dd z)\in |\omega_{C^i}|$ is the ramification divisor of $p^i$ and the corresponding section $s_{R^i}\coloneqq p^{i,\ast }\dd z$ generates $\HH^0(C^i,\omega_{C^i})^+$ the subspace of invariant sections under the $\tau^i$ action. We then have the identification
\begin{align*}
    \HH^0(E,\omega_E)\oplus\HH^0(E,\delta^i) &\overset{\sim}{\longrightarrow}  \HH^0(C^i,\omega_{C^i}) \\
    (\lambda \dd z , s)&\mapsto \lambda s_{R^i}+p^{i,\ast} s \,. 
\end{align*}
We have the following:
\begin{proposition}\label{Prop: precise description of FF}
Under these identifications, if $t=0$, the map $\F: |\omega_{C''}|\dashrightarrow |\delta''|$ is the projection from $R''$, where we identify $|\delta''|$ with $p''^\ast|\delta''|\subset |\omega_{C''}|$. If $t\geq 1$ the map $\F$ takes the form
\begin{align*}\label{Equ: Description of FF}
    \F: |\omega_{C'}|\times |\omega_{C''}| &\dashrightarrow \PPdelta \\
    \divv(\lambda s_{R'}+p'^\ast s'),\divv(\mu s_{R''}+p''^\ast s'') &\mapsto [\mu s'\oplus \lambda s''] \,,
\end{align*}
where $s^i\in \HH^0(E,\delta^i)$. 
\end{proposition}
\begin{proof}
Let $a_0=b_0$ be a coordinate for $T_0^\vee JE$ and $a_1,\dots,a_t$ (respectively $b_1,\dots,b_{g-t}$) are coordinates for $T_0^\vee P'$ (respectively $T_0^\vee P''$). This leads to homogeneous coordinates $(a_0:\cdots:a_t),(b_0:\cdots:b_{g-t})$ on $\PP(T_0 JC')^\ast \times \PP(T_0^\vee JC'')^\ast$. Then it is clear from the above description that if $(H',H'')\in \PP(T_0 JC')^\ast \oplus \PP(T_0^\vee JC'')^\ast$ are two hyperplanes, the hyperplane (if well-defined)
\[ H= \left((H'\oplus H'')\cap (TP'\oplus \Delta_{TJE}\oplus TP'')\right) / \Delta_{TJE} \subset TP'\oplus TP'' \]
is given by the coordinates $(b_0 a_1:b_0a_2:\cdots:b_0 a_t:a_0b_1:\cdots:a_0 b_{g-t})$. Via the identifications this leads to the description of the proposition.
\end{proof}

To sum things up, the following diagram is commutative:
\stepcounter{equation}
\begin{figure}[H]
\centering 
\begin{tikzcd}
W \arrow[r] \arrow[rrr,dashed,bend left=15,"\tilde{\GG}"] \arrow[d,phantom,sloped,"\subset"] & \tilde{\Xi} \arrow[r, "g\restr{\tilde{\Xi}}"'] \arrow[d,phantom,sloped,"\subset"] & \Xi \arrow[r,dashed, "\GG"'] & \PPdelta \\
C'_t\times C''_{g-t} \arrow[rrr,bend right=15,dashed,"\tilde{\GG'}\times \tilde{\GG''}"']\arrow[r,"\alpha'\times\alpha''"] & \Theta'\times \Theta''\arrow[rr,dashed,"\GG'\times \GG''"] & & {|\omega_{C'}|\times |\omega_{C''}|} \arrow[u,dashed,"\F"']
\end{tikzcd}
\caption{}\label{Diagramm: definition of GG in Egt smooth case}
\end{figure} 
\par 

Let $M=[s'+s'']\in \PPdelta$, assume $s^i\neq 0$ and let $M^i=\divv s^i$. We define
\begin{align*}
 V^i_{M^i}&\coloneqq \langle R^i,p^{i,\ast} M^i \rangle \subset |\omega_{C^i}| \,,\\
  V_M&\coloneqq \overline{\F^{-1}(M)}\,,\\
 Z^i_{M^i}&\coloneqq \{D^i\in C^i_{t^i} \,|\, D^i\leq F^i\in V^i_{M^i} \}\,,\\
 Z_M&\coloneqq \{(D',D'')\in C'_t\times C''_{g-t}\,,\, \exists (F',F'')\in V_M\,, D^i\leq F^i\}\,.
\end{align*}
There are maps
\begin{align*}
    \varphi_{s^i}:V^i_{M^i} &\overset{\sim}{\longrightarrow} \PP^1 \\
    \divv( u s_{R^i}+ v p^{i,\ast }s^i)&\mapsto (u:v) \,.
\end{align*}
Under these identifications we have 
\begin{equation}\label{Equ: V_M is product over P^1}
    V_M= V'_{M'} \times_{\PP^1} V''_{M''}\,.
\end{equation}
It is well-known \cite[page 246]{arbarello}, that the map $\tilde{\GG}'$ has the following description: for $D'\in C'_t$ such that $\hh^0(D')=1$ (i.e. $D'\notin C'^1_t$), there is a unique divisor $F'\in|\omega_{C'}|$ such that $D'\leq F'$. Then
\[ \tilde{\GG}'(D')=F'\in|\omega_{C'}|\,.\]
The same applies to $\tilde{\GG}''$. Thus $\tilde{\GG}^{i,-1}(V^i_{M^i})\subset Z^i_{M^i}$ (resp. $(\tilde{\GG}'\times\tilde{\GG}'')^{-1}(V_M)\subseteq Z_M$), but $Z^i_{M^i}$ (resp. $Z_M$) might contain components that are not in the domain of $\tilde{\GG}^i$ (resp. $\tilde{\GG}'\times\tilde{\GG}''$). We will later see that the inclusions are equalities for general $M$, except when $t=2$ and $C'$ is hyperelliptic.
\par 
\vspace{.5cm}
\textbf{Step three: A computation in cohomology.} For now, we set a general $M=[s'+s'']\in\PPdelta$ and drop the subscript on $Z$. We compute the intersection as cohomological cycles of $Z$ and $W$. We see from the definition that $Z$ is a secant variety embedded diagonally in the product $C'_t\times C''_{g-t}$ and we compute its class using a modified version of the Lemma 3.2 in \cite[page 342]{arbarello}. Let
\[ V'=\langle s_{R'},p'^\ast s' \rangle \subset \HH^0(C',\omega_{C'})\,. \]
Let $E'$ be the evaluation bundle on $C'_t$ corresponding to $\omega_{C'}$ (see \ref{Sec: The evaluation bundle}) and denote the evaluation map by
\begin{equation}\label{Equ: evalution map defining Z M}
    e':V'\otimes_\CC \BO_{C'_t} \to E' \,.
\end{equation} 
Symmetrically, let $V''=\langle s_{R''},p''^\ast s'' \rangle$, $E''$ be the evaluation bundle on $C''_{g-t}$ corresponding to $L_{C''}$ and an evaluation $e'':V''\otimes_\CC \BO_{C''_{g-t}} \to E''$ the evaluation map. The vector spaces $V'$ and $V''$ are canonically identified (by the linear map sending $s_{R'}$ to $s_{R''}$ and $p'^\ast s'$ to $p''^\ast s''$ for $s^i\in \HH^0(E,\delta^i)$). Consider the composite map of vector bundles on $C'_t\times C''_{g-t}$
\begin{center}
    \begin{tikzcd}
    V' \arrow[rrr,bend right=15,"\Psi"'] \arrow[r,hook,"\Delta_{V'}"] & V' \oplus V' \arrow[r,"\sim"] & V'\oplus V'' \arrow[r,"e'\oplus e''"]  & \mathrm{pr}_1^\ast E'\oplus \mathrm{pr}_2^\ast E'' \,,
    \end{tikzcd}
\end{center}
where we identify vector spaces with the corresponding constant vector bundles, $\Delta_{V'}$ is the diagonal embedding and $\mathrm{pr}_i$ are the projections from $C'_t\times C''_{g-t}$. Then $Z$ is the locus where $\Psi$ has a non-zero kernel. Thus by \cite[Ex. 14.4.2]{Fulton1998} we have
\begin{align*}
    [Z]&= c_{g-1}(\mathrm{pr}_1^\ast E'\oplus \mathrm{pr}_2^\ast E'') \,.
\end{align*}
By \ref{Cor: pushforward chern class of evaluation bundle}, for $l\geq 0$, we have
\begin{align*}
    \alpha'_\ast(c_{t-l}(E'))&= \alpha'_\ast [C'_l]\binom{2t-2l}{t-l}\in \HH_l(JC')\,. \stepcounter{equation}\tag{\theequation}\label{Equ: class of Z M}
\end{align*}
The same formula applies to $c_{g-t-l}(E'')$ after replacing $t$ with $g-t$. Thus we have
\begin{align*}
    \Nm_\ast& \alpha_\ast  ([Z]\cdot[W])=\Nm_\ast\alpha_\ast([Z]\cdot \alpha^\ast \Nm^\ast [\pt]) \\
    &=[\pt]\cdot \Nm_\ast \alpha_\ast(c_{g-1}(\mathrm{pr}_1^\ast E'\oplus \mathrm{pr}_2^\ast E'') \\
    &=[\pt]\cdot \Nm_\ast \alpha_\ast(c_{t-1}(E')\times c_{g-t}(E'')+c_t(E')\times c_{g-t-1}(E'')) \\
    &=[\pt]\cdot \Nm_\ast(\alpha'_\ast(c_{t-1}(E'))\times\alpha''_\ast( c_{g-t}(E''))+\alpha'_\ast(c_t(E'))\times\alpha''_\ast( c_{g-t-1}(E''))) \\
    &=[\pt]\cdot \Nm_\ast\left( [C']\times [\pt]\binom{2t-2}{t-1}\binom{2g-2t)}{g-t}+[\pt]\times [C'']\binom{2t}{t}\binom{2g-2t-2)}{g-t-1} \right) \\
    &=[\pt] 2 \left(\binom{2t-2}{t-1}\binom{2g-2t}{g-t}+\binom{2t}{t}\binom{2g-2t-2}{g-t-1} \right)\in\HH_0(JE)\,.\stepcounter{equation}\tag{\theequation}\label{Equ: intersection number Z_M with W}
\end{align*}
\par  \vspace{.5cm} \textbf{Step four: Ruling out special points.} 
We are not completely done yet, as we need to check which of the points in the intersection $Z\cap W$ lie in the domain of definition of $\tilde{\GG}=\F\circ\tilde{\GG}'\times\tilde{\GG}''$. This is done by the following:
\begin{lemma}\label{Lem: number of points of W cap Z outside of the domain of GG}
For a general $M\in \PPdelta$, the number of points in $Z_M\cap W$ outside the domain of $\tilde{\GG}$, counted with multiplicity, is
\begin{itemize}
    \item $4\ns(\tilde{C}/C)$ if $t=0$,
    \item $2^g+4\ns(\tilde{C}/C)$ if $t\geq 1$,
\end{itemize}
where $\ns(\tilde{C}/C)$ is defined in Definition \ref{Def: Xiasing in the smooth case}.
\end{lemma}
Assume the conditions of the lemma. Let $M\in \PPdelta$ general. By the lemma above the intersection $Z_M\cap W$ is finite and the number of points is given by \ref{Equ: intersection number Z_M with W}. We thus have if $t\geq1$
\begin{align*}
    \deg \GG&= \deg \tilde{\GG}/2 \\
    &= 1/2([Z_M]\cap W -2^g-4\ns(\tilde{C}/C)) \\
    &=\left(\binom{2t-2}{t-1}\binom{2g-2t}{g-t}+\binom{2t}{t}\binom{2g-2t-2}{g-t-1} \right)-2^{g-1}-2\ns(\tilde{C}/C) \,.
\end{align*}
Similarly if $t=0$ we have
\[\deg \GG=\left(\binom{2t-2}{t-1}\binom{2g-2t}{g-t}+\binom{2t}{t}\binom{2g-2t-2}{g-t-1} \right)-2\ns(\tilde{C}/C) \,.\]
\begin{proof}[Proof of Lemma \ref{Lem: number of points of W cap Z outside of the domain of GG}]
The curves $C'$ and $C''$ are not hyperelliptic if $t\geq3$. Thus by assumption $C^i$ is either of genus $2$ or not hyperelliptic. In either case, $C'^1_t$ (resp. $C''^1_{g-t}$) is of codimension $2$ (or empty) inside $C'_t$ (resp. $C''_{g-t}$). We first check the domain of definition of $\F$. Let $M=[s'+s'']\in \PPdelta$ be general. The points in $V_M=\overline{\F^{-1}(M)}$ that are not in the domain of definition of $\F$ are
\[(R',R'')\,, \quad \text{and additionally if } t\geq 1\,,\quad (p'^\ast M',p''^\ast M'')\,,\]
where $M^i=\divv(s^i)$.
\par
\emph{1) Points lying above $(R',R'')$.} By the definition of $Z_M$, the set of points of $Z_M\cap W$ lying above $(R',R'')$ is in bijection with 
\[ \{ (D',D'')\in C'_t\times C''_{g-t}\,|\, D^i\leq R^i \text{ and } \Nm(D',D'')=\delta \} \,. \]
This is in bijection with the set of the lemma after pushing forward to $E$. Fix one such point $(D',D'')\in W\cap Z_M$ with $D^i\leq R^i$. We want to compute the intersection multiplicity of $W\cap Z_M$ at $(D',D'')$. Let 
\begin{equation}\label{Equ: definition of Gamma} \Gamma^i=\{(D,H)\in C^i_{g(C^i)-1}\times |\omega_{C^i}|\,, D\leq H \} \,. 
\end{equation}
Since $C^i$ is not hyperelliptic (or of genus $2$), $\Gamma^i$ is the closure of the graph of $\tilde{\GG}^i$ \cite[page 247]{arbarello2} and we have the following commutative diagram
\begin{center}
    \begin{tikzcd}
       \Gamma'\times \Gamma'' \arrow[d,"q'\times q''"'] \arrow[dr, "\hat{\GG}"] & \\
    C'_{t}\times C''_{g-t} \arrow[r,dashed,"\tilde{\GG}'\times \tilde{\GG}'' "'] & {|\omega_{C'}|\times |\omega_{C''}|} \,.
    \end{tikzcd}
\end{center}
Let $\hat{Z}_M=\hat{\GG}^{-1}(V_M)$. Let $x=((D',R'),(D'',R''))\in \Gamma'\times \Gamma''$ and $y=(D',D'')\in C'_t\times C''_{g-t}$. By \ref{Lem: rank of projection from Gamma to C_g} $\hat{\GG}$ is étale above $(R',R'')$, in particular, $\hat{\GG}\restr{\hat{Z}_M}$ induces a local isomorphism onto $V_M$. Thus all the arrows in the following diagram are local isomorphism:
\begin{center}
\begin{tikzcd}
(\hat{Z}_M,x)  \arrow[dr,"\hat{\GG}\restr{\hat{Z}_M}",] \arrow[d,"q\restr{\hat{Z}_M}"'] \\
(Z_M,y) \arrow[r,"\tilde{\GG}'\times\tilde{\GG}''"'] & (V_M,(R',R''))\,.
\end{tikzcd}
\end{center}
Thus $Z_M$ is smooth at $y$. Moreover, varying $M$ in $\PPdelta$, the set of $V_M$'s is an open set of all lines in $|\omega_{C'}|\times |\omega_{C''}|$ going through $(R',R'')$, thus the  set of tangent directions of $\hat{Z}_M$ is an open set in $T_x(\Gamma'\times \Gamma'')$. We need the following Lemma, whose proof we do below:
\begin{lemma}\label{Lem: rank of projection from Gamma to C_g}
Let $C$ be a smooth curve of genus $g$. Let $\Gamma=\{(D,H)\in C_{g-1}\times |K_C|\,|\, D\leq H \}$. Let $q:\Gamma\to C_{g-1}$ be the projection. Then if $(D,H)\in \Gamma$ and $H$ is reduced, $\Gamma$ is smooth at $(D,H)$ and
\[ \mathrm{corank}(\dd_{(D,H)} q)=\hh^0(C,D)-1 \,. \]
\end{lemma}
Applying this to $\Gamma'$ and $\Gamma''$ we have that the set of tangent directions of $Z_M$ is an open set of a plane of $T_y(C'_t\times C''_{g-t})$ of codimension
\begin{equation}\label{Equ: dimension tangent of lines of Z_M intersecting with W} \hh^0(C',D')+\hh^0(C'',D'')-2\leq 2 \,, 
\end{equation}
where the inequality comes from \ref{Prop: covering pullback line bundle sequence}. By \ref{Prop: smoothness and singularities of W}, $W$ has a quadratic singularity of maximal rank at $y$. The intersection of a smooth curve and a quadric cone is of multiplicity $2$ if and only if the tangent direction of the curve is not contained in the corresponding quadric cone. A singular quadric cone of maximal rank and of dimension $g$ cannot contain a linear space of dimension more than $g/2$. If $g\geq 5$ we have $g-2>g/2$ thus for a general $M\in \PPdelta$, the tangent $T_y Z_M$ is not in the tangent cone to $W$ and thus the intersection multiplicity is exactly $2$. If $g=4$, and we are not in the situation where $t=2$ and $C'$ hyperelliptic, then $\hh^0(C',D')+\hh^0(C'',D'')\leq 3$ and thus $Z_M$ meets $W$ with multiplicity $2$ as well.

\emph{2) Points lying above $(p'^\ast \delta',p''^\ast \delta'')$.}
Let $x=(D',D'')\in C'_t\times C''_{g-t}$ with $D^i\leq p^{i,\ast} M^i$. Let $y=(p'^\ast M',p''^\ast M'')\in |\omega_{C'}|\times|\omega_{C''}|$. By generality of $(M',M'')$, we have $\Nm(D',D'')=\delta $ if and only if $D'$ (resp. $D''$) has exactly one point lying above each point of $M'$ (resp. $M''$). Thus there are exactly $2^g$ such points. Let's compute the multiplicities in the intersection $Z_M\cap W$. For dimension reasons and generality of $M$, $D^i$ is not in $C'^1_t$ (resp. $C''^1_{g-t}$). By generality, $D^i$ is reduced, thus $\tilde{\GG}'\times\tilde{\GG}''$ is étale at $(D',D'')$. Thus it is enough to check the multiplicity of the intersection of $Y$ and $V_M$ at $(p'^\ast M',p''^\ast M'')$, where 
\[Y\coloneqq \tilde{\GG}'\times\tilde{\GG}''(U) \]
for $U\subset W$ a small neighborhood around $(D',D'')$. Moreover, since $W$ is smooth at $x$, $Y$ is smooth. We have
\begin{align*}  (p'^\ast|\delta'|\times p''^\ast |\delta''|,y) \subset (Y,y) \,, \quad \text{thus}\\
\quad T_{p'^\ast M '} p'^\ast |\delta'| \oplus T_{p''^\ast M''} p''^\ast |\delta''| \subset T_{p'^\ast M',p''^\ast M''} Y \,. \end{align*}
We have one dimension of freedom for the lines $V_M$ passing through $(p'^\ast M',p''^\ast M'')$: fix some $s^i$ with $M^i=\divv s^i$. We define
\begin{equation*}
    M_\lambda \coloneqq [s'+\lambda s'']\in \PPdelta \,,\quad \lambda\in \CC^\ast \,.
\end{equation*}
Then the lines $V_{M_\lambda}$ all pass through $(p'^\ast M',p''^\ast M'')$, and their tangent directions generate $T_{p'^\ast M'} V'_{M'} \oplus T_{p''^\ast M''} V''_{M''}$. We have
\[ T_{p^{i\ast} M^i} V^i_{M^i} \oplus T_{p^{i\ast} M^i} p^{i\ast} |\delta^i|=T_{p^{i\ast } M^i} |\omega_{C^i}| \,, \]
thus a general line $V_{M_\lambda}$ must be transversal to $Y$ at $(p'^\ast M',p''^\ast M'')$ (else $Y$ would be singular). Thus the points above $(p'^\ast M',p''^\ast M'')$ appear with multiplicity one.
\par 
\emph{3) The domain of $\tilde{\GG}'\times\tilde{\GG}''$.} By \cite[page 247]{arbarello}, $\tilde{\GG}'$ (resp. $\tilde{\GG}''$) is undefined on the Brill-Noether loci $C'^1_t$ (resp. $C''^1_{g-t}$). we have $\Gamma'^1=q'^{-1}(C'^1_t)\subset \Gamma'$ of codimension $1$ thus $\hat{\GG}(q^{-1}(W)\cap \Gamma'^1\times \Gamma'')$ is of codimension $2$ and a general line $V_M$ going through $(R',R'')$ will not meet this set except possibly in $(R',R'')$. The same reasoning applies to $C''^1_{g-t}$. Thus, for a general $M\in \PPdelta$, no points of $Z_M\cap W$ are outside the domain of $\tilde{\GG}'\times \tilde{\GG}''$, except possibly points above $(R',R'')$, but those were treated in 1) already.
\end{proof}
\begin{proof}[Proof of Lemma \ref{Lem: rank of projection from Gamma to C_g}]
Let $\hat{\GG}:\Gamma \to |K_C|$ be the projection. Then $\hat{\GG}$ is finite, and étale above reduced divisors. In particular, if $H$ is reduced, $\Gamma$ is smooth at $(D,H)$. Let $(w_i)_{1\leq i \leq g}$ be a basis of $\HH^0(C,K_C)$ such that $H=\divv w_g$. Locally near $H$ this basis induces coordinates $(a_i)_{1\leq i\leq g-1} $ on $|K_C|=\PP\HH^0(C,K_C)$. Suppose $D=P_1+\cdots +P_{g-1}$ and let $z_j$ be coordinates on $C$ around $P_j$. Locally $\Gamma$ is defined by the vanishing of 
\[ \left(w_g(z_j)+ \sum_{i=1}^{g-1} a_i w_i(z_j)\right)_{1\leq j \leq g-1} \,. \]
Thus

\begin{align*} T_{(D,H)} \Gamma &= \Ker \left( 
\begin{array}{ccc|ccc}
\dv w_g(P-1)& {} & 0  & w_1(P_1) & \cdots & w_{g-1}(P_1) \\
{} & \ddots & & \vdots & & \vdots \\
0 & {} & \dv w_g(P_{g-1}) & w_1(P_{g-1}) & \cdots  & w_{g-1}(P_{g-1})
\end{array} 
\right) \\
& \subseteq T_D C_{g-1}\oplus T_H |K_C| \,.
\end{align*} 
A basis is given by 
\[
e_i= \left(\frac{w_i(P_1)}{\dv w_g(P_1)}, \cdots , \frac{w_i(P_{g-1})}{\dv w_g(P_{g-1})},0,\cdots, -1,\cdots , 0 \right)^t \,.
\]
for $1\leq i \leq g-1$. Thus the image of $\dd q$ at $(D,H)$ is generated by the columns of
\[
\begin{pmatrix}
\frac{w_1(P_1)}{\dv w_g(P_1)} & \cdots & \frac{w_{g-1}(P_1)}{\dv w_g(P_1)} \\
\vdots && \vdots \\
\frac{w_1(P_{g-1})}{\dv w_g(P_{g-1})} & \cdots & \frac{w_{g-1}(P_{g-1})}{\dv w_g(P_{g-1})}
\end{pmatrix}\,.\]
We recognize the Brill-Noether matrix without the last column (which is zero in this case) and with rows multiplied by a non-zero scalar. Thus the rank is $g-\hh^0(C,D)$ (see \cite{arbarello}).
\end{proof}

\par
\par  \vspace{.5cm} \textbf{Step five: The special case $t=2$ and $C'$ hyperelliptic.}
Assume $\Delta'=Q_1+Q_2+Q_3+Q_4$ and $Q_1+Q_2\sim \delta'$. Let $\iota$ be the hyperelliptic involution on $C'$. Since $C'$ is hyperelliptic $\Gamma'$ has two components:
\[ \Gamma'_1=\{(D,D+\iota D)\,|\,D\in C'_2\}\,, \quad \text{and} \quad \Gamma'_2=\{(D,H)\,|\, D\in C'^1_2\,, D\leq H\in|\omega_{C'}|\} \,. \]
$\Gamma'_1$ is isomorphic to $C'_t$ and $\Gamma'_2$ lies above $C'^1_t$ and has generically a $1$-dimensional fiber above it. Thus we can decompose
\[ \hat{Z}'_{M'}=\hat{Z}'_{M',1}\cup \hat{Z}'_{M',2}\,,\quad \text{where} \quad \hat{Z}'_{M',k}\subset \Gamma'_k \,. \]
We then define $Z'_{M',k}=q(\hat{Z}'_{M',k})$. First, suppose $C''$ not hyperelliptic. Then in the intersection $Z_M\cap W$ we need to rule out all points coming from $Z'_{M',2}\times_{\PP^1}Z''_{M''}$ since these are all not in the domain of $\tilde{\GG}$. Since $\alpha'(Z'_{M',2})=\{\eta_{C'}\}$ a single point and the morphism $Z'_{M',2}\to \PP^1$ is of degree $2$, we have
\begin{align*}
    [Z'_{M',2}\times_{\PP^1}Z''_{M''}]\cdot [W]&=2\Nm''_\ast[Z''_{M''}]\cdot[pt] \\
    &= 2 b_{g-3} p''_\ast [C'']\cdot[pt] \\
    &= 4\binom{2g-6)}{g-3} \,,
\end{align*}
where in the second line we use \ref{Equ: class of Z M}. Same as in step four, we need to rule out points above $(p'^\ast M',p''^\ast M'')$, from which $2\cdot 2^{g-2}$ are in $Z'_{M',1}\times_{\PP^1}Z''_{M''}$. We also need to rule out points above $(R',R'')$ from which there are
\[ \tns(\tilde{C}/C)\coloneqq \mathrm{Card}\{ (D',D'')\in E_2\times E_{g-2}\,|\, D^i\leq \Delta^i\,, D'+D''\sim \delta\,, D' \not\sim \delta' \}\,. \]
The multiplicities are the same as in step four, thus
\begin{align*}
    \deg \GG &= \deg \tilde{\GG}/2 \\
    &=\left([Z_M]\dot [W]-[Z'_{M',2}\times_{\PP^1} Z''_{M''}]\dot [W]- 2^{g-1}-2\tns(\tilde{C}/C)\right)/2 \\
    &=2\binom{2g-4}{g-2}+4b_{g-3}-2^{g-2}-\tns(\tilde{C}/C) \,.
\end{align*}
Finally, we deal with the case when $(g,t)=(4,2)$ and both $C'$ and $C''$ are hyperelliptic. We have the decomposition
\[Z_M=\bigcup_{1\leq i,j\leq 2} Z'_{M',i}\times_{\PP^1} Z''_{M'',j} \,,\]
where $Z''_{M'',j}$ is defined as $Z'_{M',j}$ was. In $Z'_{M',1}\times_{\PP^1} Z''_{M'',1}$ there are $4$ points above $(p'^\ast M',p''^\ast M'')$ to rule out and above $(R',R'')$ there are
\[ \tns(\tilde{C}/C)= \mathrm{Card}\{ (D',D'')\in E_2\times E_2\,|\, D^i\leq \Delta^i\,, D'+D''\sim \delta\,, D' \not\sim \delta' \} \,, \]
points to rule out, with multiplicities respectively $1$ and $2$. We have $[Z'_{M',2}\times_{\PP^1} Z''_{M'',2}]\cdot[W]=0$, thus 
\begin{align*}
    \deg \GG &= \deg \tilde{\GG}/2 \\
    &=
        \big( ([Z_M]-[Z'_{M',2}\times_{\PP^1} Z''_{M''}]-[Z'_{M'}\times_{\PP^1} Z''_{M'',2}]+[Z'_{M',2}\times_{\PP^1} Z''_{M'',2}])\cdot [W] \\
&\quad    - 4 - 2\tns(\tilde{C}/C) \big)/2 \\
    &=24-4-4-2-\tns(\tilde{C}/C) \\
    &=14-\tns(\tilde{C}/C)\,.
\end{align*}

\subsection{Numerical analysis}
We have the following:
\begin{proposition}\label{Prop: Numerical analysis Gauss degree Egt}
Let $g\geq 4$ and $2\leq t \leq g/2$, then
\[ \deg \GG(\EE_{g,1})>\deg \GG(\mathcal{J}_g)> \deg \GG(\EE_{g,t}) \,. \]
Moreover the function $t\mapsto \deg \GG(\EE_{g,t})$ is strictly decreasing on the range $1\leq t\leq g/2$.
\end{proposition}
\begin{proof}
Fix $g\geq 4$, recall that for $1\leq t \leq g/2$ we have
\[ \deg \GG(\EE_{g,t})=\left(b_{t-1}b_{g-t}+b_{t}b_{g-t-1} \right)-2^{g-1} \,, \]
where $b_n=\binom{2n}{n}$ is the middle binomial coefficient. This is clearly a decreasing function of $t$. We have
\[ \deg \GG(\EE_{g,1})= b_{g-1}+2b_{g-2}-2^{g-1}>b_{g-1} \,. \]
Using the identity $b_n=(4-2/n)b_{n-1}$ we have
\begin{align*}
 b_{g-1}&=\left(4-{2}/(g-1)\right)b_{g-2} \\
&= 2b_{g-2}+\left(2-{2}/({g-1})\right)\left(4-{2}/({g-2}) \right)b_{g-3}\,,
\end{align*}
thus
\begin{align*}
    \deg \GG(\mathcal{J}_g)-\deg \GG(\EE_{g,2}) &=
    b_{g-1}- \left(2b_{g-2}+6b_{g-3}-2^{g-1} \right) \\
    & > \left(2-\frac{12}{g-2} \right)b_{g-3} \geq 0
\end{align*}
for $g\geq 8$. The remaining values can be checked by hand:
\begin{center}
    \begin{tabular}{c|c|c|c|c}
       $g$ & 4 & 5 & 6 & 7 \\ \hline
       $\deg \GG(\EE_{g,2})$  & 16 & 60 & 228 & 860 \\
       $\deg \GG(\mathcal{J}_g)$ & 20 & 70 & 252 & 924 
    \end{tabular}
\end{center}
\end{proof}

\subsection{The Gauss locus}
In analogy to the Andreotti-Mayer loci, Codogni, Grushevski and Sernesi \cite{Gru17} define the \emph{Gauss loci} by
\[ \mathcal{G}^{(g)}_d \coloneqq \{ (A,\Theta)\in \mathcal{A}_g\,|\, \deg \GG_\Theta \leq d\} \subset \mathcal{A}_g \,, \quad \text{for $d\geq 0$.} \]
We will drop the supscript when it is clear from the context. By Codogni and Krämer \cite{KraemerCodogni}, the Gauss loci are closed in $\mathcal{A}_g$ and the Jacobian locus $\mathcal{J}_g$ is an irreducible component of $\mathcal{G}_{b_{g-1}}$. By Theorem \ref{maintheorem: degree Gauss Map on Egt, in general}, a general ppav in $\EE_{g,0}$ has Gauss degree $b_{g-1}$ as well. With the same methods as those of \cite{KraemerCodogni}, one obtains the following corollary to Theorem \ref{maintheorem: degree Gauss Map on Egt, in general}:
\begin{corollary}\label{Cor: EEg0 irreducible component of Gauss locus}
    For $g\geq 4$, the locus $\EE_{g,0}$ is an irreducible component of $\GG_{b_{g-1}}$, distinct from $\mathcal{J}_g$.
\end{corollary}
\begin{proof}
    Clearly, a general ppav in $\EE_{g,0}$ is not a Jacobian. Suppose that $\EE_{g,0}$ is strictly contained in an irreducible component $Z\subset \mathcal{G}_{b_{g-1}}$. As $\EE_{g,0}$ is an irreducible component of the Andreotti-Mayer locus $\mathcal{N}_{g-4}$ by \cite{Debarre1988}, a general ppav $(A,\Theta)\in Z$ will have $\dim \Sing(\Theta)<g-4$. The degeneration argument of \cite{KraemerCodogni} would then imply that the Gauss degree has to drop when we degenerate $(A,\Theta)$ to a ppav $(P,\Xi)\in\EE_{g,0}$, provided $\Sing(\Xi)$ is non-degenerate, i.e. not stable under translations by a positive-dimensional subvariety of $P$. From the proof of Theorem \ref{Thm: degree Gauss Map on Egt, in general} it follows that for a general $(P,\Xi)\in \EE_{g,t}$, we have
    \begin{align*} 
    \Sing(\Xi)&=\pi''^\ast \restr{P''}(\Sing(\Theta'')\cap P'')\\
    &=\pi''^\ast \restr{P''}\left( p''^\ast \Pic^1(E)+W_{g-1}(C'') \right)\cap P'' \,,
    \end{align*}
    which is not degenerate.
\end{proof}

\section{The boundary of \texorpdfstring{$\EE_{g,t}^\ast$}{Egt}}\label{sec: Pryms arising from admissible covers}
We will now investigate the ppav's in the closure of $\EE_{g,t}^\ast$. 
\subsection{Preliminaries}\label{Sec: Preliminaries boundary of EEgt}
First, let us recall some generalities about coverings of nodal curves. Let $C$ be a complete connected nodal curve. Its arithmetic genus is defined by 
\[ g=p_a(C)\coloneqq 1-\chi(\BO_C)\,.\]
From now on, we will always mean the arithmetic genus when we talk about the genus of a nodal curve. Let $\omega_C$ be its dualizing sheaf, which is a line bundle of total degree $2g-2$ \cite[Chapter 10]{arbarello2}. Let $\beta:N\to C$ denote the normalization. Let $q\in C$ be a singular point and $p_1,p_2\in N$ be its preimage by $\beta$. Then locally near $q$, $\omega_C$ is the sheaf of $1$-forms $\omega$ on $N$, regular except possibly at $p_1$ and $p_2$, verifying
\[ \mathrm{Res}_{p_1}(\omega)+\mathrm{Res}_{p_2}(\omega)=0 \,. \]
Let $\tilde{C}$ be a complete connected nodal curve of genus $g+m$, and $\pi:\tilde{C}\to C$ be a double cover, corresponding to an involution $\sigma:\tilde{C}\to \tilde{C}$. Necessarily $\Sing(C)\subset \pi(\Sing(\tilde{C}))$. At a singular point of $p\in\tilde{C}$, the involution $\sigma$ is of one of three types (see \cite{Beauville1977}):
\begin{enumerate}[1)]
    \item The involution does not fix $p$, thus it exchanges $p$ with another singular point $p'\in \tilde{C}$. Then $\pi$ is étale at $p$ and $C$ is singular at $q=\pi(p)$.
    \item The involution fixes $p$ and exchanges both branches. Then we can identify $\hat{\BO}_{\tilde{C},p}=\CC[[x,y]]/(xy)$ such that
    \[ \sigma^\ast x=y\,, \quad \text{thus} \quad \hat{\BO}_{C,q}=\CC[[x+y]]/(xy)\simeq \CC[[u]]\,, \]
    for $u=x+y$. Thus $C$ is smooth at $q$.
    \item The involution fixes $p$ but does not exchange both branches. Then we can identify the completion of the local ring $\hat{\BO}_{\tilde{C},p}=\CC[[x,y]]/(xy)$ such that the action of $\sigma$ is
    \[ \sigma^\ast x=-x\,, \quad \sigma^\ast y=-y \,. \]
    Thus $\hat{\BO}_{C,q}=\CC[[x^2,y^2]]/(xy)\simeq \CC[[u,v]]/(uv)$ for $u=x^2$, $v=y^2$ and $C$ has a node at $q$.

\end{enumerate}
Moreover, from the above description we have the following:
\begin{proposition}\label{Prop: sing of type (1,2) is flat}
    The morphism $\pi:\tilde{C}\to C$ is flat at singular points of type (1) and (2), and not flat at singular points of type (3).
\end{proposition}
\begin{proof}
    This can be checked in the completion of the local ring. For example, for a singularity of type (2) we have
    \[ \hat{\BO}_{\tilde{C},p}\otimes k(q) = \CC[[x,y]]/(xy)\otimes_{\CC[[x+y]]}\CC[[x+y]]/(x+y)=\CC[x]/(x^2)\,. \]
    Thus $\pi_\ast \BO_{\tilde{C}}$ is locally free of rank $2$ at $q$, thus $\pi$ is flat at $q$.
\end{proof}
\begin{definition}\label{Def: Coverings of type (1,2,3) and (ast)}
We say a double covering $\pi:\tilde{C}\to C$ is of type (1,2) (resp. of type 3) if $\pi$ is of type (1) or (2) (resp. 3) at all singular points. \\
We say $\pi:\tilde{C}\to C$ is of type $(\ast)$ if it is of type (3) at the singular points and étale everywhere else. This corresponds to Beauville's definition of $(\ast)$ \cite[Sec. 3]{Beauville1977}.
\end{definition}
From the above discussion, we immediately have the following:
\begin{proposition}\label{Prop: definition of delta}
    Let $\pi:\tilde{C}\to C$ be a double cover of nodal curves, and $\sigma:\tilde{C}\to \tilde{C}$ the associated involution. Then
    \[ \pi_\ast \BO_{\tilde{C}}=\BO_C\oplus \mathscr{F} \,, \]
    for some rank $1$ torsion free sheaf $\mathscr{F}$. $\BO_C$ (resp. $\mathscr{F}$) is the $+1$ (resp. $-1$) eigenspace under the action of $\sigma$ on $\pi_\ast \BO_{\tilde{C}}$. Moreover, $\mathscr{F}$ is a line bundle if and only if $\pi$ is of type (1,2), and then we call
    \[ \delta \coloneqq \mathscr{F}^{-1} \]
    the line bundle associated to $\pi$.
\end{proposition}

Let $\tilde{K}$ be the ring of rational functions on $\tilde{C}$, i.e. the product of the function fields of the irreducible components. Recall \cite[Sec. 3]{Beauville1977} that the group of Cartier divisors supported at a singular point $q\in \tilde{C}$ is $\tilde{K}^\ast_q/\BO^\ast_{\tilde{C},q}$ and we have an exact sequence
\begin{equation}\label{Equ: short exact sequence singular cartier divisors}
    0 \to \CC^\ast \to \tilde{K}^\ast_q /\BO_{\tilde{C},q} \overset{v_1,v_2}{\longrightarrow} \ZZ \oplus \ZZ \to 0 \,,
\end{equation} 
where $v_1,v_2$ are the valuations on $\tilde{K}$ at the two points sitting above $q$ in the normalization. The term on the left in the above sequence corresponds to the line bundles on $\tilde{C}$ obtained by taking the trivial line bundle on the normalization $\tilde{N}$, and identifying both fibers above $q$ with $\lambda\in \CC^\ast$.
\begin{remark}
    We can choose coordinates $x,y$ on both branches at a singular point $p\in \tilde{C}$ to split the above exact sequence and identify
    \[ \tilde{K}^\ast_q/\BO_{\tilde{C},q}\simeq \CC^\ast \times \ZZ \times \ZZ \]
    Then the involution acts in the following way in each of the three singularity types
\begin{enumerate}[Type 1)]
\item Using the coordinates $\sigma^\ast x$ and $\sigma^\ast y$ at $\sigma(p)$ we have
        \begin{align*} 
        \sigma(\lambda,m,n)_p&=(\lambda,m,n)_{\sigma(p)} \,,\\
        \pi^\ast(\lambda,m,n)_{\pi(p)}&=(\lambda,m,n)_p+(\lambda,m,n)_{\sigma(p)}\,.
        \end{align*}
\item Assume $\sigma^\ast x=y$, then
        \begin{align*}
            \sigma(\lambda,m,n)_p=(\lambda^{-1},n,m)_p\,, \quad \text{and} \quad \pi^\ast n[q]=(1,n,n) \,. 
        \end{align*}
\item Assume $\sigma^\ast x=-x$ and $\sigma^\ast y=-y$, and $u,v$ are coordinates near $q$ such that $\pi^\ast u=x^2$, $\pi^\ast v=y^2$. Then
\begin{equation*}
\sigma(\lambda, m, n)_p=((-1)^{m+n} \lambda, m, n)_p \,, \quad \text{and} \quad 
        \pi^\ast(\lambda, m, n)_q=(\lambda,2n,2m)_p \,. 
\end{equation*}
    \end{enumerate}
\end{remark}
We have the following analogue of the Hurwitz Theorem in the singular case (compare with \cite[Lem. 3.2 page 157]{Beauville1977}).
\begin{proposition}\label{Prop: definition ramification divisor}
    Let $\pi:\tilde{C}\to C$ be a double cover of nodal curves. We have
   \[ \omega_{\tilde{C}}=\pi^\ast \omega_C (R)\,, \]
    where $R$ is the Cartier divisor on $\tilde{C}$ defined by
    \begin{align*}
    R&\coloneqq R_{\mathrm{reg}}+R_{\sing} \\
    &=\sum_{\substack{q\in \tilde{C}_{\mathrm{sm}} \\\pi \text{ ramified at } q}} [q] + \sum_{\substack{q\in \Sing(\tilde{C}) \\\text{of type (2)}}} \pi^\ast([\pi(q)])+(-1)_q\,,
    \end{align*}
    where $(-1)_q\in \CC^\ast \subset \tilde{K}_q/\BO_{\tilde{C},q}$ corresponds to the trivial bundle on $\tilde{N}$ with both fibers above $q$ identified by $(-1)$, along with the trivial section (defined away from $q$). We call $R$ the \emph{ramification divisor} associated to $\pi$.
\end{proposition}
\begin{proof}
    At smooth points $p\in \tilde{C}$, the result is classical. It is also obvious that the equality holds at singular points of type (1). Let $p\in \tilde{C}$ be a singular point of type (3) for $\sigma$. Using the notations above, a generator of $\omega_{C,q}$ is 
    \[ \frac{du}{u}=-\frac{dv}{v}\,, \quad \text{and} \quad \pi^\ast \left(\frac{du}{u}\right)= \frac{ dx^2}{x^2}=\frac{2dx}{x} \,, \]
    which is a generator of $\omega_{\tilde{C},p}$. Finally, if $p$ is a singular point of type (2) for $\sigma$, a generator of $\omega_{C,q}$ is $du$ and we have
    \[ \pi^\ast(du)=dx+dy=(x-y)\frac{dx}{x} \,. \]
\end{proof}

\begin{proposition}\label{Prop: pullback delta in case 1,2)}
    Let $\pi: \tilde{C}\to C$ be a morphism of type (1,2). Let $\delta$ and $R$ be the line bundle and ramification divisor associated to $\pi$. Then
    \begin{align*}
    \omega_{\tilde{C}}&=\pi^\ast(\omega_C\otimes \delta)\,, \\
        \pi^\ast \delta &= \BO_{\tilde{C}}(R) \,,\\
        \quad \delta^{\otimes 2}&=\BO_C(\pi_\ast R) \,. 
    \end{align*} 
\end{proposition}
\begin{proof}
    $\pi$ is flat by Proposition \ref{Prop: singularities are quadratic of maximal rank}. We have
    \begin{align*}
       \sheafhom_E(f_\ast \BO_{\tilde{C}},\omega_C) &= \sheafhom_C(\BO_C\oplus \delta^{-1},\omega_C) \\
       &=\omega_C\otimes(\BO_C\oplus \delta) \\
       &=\pi_\ast(\pi^\ast(\omega_C\otimes \delta))\,.
    \end{align*}
    Thus by Serre duality for finite flat morphism (see \cite[Ex. III 6.10,7.2]{hartshorne}) we have
     \[ \omega_{\tilde{C}}=\pi^! \omega_C = \pi^\ast(\omega_C \otimes \delta) \,. \]
    The last two equalities follow from Proposition \ref{Prop: definition ramification divisor} and the fact that $\delta^{\otimes 2}=\Nm\circ \pi^\ast \delta$.
       
\end{proof}

Recall that by an effective divisor we mean a divisor associated to a regular section of some line bundle. On the smooth part, this is equivalent to having a positive coefficient in front of every point. The effective divisors supported on a singular point $q\in \tilde{C}$ correspond to $\{1\}\cup (v_1,v_2)^{-1}(\ZZ_{>0}\oplus \ZZ_{>0})$ under the identification \ref{Equ: short exact sequence singular cartier divisors}. We call an effective divisor $D$ on $\tilde{C}$ $\pi$-\emph{simple} if there is no effective divisor $F$ on $C$ such that
\[ D-\pi^\ast F \]
is effective. We have the following generalization of \cite[page 338]{Mumford1974}:
\begin{proposition}\label{Prop: Pullback line bundle sequence singular case}
    Let $\pi:\tilde{C}\to C$ be a double covering of nodal curves. Suppose $\pi_\ast \BO_{\tilde{C}}=\BO_C\oplus \mathscr{F}$. Let $L$ be a line bundle on $C$ and $D$ an effective $\pi$-simple divisor on $\tilde{C}$, supported away from the singular points of type (3) for $\pi$. Then there is an exact sequence
    \[ 0\to L \to \pi_\ast(\pi^\ast L(D)) \to L(\pi_\ast D)\otimes \mathscr{F} \to 0 \,. \]
\end{proposition}
\begin{proof}
    We adapt the proof in \cite[page 338]{Mumford1974} to our case. Let $L$ be a line bundle on $C$ and $D$ a $\pi$-simple divisor. Recall that
    \[ \pi_\ast(\pi^\ast(L))=L\oplus (L\otimes \mathscr{F}) \,, \]
    and that under this identification $L$ (resp. $L\otimes \mathscr{F}$) is the $+1$ (resp.$-1$)-eigenspace under the action of the involution $\sigma:\tilde{C}\to \tilde{C}$ associated to $\pi$. We thus have the following inclusion of sheaves
    \begin{center}
        \begin{tikzcd}
            \pi_\ast(\pi^\ast(L(\pi_\ast D))) \arrow[r,phantom,"\simeq"] & L(\pi_\ast D) \arrow[r,phantom,"\oplus"] & L(\pi_\ast D)\otimes \mathscr{F} \\
            \pi_\ast(\pi^\ast L(D)) \arrow[u,phantom,"\cup"] & \cup & \cup \\
            \pi_\ast(\pi^\ast(L))\arrow[u,phantom,"\cup"] \arrow[r,phantom ,"\simeq"] & L \arrow[r,phantom,"\oplus"] & L\otimes \mathscr{F}\,.
        \end{tikzcd}
    \end{center}

At a point $p$ where $D$ is non-trivial, the sheaf $\pi_\ast(\pi^\ast L(D))$ is generated (as an $\BO_C$-module) by $L\oplus(L\otimes \mathscr{F})$ plus a section $s$. Let $t$ be a section on $\tilde{C}$ with divisor $D$. Let $q=\pi(p)$. We can assume that all line bundles a trivial on a neighborhood $U$ of $q$ and on $V=\pi^{-1}(U)$. We can then assume $s=1$ in the trivialization of $\pi^\ast L(D)$, and the image of $s$ under the injection
\[ \pi^\ast L(D)\hookrightarrow \pi^\ast(L(\pi_\ast D)) \]
is $\sigma^\ast t$. Thus the projection of $s$ onto $\pi_\ast \pi^\ast(L(\pi_\ast D))^{-}$ is
\[ \sigma^\ast t-t \,. \]
We can check that $\sigma^\ast t-t$ generates $\pi_\ast \pi^\ast(L(\pi_\ast D))^{-}$ in every possible configuration of $D$:
\begin{itemize}
    \item If $p$ is a smooth non-ramified point, this is trivial.
    \item If $p$ is a smooth ramification point of $\pi$ then necessarily $t=z$ for some coordinate $z$, and $\sigma^\ast t-t=-2z$.
    \item If $p$ is a singular point of type (1), then say $p'=\sigma(p)$. The only effective $\pi$-simple divisors supported at $p,p'$ are of the type $D=(\lambda,n,m)_p\in \CC^\ast \times \ZZ_{>0}\times \ZZ_{>0}$. But then 
    \[ \divv(\sigma^\ast t - t)=1 \,. \]
    \item If $p$ is a singular point of type (2), then the only effective $\pi$-simple divisors supported at $p$ are of the form $(\lambda,1,1)_p$ with $\lambda\neq 1$ (resp. $(\lambda,1,2)_p,(\lambda,2,1)_p$ with $\lambda\in \CC^\ast$). We have
    \begin{align*}
    \sigma^\ast t-t &= \lambda x+ y-(x+\lambda y)=(\lambda-1)(x-y) \\
    (resp. &= \lambda x^2+y-(x+\lambda y^2)=(x-y)(-1+\lambda(x+y)) \,)\,,
    \end{align*}
    for some coordinates $x,y$ at $p$.
\end{itemize}
\end{proof}
\begin{remark}
    The above proposition fails if $D$ intersects points of type (3). For instance, if $D=(\lambda,1,1)_p$ for some point $p\in \tilde{C}$ of type (3), then using the notations of the proof above we have
    \[ \sigma^\ast t - t =(-x-\lambda y)-(x+\lambda y)=-2(x+\lambda y) \,. \]
    But the local generators for $\pi_\ast(\pi^\ast (L(\pi_\ast D)))^{-}$ at $q$ are $x,y$.
\end{remark}

\begin{remark}
The above discussion carries out to families of double coverings of connected nodal curves $\tilde{\cC}\overset{\pi}{\longrightarrow} \cC \longrightarrow S$ (by this we mean that the restriction over each fiber $s\in S$ is a double covering of connected nodal curves in the usual sense), where we assume that $\tilde{\cC}\to S$ (and thus $\cC\to S$) is flat. We will assume $S$ to be a smooth curve. The relative dualizing sheaves $\omega_{\tilde{\cC}/S}$ and $\omega_{\cC/S}$ are line bundles, and there is an exact sequence 
\[ 0\to \pi^\ast \omega_{\cC/S}\to \omega_{\tilde{\cC}/S} \to \BO_{\mathcal{R}_S} \to 0 \]
for some Cartier divisor $\mathcal{R}_S\subset \tilde{\cC}$. We call $\mathcal{R}_S$ the \emph{relative ramification divisor}. We define
\[ \Delta_S \coloneqq \pi_\ast \mathcal{R}_S \]
to be the \emph{relative branch divisor}. Notice that since singular points of type (2) are mapped to smooth points, $\Delta_S$ is supported on the relative smooth locus of $\cC$. For $s\in S$, the restriction $\mathcal{R}_s$ corresponds to the ramification divisor of $\pi_s$ defined in Proposition \ref{Prop: definition ramification divisor}. For $i\in\{1,2,3\}$, we will say $\pi$ is of type $(i)$ if $\pi_s:\tilde{C}_s\to C_s$ is of type $(i)$ for all $s\in S$. By Proposition \ref{Prop: sing of type (1,2) is flat}, $\pi$ is flat if and only if it is of type $(1,2)$. In that case 
\[ \pi_\ast \BO_{\tilde{\cC}}=\BO_{\cC}\oplus \delta^{-1} \]
for some line bundle $\delta$ on $\cC$ and we have the analogue of Proposition \ref{Prop: pullback delta in case 1,2)}
\begin{align*}
\pi^\ast (\omega_{\cC/S}\otimes \delta) = \omega_{\tilde{\cC}/S}\,, \\
\pi^\ast \delta = \BO_{\tilde{\cC}}(\mathcal{R})\,, \\
\delta^{\otimes 2} = \pi_\ast \mathcal{R} \eqqcolon \Delta\,.
\end{align*}
\end{remark}


\subsection{Degenerations in \texorpdfstring{$\EE_{g,t}$}{Egt}}\label{Subsec: Degenerations in Egt}
We will now investigate the ppav's in the closure of the loci $\EE_{g,t}^\ast$ studied in Section \ref{Sec: Construction of the Families Eg,t}. The main result is Lemma \ref{Lemma: decomposition of Egt}, in which we completely characterize such degenerations. The main ingredient in the proofs is Beauville's theory of generalized Prym varieties \cite{Beauville1977}. \par 
Let $(P,\Xi)\in {\EE}_{g,t}\setminus \EE_{g,t}^\ast$. There is a deformation of $(P,\Xi)$ with general member in $\EE_{g,t}^\ast$, thus by \cite[Section 6]{Beauville1977} and \cite[p. 209]{arbarello2}, there is a smooth curve $S$ and a family of stable curves $q:\tilde{\cC}\to S$ and involutions $\sigma,\sigma',\sigma'':\tilde{\cC}\to \tilde{\cC}$ such that
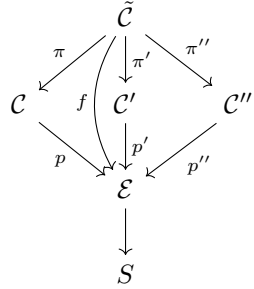
\begin{center}
    \begin{tikzcd}
        & \tilde{\cC} \arrow[dl,"\pi"'] \arrow[d,"\pi'"] \arrow[dr,"\pi''"] \arrow[dd,bend right, "f"'] & \\
        \cC \arrow[dr,"p"'] & \cC'\arrow[d,"p'"] & \cC'' \arrow[dl,"p''"] \\
        & \mathcal{E} \arrow[d] & \\
        & S &
    \end{tikzcd}
\end{center}

\begin{enumerate}
    \item For $s\neq 0$, $\tilde{\cC}_s$ is smooth.
    \item $\sigma=\sigma'\circ \sigma''$.
    \item $\cC'\coloneqq \tilde{\cC}/\sigma'\to S $ (resp. $\cC''\coloneqq \tilde{\cC}/\sigma''\to S$) is a flat family of nodal curves of genus $t+1$ (resp. $g-t+1$).
    \item $\eE\coloneqq \tilde{\cC}/\langle \sigma, \sigma' \rangle \to S$ is a flat family of nodal curves of genus $1$.
    \item For $\mu\in \{\emptyset,',''\}$, $p^\mu:\cC^\mu \to \eE$ corresponds to an involution $\tau^\mu:\cC^\mu\to\cC^\mu$.
    \item For each $s\in S$, $\mu\in\{\emptyset,',''\}$, $\sigma_s^\mu:\tilde{\cC}_s\to \tilde{\cC}_s$ and $\tau^\mu_s:\cC_s\to\cC_s$ are not the identity on any irreducible component.
\end{enumerate}
The flatness of the families $\cC^\mu\to T$ and $\eE\to T$ follows from the argument in \cite[p. 210]{arbarello2}. The last assertion is a consequence of \cite[p. 176]{Beauville1977}. Let
\[ \tilde{C}\coloneqq \tilde{\mathcal{C}}_0 \,, \quad \text{and} \quad C\coloneqq \mathcal{C}_0 \,. \]
Thus $P$ is the Prym corresponding to $\pi_0:\tilde{C}\to C$. Moreover, under the assumption that $(P,\Xi)$ is neither a Jacobian nor decomposable, by \cite{Beauville1977} we can assume that $\pi_0$ is of the following type
\begin{center}
$(\ast)$ \quad $\pi_0$ is of type (3) at the singular points, and étale everywhere else.
\end{center}
We have the following:
\begin{proposition}\label{Prop: pi of type (3) and p of type (1,2) etc}
If $\pi$ is of type $(3)$, then $p'$, $p''$ are of type $(3)$ and $\pi'$, $\pi''$, $p$ are of type $(1,2)$.
\end{proposition}
\begin{proof}
The Galois group associated to $\tilde{\cC}\to \eE$ is $\langle \sigma, \sigma', \sigma'' \rangle\simeq (\ZZ/2\ZZ)^2$, thus if we look at the action of $\langle \sigma,\sigma',\sigma''\rangle$ on the level of the normalization of the curves, there cannot be a ramification point of order $4$. Since $\pi$ is already of type $(3)$ at all the singular points, $p$ has to be of type $(1,2)$. The assertion then follow for $\pi',\pi'',p',p''$ by noticing that $\tau:\cC\to \cC$ corresponds to the action of $\sigma'$ (which equal to that of $\sigma''$) on pairs of points $\{P,\sigma(P)\}$ for $P\in \tilde{\cC}$ (and the same goes for $\tau'$ and $\tau''$).
\end{proof}
Thus by the previous Section, we can associate to $p$ a line bundle $\delta_S$ and a branch divisor $\Delta_S$ on $\eE$ (not intersecting the relative singular locus $\Sing(\eE/S)$), and a ramification divisor $\mathcal{R}_S$ on $\cC$, such that
\begin{equation}\label{Equ: identities of delta and relative omega_C}
\begin{split}
p^\ast (\omega_{\eE/S}\otimes \delta_S) = \omega_{{\cC}/S}\,, \\
p^\ast \delta_S = \BO_{\cC}(\mathcal{R}_S)\,, \\
\delta_S^{\otimes 2} = \BO_E(\Delta_S)\,.
\end{split}
\end{equation}
For $i\in\{',''\}$ define $\mathcal{R}^i_S \subset \cC^i$ to be the ramification divisor of $p^i$ defined in the previous section. By \ref{Prop: definition ramification divisor}, we have
\begin{align*}
    p^{i,\ast} \omega_{\eE/S}=\omega_{\cC^i/S}(\mathcal{R}_S^i)\,,\\
    \Delta^i_S \coloneqq p^i_\ast \mathcal{R}^i_S\,, \\
    \Delta_S=\Delta'_S+\Delta''_S\,.
\end{align*}
Let $\tilde{\mathcal{R}}'_S\subset \tilde{\cC}$ (resp. $\tilde{\mathcal{R}}''_S$) be the ramification divisor of $\pi'$ (resp. $\pi''$) and 
\begin{align*}  
\tilde{\Delta}_S' \coloneqq \pi'_\ast(\tilde{\mathcal{R}}_S')\,, \quad \tilde{\Delta}_S'' \coloneqq \pi''_\ast(\tilde{\mathcal{R}}_S'')
\end{align*}
be the associated branch divisors. Also by the previous section for $i\in\{',''\}$ there are line bundles $\tilde{\delta}^i_S$ associated to $\pi^i$ such that
\begin{align*}
    \pi'^{\ast} \tilde{\delta}'_S&=\BO_{\tilde{\cC}}(\tilde{\mathcal{R}}'_S) \,, & \pi''^{\ast} \tilde{\delta}''_S&=\BO_{\tilde{\cC}}(\tilde{\mathcal{R}}''_S) \,, \\
         {\tilde{\delta'}_S}^{\otimes 2}&=\BO_{\cC'}(\tilde{\Delta}'_S) \,, &
    \tilde{\delta''}_S^{\otimes 2}&=\BO_{\cC''}(\tilde{\Delta}''_S) \,.
\end{align*}
We have
\begin{align}\label{Equ: pullback ramification divisor R tilde in Egt}
\begin{aligned}
\tilde{\Delta}_S'&=p'^\ast \Delta_S'' \,, \\
     \tilde{\mathcal{R}}_S'&=\pi''^{\ast}(\mathcal{R}_S'')\,, \\
p'^\ast \delta_S &= \BO_{\cC'}(\mathcal{R}'_S)\otimes \tilde{\delta}'_S\,,
\end{aligned} &&
\begin{aligned}
     \tilde{\Delta}_S''&=p''^\ast \Delta_S' \,, \\
     \tilde{\mathcal{R}}_S''&=\pi'^\ast(\mathcal{R}_S')\,, \\
     p'^\ast \delta_S &= \BO_{\cC''}(\mathcal{R}''_S)\otimes \tilde{\delta}''_S\,.
\end{aligned}
\end{align}
The equality is easily seen to hold away from the singular fiber, and thus holds everywhere by a deformation argument. All the identities above specialize to the corresponding identities on the singular fiber. For each of the notations above, to denote the restriction to the fiber above $0\in S$ we omit the subscript $S$.

Notice $\Delta$ has double points at all singular points of type (2) of $p$, but the points of $\Delta'$ and $\Delta''$ correspond to regular ramification points (i.e. ramification points on the smooth locus) of $p'$ and $p''$, thus $\Delta'$ and $\Delta''$ must be reduced. Thus $\Delta'$ and $\Delta''$ must contain each of the singular points of type (2) of $p$ with multiplicity one. In particular, if $t=0$ then $\Delta'=0$. We have shown: 
\begin{proposition}\label{Prop: p is of type (1) when t=0}
If $t=0$ then $p$, $\pi'$ and $\pi''$ are of type (1).
\end{proposition}

We define the relative Prym varieties
\begin{align*} 
\mathcal{P}&=\Ker^0(\Nm_\pi : J\tilde{\cC}\to J\cC ) )\,, \\
\mathcal{P}'&= \Ker^0(\Nm': J \cC'\to J\eE  )\,, \\
\mathcal{P}''&=\Ker^0(\Nm'': J\cC''\to J\eE) )\,. 
\end{align*}
By \cite[Section 6]{Beauville1977}, the abelian scheme $\pP\to S$ is canonically endowed with a principal polarization $L_\pP$ (which is one half the induced polarization from $J\tilde{\cC}$). We endow $\pP'$ (resp. $\pP''$) with the induced polarization $L_{\pP'}$ (resp. $L_{\pP''})$) from $J\cC'$ (resp. $J\cC''$). By Proposition \ref{Prop: pullback of L_P to P'xP''} we know that $L_{\pP'}$ and $L_{\pP''}$ are of type $(1,\dots,1,2)$ and the pullback induces an isogeny of polarized abelian schemes of degree $4$ (resp. $2$) for $t\geq 1$ (resp. $t=0$):
\[ g\coloneqq (\pi'^\ast+\pi''^\ast):\pP'\times \pP'' \to \pP \,, \quad g^\ast L_\pP=L_{\pP'}\boxtimes L_{\pP''} \,. \]
Thus by specialization we have the following:
\begin{proposition}\label{Prop: covering g:P' times P'' to P in the singular case}
Let $(P,\Xi)\in \EE_{g,t}$. If $t\geq 1$ (resp. $t=0$), there is a degree $4$ (resp. $2$) isogeny of polarized abelian varieties
\[ g\coloneqq (\pi'^\ast + \pi''^\ast)\restr{P'\times P''} : P'\times P'' \to P \,, \qquad g^\ast L_P=L_{P'}\times L_{P''} \,, \]
where for $i\in \{',''\}$, $P^i=\Ker(\Nm_{p^i}:JC^i\to JE)$ and $L_{P^i}=L_{JC^i}\restr{P^i}$ is the restriction of the usual polarization on $JC^i$.
\end{proposition}
We say that a curve $E$ is a cycle of $\PP^1$'s if $E$ is a nodal curve, its normalization is a union of rational curves and its dual graph is cyclic. We have the following:
\begin{lemma}\label{Lemma: E cannot have tails}
Let $(P,\Xi)\in \BE_g\setminus (\mathcal{J}_g\cup \mathcal{A}_g^\dec)$. With the notations introduced above, $\pi$ is of type $(\ast)$ and $E$ is either smooth and irreducible, or $E$ is a cycle of $\PP^1$'s. \\
Moreover, if $E=E_1\cup \cdots E_n$ and $(d_1,\dots,d_n)=\underline{\deg}(\delta)$, then $d_i>0$ for $1\leq i \leq n$.
\end{lemma}
\begin{proof}
By \cite[Th. 5.4]{Beauville1977}, under the above assumptions, $\pi$ is of type $(\ast)$. We will show that $E$ does not have tails, i.e. we cannot write $E$ as a union
\[ E=E_1\cup E_2 \,, \]
where $E_1$ and $E_2$ meet in a single point. This will prove the first part of the Lemma as $E$ is nodal, irreducible, and of arithmetic genus $1$. Assume that $E=E_1\cup E_2$ with $E_1$ and $E_2$ meeting at a single point. As the arithmetic genus of $E$ is $1$, we can assume that $E_2$ is irreducible and of genus $0$. For $i\in\{',''\}$ and $j\in \{1,2\}$ let $C^i_j=p^{i,-1}(E_j)$ and
\[ P^i_j=\Ker^0(\Nm_{p^i}:JC^i_j\to JE_j)\,, \quad L_{P^i_j}= L_{JC^i_j}\restr{P^i_j}\,. \]
Then $P^i=P^i_1\times P^i_2$ and $L_{P^i}=L_{P^i_1}\boxtimes L_{P^i_2}$. Moreover, since $E_2$ is of genus $0$, we have
\[ JE_2=\{0\}\,, \quad \text{thus} \quad P^i_2\simeq JC^i_2 \,. \]
Thus by \ref{Prop: covering g:P' times P'' to P in the singular case}, there is an isogeny of polarized abelian varieties
\[ g: P'_1\times JC'_2  \times P''_1\times JC''_2 \to P \,, \]
where the polarization on the left-hand side is $L_{P'_1}\boxtimes L_{JC'_2}\boxtimes L_{P''_1}\boxtimes L_{JC''_2}$. But since $L_{JC'_2}$ and $L_{JC''_2}$ are principal, the kernel of $g$ has to be in $P'_1\times\{0\}\times P''_1\times\{0\}$. Thus $P$ has $JC'_2\times JC''_2$ as a factor. Thus by assumption both have to be trivial, i.e. $C'_2$ and $C''_2$ are of genus $0$. By \ref{Prop: pi of type (3) and p of type (1,2) etc} $p'$ and $p''$ are of type (3), thus for $i\in\{',''\}$, $C^i_2\to E_2$ ramifies at the point $C^i_2\cap C^i_1$ and another point. Thus $\deg \Delta'\restr{E_2}=\deg \Delta''\restr{E_2}= 1$. Thus $\deg \Delta\restr{E_2}=2$ and thus since $p$ is of type (1,2), $C_2=p^{-1}(E_2)$ is of genus $0$, and $\# (C_1\cap C_2)=2$. Since $\pi$ is of type (3), this implies $\tilde{C}_2=\pi^{-1}(C_2)$ is of genus $0$ and $\#(\tilde{C}_1\cap \tilde{C}_2)=2$, thus contradicting the stability of $\tilde{C}$. \\
The last assertion follows similarly, since if $d_i=0$, then $p^{-1}(E_i)$ is a disjoint union of two $\PP^1$'s, each attached to $C$ at two points. But then $\tilde{C}$ contains a rational component attached at only two points, contradicting its stability.
\end{proof}

We will now focus on the case where $\pi$ is of type $(\ast)$ and $E$ is a cycle of $\PP^1$'s. By previous lemma $\omega_E \simeq \BO_E $ is trivial, thus by \ref{Prop: definition ramification divisor}, we have
\[
    \omega_{C'}=\BO_{C'}(R')\,, \quad \text{and} \quad \omega_{C''}=\BO_{C''}(R'') \,.\]
Moreover, by \ref{Equ: pullback ramification divisor R tilde in Egt} we have
    \[ \pi'^\ast(\omega_{C'})\otimes \pi''^\ast(\omega_{C''})=\omega_{\tilde{C}}\,. \]
Since $p'$ and $p''$ are of type (3), the dual graphs of $E$, $C'$, and $C''$ are canonically identified, and we will from now on identify the multidegrees of line bundles and Cartier divisors. Moreover, since the graph of $C^i$ is cyclic, the multidegrees $\underline{\deg}(\omega_{C'})$ and $ \underline{\deg} (\omega_{C''})$ are even, and we define 
\begin{align*}
\underline{d}'&\coloneqq \frac{1}{2}\underline{\deg}(\omega_{C'})=\frac{1}{2}\underline{\deg}(R')\,, \quad \text{and}\\
\underline{d}''&\coloneqq \frac{1}{2}\underline{\deg}(\omega_{C''})=\frac{1}{2}\underline{\deg}(R'')\,.
\end{align*}
In particular, we have
\begin{align*}
    \underline{\deg}(\delta)&=\frac{1}{2}\underline{\deg}(\Delta) \\
    &= \frac{1}{2}(\underline{\deg}(\Delta')+\underline{\deg}(\Delta''))\\
    &= \underline{d}'+\underline{d}''\\
    &\eqqcolon \underline{d}\,.
\end{align*}
Just as in the smooth case we have the following:
\begin{proposition}\label{Prop: g is of degree 2, singular case}
Let 
\[ \tilde{\Xi}\coloneqq (\Theta'\times \Theta'')\cap \Nm^{-1}(\delta)\subset \Pic^{\ud'}(C')\times \Pic^{\ud''}(C'')\,. \]
The pullback map induces a generically finite morphism of degree $2$
\[
    (\pi'^\ast+\pi''^\ast)\restr{\Tilde{\Xi}}: \tilde{\Xi} \to \Xi\,,
    \]
corresponding on the finite locus to the quotient by the involution $\iota\restr{\tilde{\Xi}}$ where
\begin{align*} 
\iota: \Pic^{\ud'}(C')\times \Pic^{\ud''}(C'') &\to \Pic^{\ud'}(C')\times \Pic^{\ud''}(C'') \\
(L',L'')&\mapsto (\omega_{C'}-\tau'L',\omega_{C''}-\tau''L'') \,. 
\end{align*}
If $t=0$, the morphism $(\pi''^\ast)\restr{\tilde{\Xi}}$ is étale of degree $2$.
\end{proposition}
\begin{proof}
    The proof of Proposition \ref{Prop: g is of degree 2} follows ad verbatim, using Proposition \ref{Prop: Pullback line bundle sequence singular case}.
\end{proof}
It turns out that it is enough to study degeneration's in ${\EE}_{g,0}$, as they cover all the cases by the following lemma:
\begin{lemma}\label{Lemma: Singular Pryms all come from EE g,0}
Suppose $(P_1,\Xi_1)\in {\EE}_{g,t}$ corresponds to the datum of a cycle of $n$ $\PP^1$'s $E_1$, a line bundle $\delta_1$, degrees $\underline{d}'=(d'_1,\dots,d'_k,0,\dots,0)$, $\underline{d}''=(d''_1,\dots,d''_l,0,\dots,0)$ and ramification divisors $(\Delta'_1,\Delta''_1)$. Then $(P_1,\Xi_1)$ is isomorphic to the Prym $(P_2,\Xi_2)\in {\EE}_{g,0}$ corresponding to the datum $(E_2,\delta_2,\Delta_2)$, where 
\begin{itemize}
    \item $E_2$ is a cycle of $k+l$ $\PP^1$'s,
    \item $\underline{d}_2\coloneqq \underline{\deg} \,\delta_2=(d'_1,\dots,d'_k,d''_1,\dots,d''_l)$\,,
    \item $\Delta_2$ is equal to $\Delta'_1$ when restricted to the first $k$ $\PP^1$'s and equal to $\Delta''_1$ on the remaining ones\,,
    \item under the identifications $\Pic^d(E_1)\simeq \CC^\ast \simeq \Pic^d(E_2)$, we have
    \[\delta_1=\delta_2 \,.\]
\end{itemize}
\end{lemma}
\begin{proof}
We will use the notations introduced at the beginning of this section indexed by $1$ or $2$ depending on whether we are referring to $(P_1,\Xi_1)$ or $(P_2,\Xi_2)$. Suppose $\underline{d}'=(d'_1,\dots,d'_n)$, $\underline{d}''=(d''_1,\dots,d''_n)$. Let $N'_1$ and $N''_1$ be the normalization of $C'_1$ and $C''_1$ respectively. For $i\in \{',''\}$ we have the following exact diagrams
\begin{center}
    \begin{tikzcd}
        {} & 0\arrow[d] &0 \arrow[d] &0 \arrow[d]& \\
        0 \arrow[r] & \ZZ/2\ZZ \arrow[r] \arrow[d] & P^i_1 \arrow[r] \arrow[d] & JN^i_1 \arrow[r] \arrow[d] & 0 \\
        0 \arrow[r] & \CC^\ast \arrow[r] \arrow[d,"\times 2"] & JC^i_1 \arrow[r] \arrow[d,"\Nm^i"]& JN^i_1 \arrow[r] \arrow[d] & 0 \\
        0  \arrow[r] & \CC^\ast \arrow[r] \arrow[d] & JE_1 \arrow[r] \arrow[d] & 0 & \\
        & 0 & 0 & & \,. 
    \end{tikzcd}
\end{center}
We thus have the exact sequence
\[ 0 \to (\ZZ/2\ZZ)^2 \to P'_1\times P''_1 \to JN'_1\times JN''_1 \to 0 \,. \]
We know by \ref{Prop: covering g:P' times P'' to P in the singular case} that the pullback $g_1:P'_1\times P''_1\to P_1$ is a degree $4$ isogeny. Clearly $\ZZ/2\ZZ = J_2E_1$ embedded anti-diagonally in $(\ZZ/2\ZZ)^2$ is in the kernel of $g_1$. Let $\nu_1\in P'_1\times P''_1$ be the non-zero element in $J_2E_1 \hookrightarrow \Ker( g_1)$. Let $\delta'_1$ be a square root of $\BO_{E_1}(\Delta'_1)$ in $\Pic^{\underline{d}'}(E_1)$. Let $\delta''_1\coloneqq \delta_1- \delta'_1$, and 
\[ \eta_1=(\eta'_1,\eta''_1) \coloneqq \left(p_1'^\ast \delta'_1(-R_1'),p_1''^\ast\delta''_1(-R_1'')\right)\in \Pic^{0}(C'_1)\times \Pic^{0}(C''_1)\,. \]
We claim that $\langle \eta_1,\nu_1 \rangle=\Ker(g_1)$. We have
 \[g_1^\ast \eta_1 =f_1^\ast \delta_1 -\pi_1^\ast \BO_{C_1}(R_1) =0\,. \]
Clearly $\eta_1 \in P'_1\times P''_1$. Finally, if $N'_{1,1},\dots,N'_{1,n}$ are the irreducible components of $N'_1$, and ${Q}_i^{0},{Q}_i^{\infty}  \in N_{1,i}$ are the points sitting above the nodes in $C'_1$, then
    \[ \eta_1'\restr{N'_{1,i}} =p_1'^{\ast}\restr{N_i} (\BO_{\PP^1}(d'_i))-\BO_{N'_{1,i}}(R'_{1,i})=\BO_{N'_{1,i}}(Q^{0}_i-Q^\infty_i) \,, \]
and similarly for $\eta''_1$. Thus $\eta_1$ is non-zero and distinct from the antidiagonal embedding of $J_2E_1$. By \cite[Cor. 12.5]{OdaSeshadri1979GenJac}, the extension
\[ 0 \to \CC^\ast \to JC_1' \to JN_1'\to 0 \]
corresponds to the line bundle $\eta'_1\restr{JN_1'}=\sum_i \BO_{N'_{1,i}}(P_i^0-P_i^\infty) \in \Pic^0(N'_1)\simeq \Ext(JN'_1,\CC^\ast)$, and similarly for $C''_1$. \\
We now turn to $P''_2$. Clearly in our setup, we have 
\[ JN''_2=JN'_1\times JN''_1 \,. \]
We also have an exact sequence
\[ 0 \to \ZZ/2\ZZ \to P''_2 \to JN'_1\times JN''_1 \to 0 \,. \]
The morphism $g_2:P''_2\to P_2$ is an isogeny of degree $2$ by \ref{Prop: covering g:P' times P'' to P in the singular case}, and by the same reasoning the kernel is generated by 
\[ \eta_2\coloneqq p_2''^\ast \delta_2 (- R_2'') \in JC''_2 \,. \]
Clearly $\eta_2\restr{JN''_2}=(\eta_1'\restr{JN'_1},\eta_1''\restr{JN''_1})$. Moreover, the extension
\[ 0 \to \CC^\ast \to JC''_2 \to JN''_2 \to 0 \]
corresponds to $\eta_2\restr{JN''_2}$. The extension above is thus isomorphic to the quotient of the extension
\[ 0 \to \CC^\ast\times \CC^\ast \to JC'_1\times JC''_1 \to JN'_1\times JN''_1 \to 0 \]
by $\CC^\ast \hookrightarrow \CC^\ast \times \CC^\ast $ embedded anti-diagonally (by functoriality of $\Ext$, the above extension corresponding to $(\eta'_1,\eta_1'')\in \Ext( JN'_1,\CC^\ast)\times \Ext(JN''_1,\CC^\ast)\subset \Ext(JN'_1\times JN''_1,\CC^\ast \times \CC^\ast)$). Thus
\[ (P'_1\times P''_1)/{\nu_1} \simeq P''_2 \,, \]
and 
\[ P_1=(P'_1\times P''_1)/\langle \eta_1 ,\nu_1 \rangle =P''_2/\eta_2=P_2 \,. \]
Moreover the polarization on $P'_1\times P''_1$ and $P''_2$ is induced by that of $JN'_1\times JN''_1=JN''_2$. Thus both Pryms $(P_1,\Xi_1)$ and $(P_2,\Xi_2)$ are isomorphic.
\end{proof}
Just as in the smooth case we have the following: 
\begin{proposition}	\label{Prop: Boundary: construction of the covering from E and a branch divisor}
The following two sets of data are equivalent:
\begin{enumerate}
    \item Tuples $(E,\delta,\Delta',\Delta'')$ where $E$ is a cycle of $n$ $\PP^1$'s, $\delta$ is a line bundle on $E$ of total degree $g$ and $\Delta',\Delta''$ are two effective reduced non-singular divisors on $E$ such that
\[ \delta^{\otimes 2}=\BO_E(\Delta'+\Delta'') \,. \]
 \item Galois towers of double coverings of nodal curves
\[   \begin{tikzcd}
        & \tilde{C} \arrow[dl,"\pi"'] \arrow[d,"\pi'"] \arrow[dr,"\pi''"] & \\
        C \arrow[dr,"p"'] & C'\arrow[d,"p'"] & C'' \arrow[dl,"p''"] \\
        & E & 
    \end{tikzcd}\,, \]
    where $\pi$ is of type (3) and $E$ is a cycle of $n$ $\PP^1$'s.
\end{enumerate}

\end{proposition}
\begin{proof}
We have already shown how to construct the data $(E,\delta,\Delta',\Delta'')$ given the tower above (Section \ref{Sec: Preliminaries boundary of EEgt} and Proposition \ref{Prop: definition of delta}). To go in the other direction, let 
\[ C\coloneqq \Spec(\BO_E\oplus \delta^{-1}) \,, \]
where $\BO_E\oplus \delta^{-1}$ has an $\BO_E$-algebra structure induced by
\begin{align*}
(\BO_E\oplus \delta^{-1})\times (\BO_E\oplus \delta^{-1}) &\to \BO_E\oplus \delta^{-1} \\
(a,b), (a',b') &\mapsto (aa'+\phi(bb'),ab'+a'b) \,,
\end{align*}
where $\phi: \delta^{-2}=\BO_E(\Delta'+\Delta'') \hookrightarrow \BO_E$. Let $E_1,\dots,E_n$ be the irreducible components of the normalization of $E$. Suppose $E_i$ meets $E_{i-1}$ at $Q^0_i$ and $E_{i+1}$ at $Q^\infty_i$. Let $N_i=\pi^{-1}(E_i)$. Let $P_i^0,P_i^{0'}$ (resp. $P_i^\infty,P_i^{\infty'}$) be the two points in $N_i$ sitting above $Q_i^0$ (resp. $Q_i^\infty$). We assume that $C$ is obtained by identifying $P_i^0$ with $P_{i-1}^\infty$ (resp. $P_i^{0'}$ with $P_{i-1}^{\infty'}$) for $i\in \ZZ/n\ZZ$. There is a commutative diagram
\[
    \begin{tikzcd}
        C\arrow[d,"p"'] & N_i \arrow[l,"\beta_i"'] \arrow[d,"p_i"] \\
        E & E_i=\PP^1 \arrow[l]
    \end{tikzcd}\,.
\]
 Let $\pi_i:\tilde{N}_i \to N_i$ be the double cover associated to the line bundle $p_i^\ast \BO_{\PP^1}(1)$ ramified at $P_i^0+P_i^{0'}+P_i^\infty+P_i^{\infty'}$. Then $\tilde{C}$ is the nodal curve with irreducible components $\tilde{N}_1,\dots,\tilde{N}_n$ obtained by identifying $\pi_i^{-1}(P_i^0)$ with $\pi_{i-1}^{-1}(P_{i-1}^\infty)$ (resp. $\pi_i^{-1}(P_i^{0'})$ with $\pi_{i-1}^{-1}(P_{i-1}^{\infty'})$) for $i\in \ZZ/n\ZZ$.
\end{proof}

When $\Delta'=0$, we write $(E,\delta,\Delta)$ for the data in the proposition above. It is clear from the above proofs that the ordering of $(d_1,\dots,d_n)$ does not play a role. In other words, if $\ud$ and $\ue$ define the same partition of $g$, we have $\SE_\ud=\SE_\ue$. We thus make the following definition:
\begin{definition}\label{Def: SE ud}
Let $g\geq 0$ and $\ud=(d_1,\dots,d_n)\in \mathscr{P}_g$ be a partition of $g$. We denote by $\SE_\ud$ the set of Pryms in ${\EE}_{g,0}$ corresponding to the datum $(E,\delta,\Delta)$ where 
\begin{itemize}
    \item $E$ is the cycle of $n$ $\PP^1$'s.
    \item $\underline{\deg} \delta=\ud$.
\end{itemize}
\end{definition}
We also make the following definition:
\begin{definition}\label{Def of Egt prime}
    Let $\EE'_{g,t}\subset {\EE}_{g,t}$ be the set of $(P,\Xi)=\mathrm{Prym}(\tilde{C}/C)\in {\EE}_{g,t}$ corresponding to a tower of nodal curves 
   \[   \begin{tikzcd}
        & \tilde{C} \arrow[dl,"\pi"'] \arrow[d,"\pi'"] \arrow[dr,"\pi''"] & \\
        C \arrow[dr,"p"'] & C'\arrow[d,"p'"] & C'' \arrow[dl,"p''"] \\
        & E & 
    \end{tikzcd}\,, \]
with $E$ a smooth elliptic curve and $\pi$ of type $(\ast)$. In the case $t=2$ we require additionally that $C'$ is non-hyperelliptic. We denote by $\EE_{g,2}^\mathrm{h}\subset \EE_{g,2}$ the case where $t=2$, $\pi$ is of type $(\ast)$ and $C'$ is hyperelliptic. \\
We thus have $\EE_{g,t}^\ast\subset \EE'_{g,t}$ but $\EE'_{g,t}$ contains additionally the cases where $C$ and $\tilde{C}$ are nodal but $p$ is of type (2). This corresponds to degenerations where a point of $\Delta'$ converges to a point of $\Delta''$.
\end{definition}
We have the following converse to Lemma \ref{Lemma: E cannot have tails}:
\begin{lemma}\label{Lem: EEgt and Sd not Jacobians}
    Let $g\geq 4$, $0\leq t \leq g/2$, and $(P,\Xi)\in \EE'_{g,t}$, or $(P,\Xi)\in \EE_{g,2}^{\mathrm{h}}$, or $(P,\Xi)\in \SE_\ud$ with $\deg \ud= g$. Then $(P,\Xi)$ is neither a Jacobian nor decomposable.
\end{lemma}
\begin{proof}
 By assumption $\pi$ is of type $(\ast)$ thus by \ref{Prop: pi of type (3) and p of type (1,2) etc}, $p$ is of type (1,2). Thus by \cite{Shokurov1984PrymVarieties}, if $(P,\Xi)$ were a Jacobian, $C$ would be either hyperelliptic, trigonal, quasi-trigonal or a plane quintic. The quasi-trigonal case is impossible because then $E$ is a rational curve with two points identified and $p$ has to be of type (3) above the singular point. The other cases are impossible by \ref{Prop: covering pullback line bundle sequence}: If $L=p^\ast M(D)$ with $M\in \Pic^k(E)$ and $D$ $p$-simple, then
    \[ \hh^0(C,L)\leq \hh^0(E,M)+\hh^0(E,M(p_\ast D) \otimes \delta^{-1}) \,, \]
    where $\deg \delta=g$. Finally, $(P,\Xi)$ cannot be decomposable by Theorem \ref{Thm: degree Gauss Map on Egt, in general} and \ref{Thm: Degree Egt with E cycle of P1's}, since the Gauss degree would be zero in these cases (we do not use Lemma \ref{Lem: EEgt and Sd not Jacobians} in the proof of either these theorems).
\end{proof}
The following lemma sums up the results of this section:
\begin{lemma}\label{Lemma: decomposition of Egt}
For $g\geq 4$ and $0\leq t\leq g/2$, we have
\[ \EE_{g,t}\setminus (\mathcal{J}_g\cup \mathcal{A}_{g}^\mathrm{dec})\cap \EE_{g,t}= \EE'_{g,t}\cup \left( \EE_{g,2}^\hh \right)_{\text{if $t=2$}} \cup  \bigcup_{ \ud \leq (t,g-t)} \SE_\ud \,, \]
where $\ud\leq \ue$ if and only if $\ud$ is a subdivision of the partition $\ue$.
\end{lemma}
\begin{proof}
    By Lemmas \ref{Lemma: E cannot have tails} and \ref{Lemma: Singular Pryms all come from EE g,0}, if $(P,\Xi)\in {\EE}_{g,t}$ is not a Jacobian or decomposable, then either $E$ is a cycle of $\PP^1$'s, and then we are in $\SE_\ud$ for some $\ud$, or $E$ is smooth and irreducible, and in that case the only singularities of $C$ come from singularities of type (2) for $p$. This corresponds to limits in $\EE_{g,t}^\ast$ where a branch point $Q'\in \Delta'$ and a branch point $Q''\in \Delta''$ converge together. The converse is given by Lemma \ref{Lem: EEgt and Sd not Jacobians}.  \\
    Finally, for $\ud,\ue\in \mathscr{P}_g$ we have
    \[ \SE_\ud \subset \overline{\SE}_{\underline{e}} \iff \ud \leq \underline{e} \,. \]
    Indeed, the singular elliptic curve can degenerate to have more components and the ramification points can move to the new component. It is also clear from the proof of \ref{Lemma: Singular Pryms all come from EE g,0} that the ordering of the coefficients in $\ud$ do not play a role.  \par 
\end{proof}

\subsection{The Gauss degree on the boundary of \texorpdfstring{$\EE_{g,t}^\ast$}{Egt}}\label{Sec: The Gauss degree on the boundary}
\subsubsection*{The Gauss degree on $\EE'_{g,t}$}
In this case, the computation of the Gauss degree is immediate:
\begin{proposition}\label{Prop: Gauss degree on Egt'}
    Let $g\geq 4$, $0\leq t\leq g/2$ and $(P,\Xi)\in \EE'_{g,t}$ or $(P,\Xi)\in \EE_{g,t}^\hh$. Then the degree of the Gauss map $\GG:\Xi \dashrightarrow \PP^{g-1}$ is the same as the one given by Theorem \ref{Thm: degree Gauss Map on Egt, in general}.
\end{proposition}
\begin{proof}
    By definition, we have a tower
   \[   \begin{tikzcd}
        & \tilde{C} \arrow[dl,"\pi"'] \arrow[d,"\pi'"] \arrow[dr,"\pi''"] & \\
        C \arrow[dr,"p"'] & C'\arrow[d,"p'"] & C'' \arrow[dl,"p''"] \\
        & E & 
    \end{tikzcd}\,, \]
where $E$ is smooth. By \ref{Prop: pi of type (3) and p of type (1,2) etc}, $p'$ and $p''$ are of type (3), thus $C'$ and $C''$ are smooth. By \ref{Prop: g is of degree 2, singular case}, there is a map of generic degree $2$:
\[ (\pi'^\ast+\pi''^\ast): \tilde{\Xi} \to \Xi \,.\]
The proof of Theorem \ref{Thm: degree Gauss Map on Egt, in general} to compute the Gauss degree on $\tilde{\Xi}$ then follows through ad verbatim, as $C'$ and $C''$ are smooth bielliptic curves.
\end{proof}

\subsubsection*{The Gauss degree on $\SE_\ud$}
For the remainder of the section, we will focus on the case where $E$ degenerates to a cycle of $\PP^1$'s. Let $\ud=(d_1,\dots,d_n)\in \mathscr{P}_g$ be a partition of $g$ in its reduced form (we assume $d_i>0$ for $1\leq i \leq n$). We make the following definition:
\begin{equation}\label{Equ: Boundary Bielliptic Pryms, definition of mu}
 \begin{aligned}
        \mu_\ud &\coloneqq b_{\ud-(1^n)}\left\lbrace (1-4x)^{-3/2} \prod_{i=1}^n \left(1-2x/d_i \right) \right\rbrace_{x^{n-1}} \\
        &=\frac{1}{2} b_{\ud-(1^n)}\sum_{\epsilon\in \{0,1\}^{n} } (n-l(\epsilon)) b_{n-l(\epsilon)} \prod_{i=1}^n (-2/d_i)^{\epsilon_i}  \,,
    \end{aligned}
\end{equation}
where $l(\epsilon)\coloneqq \#\{i\,|\, \epsilon_i\neq 0\}$ is the length of the partition $\epsilon$, $(1^n)=(1,\dots,1)$, $b_k=\binom{2k}{k}$ is the middle binomial coefficient and
\[ b_\ud=b_{d_1}\cdots b_{d_n} \,. \] 
See Section \ref{Sec: App: Computation of the coefficient mu} for alternate forms of $\mu_\ud$. We have the following:
\begin{theorem}\label{Thm: Degree Egt with E cycle of P1's}
Let $g\geq 1$, $\ud=(d_1,\dots,d_n)\in \mathscr{P}_g$ and $(P,\Xi)\in \SE_\ud$. The degree of the Gauss map $\GG:\Xi\dashrightarrow \PP^{g-1}$ is given by
\[ \deg(\GG)= \mu_\ud-2\ns(\tilde{C}/C) \,. \]
\end{theorem}
\begin{remark}
    We will show in a subsequent paper that this result can be recovered from a more general calculation of Chern-Mather class for cyclic curves \cite{podelski2023boundary}.
\end{remark}
The proof of Theorem \ref{Thm: degree Gauss Map on Egt, in general} adapts to our case, so we will mirror it and point out where additional details are required.
\par 
\par  \vspace{.5cm} \textbf{Step one: Setting the scene.}
Let $(P,\Xi)\Prym(\tilde{C}/C)\in \SE_\ud$ correspond to the datum $(E,\delta,\Delta)$ with
\[ \underline{\deg}(\delta)=\ud=(d_1,\dots,d_n) \,. \]
We keep the notations of the previous section. Let $E_1,\dots,E_n$ be the irreducible components of the normalization of $E$. We fix an identification $E_i=\PP^1$ such that $E_i$ meets $E_{i-1}$ at $0$ and $E_{i+1}$ at $\infty$ for $i\in \ZZ/n\ZZ$. Let $\beta: N\to C''$ be the normalization. For $0\leq i \leq n$ let $N_i$ be the component of $N$ above $E_i$, and $P_i^0,P_i^\infty$ be the points in $N_i$ sitting above $0,\infty\in E_i$ respectively. The following diagram is commutative
\begin{center}
    \begin{tikzcd}
       N_i \arrow[rr,bend left=30,"\beta_i"] \arrow[r,hook] \arrow[d,"p_i"'] & N \arrow[r,"\beta"'] \arrow[dr,"p_N"'] & C''\arrow[d,"p''"] \\
        E_i=\PP^1 \arrow[rr]  & & E \,.
    \end{tikzcd}
\end{center}
We will use the notation
\[ N_\ud \coloneqq N_{1,d_1}\times\cdots N_{n,d_n} \,, \]
where $N_{i,d_i}$ is the $d_i$-symmetric product of $N_i$. By abuse of notation, we write $\Nm: N_{\underline{d}}=N_{1,d_1}\times\cdots\times N_{n,d_n}\dashrightarrow JE$ for the sum of the norm maps associated to the composition $p''\circ \beta_i:N_i\to E$, which is defined away from divisors containing a point $P^0_i$ or $P^\infty_i$. We define the multiplication map (defined away from divisors containing both $0$ and $\infty$) by
\begin{align*}
     m_i:\PP^1_{d_i} &\dashrightarrow \PP^1\\
     P_1+\cdots + P_n &\mapsto P_1 \cdots P_n \,,
\end{align*}
and 
\begin{align*}
    m:\PP^1_{\underline{d}}& \dashrightarrow \PP^1 \\
    (D_1,\dots,D_n) &\mapsto m_{1}(D_1)\cdots m_{n}(D_n) \,,
\end{align*}
where $\PP^1_{\underline{d}}$ is defined in the same way as $N_\ud$. The following diagram is commutative
\begin{center}
\begin{tikzcd}
 {N_\ud} \arrow[rr,dashed, "\alpha" ] \arrow[d, "{p_{N\ast}}"' ] & &{\Pic^\ud(C'')} \arrow[d,"{\Nm}"]  \\
    {\PP^1_\ud} \arrow[r,dashed,"m"'] & {\PP^1} & \CC^\ast={\Pic^\ud(E)} \arrow[l,phantom,"\supset"]
\end{tikzcd}\,,
\end{center}
where $\alpha:N_\ud \dashrightarrow \Pic^\ud(C'')$ is the Abel-Jacobi map, defined away from singular divisors. Let
\[ W \coloneqq \overline{\Nm^{-1}(\delta) }\subset N_{\underline{d}}\,. \]
We now compute the cohomology class of $W$. Notice that $\PP^1_{d_i}=\PP \HH^0(\PP^1,\BO(d_i))=\PP^{d_i}$ and that the canonical coordinates are given by
\[ [a^i_0:\cdots: a^i_{d_i}]= V( a^i_0 u^{d_i} + a^i_1 u^{d_i-1} v + \cdots + a^i_{d_i} u^{d_i} ) \subset \PP^1 \,, \]
Where $[u:v]$ are the canonical coordinates on $\PP^1$. With these coordinates, the map $m$ is
\[ m:(a^1,\dots,a^n)\mapsto (a^1_0 \cdots a^n_0: a^1_{d_1}\cdots a^n_{d_n}) \,. \]
Thus 
\[ [\overline{m^{-1}\{pt\}}]=c_1(\BO_{\PP^1_{\underline{d}}}(1,\dots,1))=h_1+\cdots+h_n \in \HH^1(\PP^1_{\underline{d}}) \,, \]
where $h_i$ is the hyperplane class in $\PP^1_{d_i}$, pulled back to $\PP^1_{\underline{d}}$ by projection. Thus
\begin{equation}\label{Equ: Class of W in N underline d}
[W]=(p_1\times\cdots \times p_n)^\ast (h_1+\cdots+h_n)=2(x_1+\cdots+x_n)\in \HH^1( N_{\underline{d}} )\,.
\end{equation}
where $x_i$ is the pullback of the class $[N_{i,d_i-1}]\in \HH^1(N_{i,d_i})$. \par 
\par  \vspace{.5cm} \textbf{Step two: A configuration of Gauss maps.}
Just as in the smooth case, we have the following commutative diagram, which we explain below
\begin{equation}\label{Tikzcd: Gauss map configuration in the Egt singular case}
    \begin{tikzcd}
        \delta \arrow[d,hook] & W\arrow[l,dashed] \arrow[r,dashed] \arrow[d,hook] \arrow[rrr,dashed,bend left=20,"\tilde{\GG}"] & \tilde{\Xi} \arrow[r,"\pi''^\ast"'] \arrow[d,hook]& \Xi \arrow[r,dashed, "\GG"'] & \PP( \HH^0(N,\omega_{N})) \\
        JE & N_{\underline{d}} \arrow[rrr,dashed,bend right=20,"\tilde{\GG''}"'] \arrow[l,dashed,"\Nm"] \arrow[r,dashed,"\alpha"] & \Theta'' \arrow[rr,dashed,"\GG''"] & & \PP(\HH^0(C'',\omega_{C''})) \arrow[u,dashed,"\mathcal{F}"']
    \end{tikzcd}\,.
\end{equation}
Recall that sections of $\omega_{C''}$ are $1$-forms $\omega$ on $N$ which can have poles at $P^0_i$ and $P^\infty_i$, subjected to the conditions
\[ \Res_{P^\infty_i}\omega + \Res_{P^0_{i+1}} \omega = 0 \,, \quad \text{for} \quad i\in \ZZ/n\ZZ \,. \]
We thus have an inclusion of $\BO_{C''}$-modules
\[ \beta_\ast \omega_N \subset \omega_{C''} \subset \beta_\ast \omega_N(\sum_i P^0_i+P^\infty_i) \,, \]
thus
\[ \HH^0(N,\omega_N) \subset \HH^0(C'',\omega_{C''}) \subset \HH^0(N,\omega_N(\sum_i P_i^0+P_i^\infty)) \,. \]
Let $s_E$ be a generator of $\HH^0(E,\omega_E)$. As a $1$-form, $s_E$ is given on $E_i$ by $dz/z$ for $z$ the usual coordinate on $\PP^1=E_i$ centered at $0$. Let $s_{R''}=p''^\ast s_E$. As $p_i:N_i\to E_i$ is ramified at $R_i+P_i^0+P_i^\infty$ we have $\divv (s_{R''})=R''$ as a section of $\omega_{C''}$. For dimension reasons we have
\[ \HH^0(C'',\omega_{C''})\simeq \HH^0(N,\omega_N)\oplus \langle s_{R''} \rangle \,. \]
We see from the above discussion that $\HH^0(N,\omega_N)$ (resp. $\langle s_{R''}\rangle$) is the -1 (resp. +1) eigenspace for the action of $\tau''$ on $\HH^0(C'',\omega_{C''})$. We can thus identify
\[ T^\vee P''= T^\vee JN = \HH^0(N,\omega_N) \,. \]
We define $\mathcal{F}:\PP \HH^0(C'',\omega_{C''})\dashrightarrow \PP \HH^0(N,\omega_N)$ to be the projection with center $[s_{R''}]$. The diagram (\ref{Tikzcd: Gauss map configuration in the Egt singular case}) is commutative for the same reasons as in the smooth case. \par
\vspace{.5cm} \textbf{Step three: A computation in cohomology.} Let $M=[s]\in |\omega_N|$ be general, let $s_i=s\restr{N_i}$ and $M_i=\divv s_i$. Let $\tilde{s}\in \HH^0(C'',\omega_{C''})$ be the corresponding section of the dualizing bundle on $C''$. Let
\[ \tilde{M}\coloneqq \divv \tilde{s} = \sum_{i=1}^n \beta_{i,\ast} M_i + (\lambda_i,1,1)_{Q_i} \in |\omega_{C''}|\]
for some coefficients $\lambda_i\in \CC^\ast$ which depend on the coefficients of $s$ in the decomposition $\HH^0(N,\omega_N)=\oplus_i \HH^0(N_i,\omega_{N_1})$. Let 
\begin{align*}
    V&\coloneqq \overline{\mathcal{F}^{-1}(M)}=\langle \tilde{M} , R'' \rangle \subset |\omega_{C''}|\,, \\
    \text{and} \quad Z & \coloneqq \overline{\tilde{\GG}''^{-1}(V)}=\{ D\in N_{\underline{d}} \,|\, D \leq \beta^\ast F \,, \text{for } F\in V\} \subset  N_{\underline{d}}\,. 
\end{align*}
Let $E_{K,i}$ be the evaluation bundle on $N_{i,d_i}$ corresponding to the line bundle $\omega_{N_i}(P_i^0+P_i^\infty)$ (see Section \ref{Sec: The evaluation bundle}). Let 
\[ \mathrm{ev}_i: \HH^0(N_i,\omega_{N_i}(P_i^0+P_i^\infty))\otimes_\CC \BO_{N_{i,d_i}} \to E_{K,i} \]
be the corresponding evaluation maps. The locus $Z$ is then the $1$-th determinantal variety associated to the morphism of vector bundles
\[ \langle \tilde{s} ,s_{R''} \rangle \otimes_\CC \BO_{N_{\underline{d}}} \overset{\oplus_i\mathrm{Res}_{N_i}}{\longrightarrow} \bigoplus_{i=1}^n \HH^0(N_i,\beta_i^\ast \omega_{C''}) \otimes_\CC \BO_{N_{\underline{d}}} \overset{\oplus_i \mathrm{ev}_i }{\longrightarrow} \bigoplus_{i=1}^n \mathrm{pr}_i^\ast E_{K,i} \,, \]
where $\mathrm{pr}_i : N_{\underline{d}}\to N_{i,d_i}$. Z is of (the expected) dimension $1$ since for any divisor $F\in V$, there are finitely many $D\leq \beta^\ast F$. In what follows we denote by $\HH_\ast(-)=\HH_{2\ast}(-,\QQ)$ and $\HH^\ast(-)=\HH^{2\ast}(-,\QQ)$ the even-dimensional homology and cohomology with $\QQ$-coefficients. By intersection theory we have
\[ [Z]=c_{g-1}\left(\bigoplus_{i=1}^n \mathrm{pr}_i^\ast E_{K,i} \right)\cap [N_\ud] \in \HH_1(N_\ud) \,. \]
By \ref{Cor: pushforward chern class of evaluation bundle} and \ref{Equ: Class of W in N underline d}, we thus have in $\HH_\ast(\Nd)$
\begin{align*}
    [W]\cdot [Z] &=2(x_1+\cdots+x_n)\cdot c_{g-1}(\oplus_i \mathrm{pr}_i^\ast E_{K,i}) \\
    &=2\sum_{i=1}^n x_i \cap c_{d_i-1}(E_{K_i}) \prod_{j\neq 1} c_{d_i}(E_{K_i}) \\
    &=2 \sum_{i=1}^n \binom{2d_i-1}{d_i}\prod_{j\neq i} \binom{2d_j}{d_j} \\
    &=n b_\ud\,,
\end{align*}
where $b_\ud \coloneqq b_{d_1}\cdots b_{d_n}$.
\par  \vspace{.5cm} \textbf{Step four: Ruling out special points.}
To complete the computation of the degree of $\GG$, we need to rule out special points, which is done by the following lemma:
\begin{lemma}\label{Lemma: Points outside the domain of GG in the singular case}
Under the assumptions of Theorem \ref{Thm: Degree Egt with E cycle of P1's}, for general $M\in \PP \HH^0(N,\omega_N)$, the intersection $W\cap Z$ is zero-dimensional and the number of points outside the domain of definition of $\tilde{\GG}:W\dashrightarrow \PP \HH^0(N,\omega_N)$ is
\[ 2\tilde{\mu}(\underline{d})+4\ns(\tilde{C}/C) \,,\]
counted with multiplicity, where
\[ \tilde{\mu}(\underline{d})\coloneqq \sum_{\epsilon^0,\epsilon^\infty\in \{0,1\}^n} \min(|\epsilon^0|,|\epsilon^\infty|)\prod_{i=1}^n \binom{2d_i-\epsilon^0_i-\epsilon^\infty_i}{d_i-\epsilon^0_i-\epsilon^\infty_i} \,, \]
and
\[ \ns(\tilde{C}/C)\coloneqq \frac{1}{2}\#\{ D\leq \Delta\,|\, \underline{\deg}(D)=\underline{d} \,, \text{and } m(D)=\delta \}\,, \]
where $m:\PP^1_{\underline{d}}\dashrightarrow \PP^1$ is the multiplication map and $\delta$ is identified with a point in $\PP^1$.
\end{lemma}
With the above Lemma, Theorem \ref{Thm: Degree Egt with E cycle of P1's} follows at once: By Lemma \ref{Lem: App: Computation of coefficient mu} below we have
\[ \tilde{\mu}(\ud)=\frac{n}{2}b_\ud-\mu_\ud=\frac{1}{2}[W]\cdot [Z]-\mu_\ud \,. \]
The map $g\restr{\tilde{\Xi}}:\tilde{\Xi}\to \Xi$ is of degree $2$ since it is the restriction of the isogeny $P''\to P$. $W\to \tilde{\Xi}$ is birational (a general line bundle in $\tilde{\Xi}$ has a unique non-zero section, vanishing away from the singular locus). We thus have
\[ \deg(\GG)=\frac{1}{2}( [W]\cdot [Z]- 2\tilde{\mu}(\underline{d})-4\ns(\tilde{C}/C)))=\mu_\ud-2\ns(\tilde{C}/C) \,. \]

\begin{proof}[Proof of Lemma \ref{Lemma: Points outside the domain of GG in the singular case}]
We need to check which points of $W\cap Z$ lie in the locus of indeterminacy of the following composition
\begin{center}
    \begin{tikzcd}
    W \arrow[r,hook] \arrow[rrrr,dashed, bend left=15 ,"\tilde{\GG}"] & N_{\underline{d}} \arrow[r,dashed,"\alpha"'] & \Theta'' \arrow[r,dashed, "{\GG''}"'] & {|\omega_{C''}|} \arrow[r,dashed,"\mathcal{F}"'] & {|\omega_N|}
     
    \end{tikzcd}
\end{center}
First notice that $\mathcal{F}:V\dashrightarrow M$ is undefined at $\{R''\}$. Every $D\leq R''$ of multidegree $\underline{d}$ such that $\Nm(D)=\delta$, corresponds to a point in $W\cap Z$ lying above $R''$. We will show at the end of the proof that the intersection multiplicity of $W\cap V$ at these points is $2$. $\GG'':\Theta''\to |\omega_{C''}|$ is undefined at $W^1_{g}(C'')$, which is of codimension $3$ inside $\Theta''$ and the image of $\Lambda_{\Theta''}\restr{W^1_{g}}\to |\omega_{C''}|$ is of codimension $2$. Thus a general line $V\subset |\omega_{C''}|$ will not meet this locus, except eventually at $R''\in|\omega_{C''}|$, but this point was already taken care of above. Finally, $\alpha:N_{\underline{d}}\dashrightarrow \Theta''$ is undefined exactly at divisors $D$ containing $P_i^0$ or $P_i^\infty$ for some $i$. At these points,
\[ N_{\underline{d}}\overset{\Nm_{p_i}}{\dashrightarrow} \PP^1_{\underline{d}} \overset{m}{\dashrightarrow} \PP^1 \dashrightarrow JE \]
is not well-defined either. For $\lambda\in \CC^\ast$, a point is in the closure $\overline{m^{-1}(\lambda)}$ if it contains $Q^0_i$ and $Q^\infty_j$ for some $i,j$, but not if it contains only $Q^0_i$'s or only $Q^\infty_i$'s. Since $s_{R''}$ is non-zero at the singular points, and $\tilde{s}$ vanishes at the singular points, we see that the only divisor $F\in V$ containing singular points is $\tilde{M}=\divv \tilde{s}=\sum_i \beta_{i,\ast} M_i+(\lambda_i,1,1)_{Q_i}$. The points in $W\cap Z$ sitting above $\tilde{M}$ are thus (for $M$ general)
\[ \{ D\leq \beta^\ast \tilde{M} \, |\, P^0_i+P^\infty_j \leq D \,, \text{for some }i,j \} \,. \]
We have to show that the multiplicity of these points in the intersection $W\cap Z$ is exactly
\[ 2\min( \# \text{ of }P^0_i\text{'s}\,,\, \#\text{ of }P^\infty_i\text{'s}) \,. \]
We will compute this multiplicity explicitly. Suppose
\[ D=\sum_{i\in I} P^0_i+\sum_{i\in J} P^\infty_i + P^1+\cdots+P^l \,. \]
Let $(z:w)\in\PP^1$ be the usual coordinates. Locally near $P^j$, $p_i:N_i\to \PP^1$ is étale, thus $(z:1)$ lifts to a coordinate near $P^j$ which we call $z_j$. At $P^0_i$ and $P^\infty_i$, $p_i$ is ramified, thus there is a coordinate $z_{0,i}$ (resp. $w_{\infty,i}$) near $P^0_i$ (resp. $P^\infty_i$) such that $p_i(z_{0,i})=(z_{0,i}^2:1)$ (resp. $p_i(w_{\infty,i})=(1:w_{\infty,i}^2)$). Let $(\lambda,1)\in \PP^1$ be the preimage of $\delta$ under the map $\PP^1\dashrightarrow JE$. Then $W$ is defined near $D$ by
\begin{equation}\label{Equ: coordinate definition of W}
     \prod_{i\in I} z_{0,i}^2 \prod_{i=1}^l z_i- \lambda \prod_{i\in J } w_{\infty,i}^2 =0 \,.
\end{equation} 
$z_i$ is non-zero at $P^i$, thus $W$ has a singularity of degree $2\min(|I|,|J|)$ at $D$. To prove the Lemma, we have to show that the intersection with $Z$ is ``transverse", in other words, that the tangent line to $Z$ is not contained in the tangent cone of $W$ at $D$. For simplicity, we show this for $n=1$ (the proof is the same for higher $n$, the notation just becomes a bit more tedious). Recall $s\in \HH^0(N,\omega_N)=p_N^\ast \HH^0(\PP^1,\BO_{\PP^1}(g-1))$, thus there is $F=Q^1+\cdots+Q^{g-1}\in \PP^1_{g-1}$ such that $\divv s=p_N^\ast F$. Consider $\tilde{s}$ and $s_{R''}$ as sections of $\omega_N(P^0+P^\infty)$ on $N$. Recall $\divv \tilde{s}=P^0+P^\infty +\divv s $. Suppose $D=P^0+P^\infty+P^1+\cdots+P^{g-2}$. Since $\tilde{s}$ has simple zeroes and $s_{R''}$ is non-zero at the points of $D$, for $t\in \CC$ small enough, there are holomorphic functions $P^i(t):(\CC,0)\to (N,P^i)$ such that
\[ \tilde{s}(P^i(t))-ts_{R''}(P^i(t))=0\,, \quad \text{for $i\in \{\infty,0,1,\dots,2g-2\}.$} \]
Then locally $t\mapsto D(t)=P^0(t)+P^\infty(t)+P^1(t)+\cdots+P^{g-2}(t)$ is a parametrization of $Z$ near $D$. Viewing $P^i(t)$ as having value in $\CC$ using the coordinates $z_i$, $z_{0,i}$ and $w_{\infty,i}$ respectively, and plugging into \ref{Equ: coordinate definition of W}, the multiplicity of $Z\cap W$ is the degree of vanishing at $t=0$ of
\begin{equation}\label{Equ: function degree of vanishing Z cap W}
     (P^0(t))^2 P^1(t) \cdots P^{d-2}(t)-\lambda (P^\infty(t))^2 \,.
\end{equation}
We have
\begin{align*} 
(\tilde{s}-ts_{R''})\cdot \tau''^\ast(\tilde{s}-ts_{R''})&=(\tilde{s}-t s_{R''})\cdot (\tilde{s}+ts_{R''}) \\
&= \tilde{s}^2-t^2 s_{R''}^2 \\
&= p_N^\ast ( f-t^2 h)\,,
\end{align*}
where
\[ f=zw\prod_{i=1}^{g-1}(z-Q^i w)^2 \,, \quad \text{and} \quad h=\prod_{i=1}^{2g} (z-Q_i w)\in \HH^0(\PP^1,\BO(2g)) \,. \]
$Q^0(t)$ and $Q^\infty(t)$ are the zeroes of $f-t^2 h$ converging to $0$ and $\infty$ respectively as $t$ goes to $0$. Making a change of variables $t^2=s$, we have
\[ \frac{\dv Q^0}{\dv s}\restr{0}= \frac{1}{(f/h)'\restr{(0,1)}}=\frac{\prod_{i=1}^{2g} Q_i}{\prod_{i=1}^{g-1} (Q^i)^2} \,, \]
and
\[ \frac{\dv Q^\infty}{\dv s}\restr{0}= \frac{1}{(f/h)'\restr{(1,0)}}=1 \,.\] 
Recall that $Q_1\cdots Q_{2g}=\lambda^2$, thus the leading coefficient in \ref{Equ: function degree of vanishing Z cap W} is $t^2$ times
\[ \frac{ \lambda^2 \prod_{i=1}^{g-2} p_N(P^i)}{\prod_{i=1}^{g-1} (Q^i)^2}- \lambda \,. \]
The $p_N(P^i)$ are among the $Q^i$'s, but since the $Q^i$'s are general in $\PP^1$, the above is non-zero.\par 
Finally, we compute the intersection multiplicity of $Z\cap W$ at a point $D\leq R''$ such that $\Nm(D)=\delta$. As before, we first assume $n=1$. Suppose $D=P_1+\cdots+P_g$. As before, let $P_i(t)$ be the zeroes of $t\tilde{s}-s_{R''}$, and $Q_i(t)=p_N(P_i(t)$. Then $Q_i(t)$ are the zeroes of $t^2 f-h$ converging to $Q_i$ as $t$ goes to zero. Setting $s=t^2$ we have
\[ \frac{\dv Q_i}{\dv s}\restr{0}= \frac{1}{(h/f)'\restr{(Q_i,1)}}=\frac{Q_i \prod_{j=1}^{g-1}(Q_i-Q^j)^2}{\prod_{j\neq i } (Q_i-Q_j)} \,. \]
Thus the coefficient in front of $Q_1(t)\cdots Q_g(t)-\lambda$ is $t^2$ times
\[ \sum_{i=1}^g Q'_i(0) \prod_{j\neq i} Q_j = \sum_{i=1}^g \lambda \frac{\prod_{j=1}^{g-1}(Q_i-Q^j)^2}{\prod_{j\neq i}{Q_i-Q_j}}\,, \,. \]
which is non-zero as the $Q^j$'s are general. In the case $n>1$, recall that $s=\sum_i s_i \in \oplus_i \HH^0(N_i,\omega_{N_i})$. Replacing $s$ with $\sum_i \mu_i s_i$ for general coefficients $\mu_i$, we change the values of $Q'_{i,j}(0)$, thus the coefficient in front of $t^2$ is non-zero. This finishes the proof of the Lemma.
\end{proof}

\subsubsection{Computation of the coefficient \texorpdfstring{$\mu$}{mu}}\label{Sec: App: Computation of the coefficient mu}
For $\alpha\in \CC$ and $k\in \NN_{>0}$, we make the following convention
\[ \binom{\alpha}{0}=1\,, \quad \binom{\alpha}{k}=\frac{\alpha(\alpha-1)\cdots (\alpha-k+1)}{k!}\,, \quad \binom{\alpha}{-k}=0\,.  \]
For all $\alpha\in \mathbb{R}$, we have
\begin{equation}
    (1+x)^\alpha=\sum_{k\geq 0} \binom{\alpha}{k} x^k \,. 
\end{equation}
From this above, we can deduce the generating function for the middle binomial coefficients $b_n\coloneqq \binom{2n}{n}$
    \begin{align}
        &(1-4x)^{-1/2}=\sum_{n\geq 0} b_n x^n \,,
    \end{align}
    and its derivative
    \begin{align}
        &2(1-4x)^{-3/2}=\sum_{n\geq 0} (n+1)b_{n+1} x^{n}\,. \label{Equ: generating series of bn derived}
    \end{align}
For $\ud=(d_1,\dots,d_n)\in \NN^n$ we denote by $l(\ud)=\#\{i\,|\, d_i\neq 0\}$. We define
\[ \ud-\underline{1}\coloneqq \big(\max(d_i-1,0)\big)_{1\leq i \leq n} \,. \] 
We now make the following computation:
\begin{lemma}\label{Lem: App: Computation of coefficient mu}
    Let $\ud=(d_1,\dots,d_n)\in (\NN_{>0})^n$, then
    \begin{equation}\label{Equ: App: Lem: Computation coeff mu}
    \sum_{\epsilon^0,\epsilon^\infty\in \{0,1\}^n} \min(|\epsilon^0|,|\epsilon^\infty|)\prod_{i=1}^n \binom{2d_i-2}{d_i-\epsilon^0_i-\epsilon^\infty_i}=\frac{n}{2}  b_\ud  -\mu_\ud \,, 
    \end{equation}
    where 
    \begin{align*}
        \mu_\ud &\coloneqq b_{\ud-\underline{1}}\left\lbrace (1-4x)^{-3/2} \prod_{i\,|\,d_i \neq 0} \left(1-2x/d_i \right) \right\rbrace_{x^{l(\ud)-1}} \\
        &=\frac{1}{2} b_{\ud-\underline{1}}\sum_{\epsilon\in \{0,1\}^{l(\ud)} } (l(\ud)-l(\epsilon)) b_{l(\ud)-l(\epsilon)} \prod_{i\,|\,d_i\neq 0} (-2/d_i)^{\epsilon_i}\\
        &= \sum_{\epsilon \in \{0,1\}^n} l(\epsilon) 2^{l(\ud)-l(\epsilon)} \binom{l(\epsilon)-1}{\lfloor (l(\epsilon)-1)/2 \rfloor} \prod_{i=1}^n \binom{2d_i-2}{d_i-\epsilon_i}\,.
    \end{align*}
\end{lemma}
\begin{remark}
   Note that $\mu_\ud$ only depends on the partition induced by $\ud$.
\end{remark}
\begin{proof} 
Note that if $d_i=0$, the product on the left in (\ref{Equ: App: Lem: Computation coeff mu}) is zero if $\epsilon^0+\epsilon^\infty \neq 0$. After changing $n$ we can thus assume $d_i>0$ for all $i$. 
For all $k\neq 0$ we have
\[ \binom{2k-2}{k}=\binom{2k-2}{k-2} \,, \]
thus the product on the left in the lemma only depends on the value of $\epsilon\coloneqq \overline{\epsilon^0+\epsilon^\infty}\in (\ZZ/2\ZZ)^n$. The computation of the coefficient in front of this term is a straightforward manipulation of binomial coefficients. Suppose $\epsilon=(1,\dots,1,0,\dots,0)$ with $k$ ones followed by $n-k$ zeroes. Then
\begin{align*}
    \sum_{\epsilon^0+\epsilon^\infty = \epsilon } \min(|\epsilon^0|,\epsilon^\infty|)&= \sum_{i=0}^k \sum_{j=0}^{n-k} (j+\min(i,k-i))\binom{k}{i} \binom{n-k}{j} \\
    &= \sum_{i=0}^k (2^{n-k} \min(i,k-i) + (n-k)2^{n-k-1} ) \binom{k}{i} \\
    &=n 2^{n-1} - k 2^{n-k} \binom{k-1}{\lfloor (k-1)/2 \rfloor} \,. 
\end{align*}
For all $k\neq 0$ we have $\binom{2k}{k}=2\binom{2k-1}{k}$, thus
\begin{align*}
n 2^{n-1} \sum_{\epsilon \in \{0,1\}^n} \prod_{i=1}^n \binom{2d_i-2}{d_i-\epsilon_i}
&= n 2^{n-1}\sum_{\epsilon \in \{0,1\}^n} \prod_{i=1}^n \binom{2d_i-2}{d_i-\epsilon_i}\\
    &= n 2^{n-1}\prod_{i=1}^n (\binom{2d_i-2}{d_i}+\binom{2d_i-2}{d_i-1}) \\
    &= n 2^{n-1}\prod_{i=1}^n \binom{2d_i-1}{d_i}= \frac{n}{2} b_\ud \,.
\end{align*}
For all $k\neq 0$ we have $\binom{2k-2}{k}=(1-1/k)\binom{2k-2}{k-1}$, thus
\begin{align*}
    \mu_\ud &= b_{\ud-\underline{1}}\sum_{\epsilon \in \{0,1\}^n} l(\epsilon) 2^{l(\ud)-l(\epsilon)} \binom{l(\epsilon)-1}{\lfloor (l(\epsilon)-1)/2 \rfloor} \prod_{i\,|\, \epsilon_i=0} (1-\frac{1}{d_i}) &\\
    &= b_{\ud-\underline{1}} \sum_{I\subset [\![1,n]\!]} \frac{ (-1)^{l(I)}}{d_I} \sum_{ \substack{ \epsilon \in \{0,1\}^n \\ \forall i\in I\,, \epsilon_i=0} } l(\epsilon) 2^{n-l(\epsilon) } \binom{l(\epsilon)-1}{\lfloor (l(\epsilon)-1)/2 \rfloor} &\\
    &=b_{\ud-\underline{1}} \sum_{I\subset [\![1,n]\!]} \frac{ (-1)^{l(I)}}{d_I} \sum_{ k\geq 0 } k 2^{n-k } \binom{k-1}{\lfloor (k-1)/2 \rfloor} \binom{n-l(I)}{k} &\\
    &=b_{\ud-\underline{1}}\frac{1}{2} \sum_{I\subset [\![1,n]\!] } (n-l(I)) b_{n-l(I)} \prod_{i\in I} -2/d_i   \\
    &= b_{\ud-\underline{1}}\left\lbrace (1-4x)^{-3/2} \prod_{i=1}^n\left(1-2x/d_i \right) \right\rbrace_{x^{n-1}}\,. & \text{(by \ref{Equ: generating series of bn derived})}
\end{align*}
where we use the identity $\sum_{k=0}^n 2^{n-k} \binom{k}{\lfloor k/2 \rfloor} \binom{n}{k}=b_{n+1}/2$. The lemma follows.
\end{proof}

\section{Additional isolated singularities}\label{Sec: additional isolated Singularities}
Let $(P,\Xi)=\mathrm{Prym}(\tilde{C}/C)\in \EE'_{g,t}$. By \ref{Prop: singularities are quadratic of maximal rank}, $\Xi$ has exactly $\tns(\tilde{C}/C)$ isolated singularities. These singularities are quadratic of maximal rank, and located at $2$-torsion points. The same holds for ppav's in $\SE_\ud$ (with $\ns(\tilde{C}/C)$). For each of these singularities, the degree of the Gauss map drops by $2$. In this section, we will investigate what values $\ns(\tilde{C}/C)$ can take. We will first reduce this to a combinatorial question on groups, and then carry out a complete study in the case $g=5$.

\subsection{A combinatorial question on groups}
Let $\ud=(d_1,\dots,d_n)\in \NN^n$, and $g=\deg \ud =d_1+\cdots+d_n$. Let
\[d^0=1 \,, \quad \text{and}\quad  d^k \coloneqq d_1+d_2+\cdots +d_k\,, \quad \text{for $1\leq k \leq n$.}  \]
We say that a set of indices $I\subset[\![1,2g]\!]$ is a \emph{relation} if \[\mathrm{Card}(I\cap [\![2d^{k-1}+1,2d^k]\!])=d_k \quad \text{for $1\leq k \leq n$.} \]
Let $\Rel_\ud$ be the set of relations in $[\![1,2g]\!]$ with respect to $\ud$. Let $G$ be an abelian group and $X=(x_1,\dots,x_{2g})\in G^{2g}$. We define 
\begin{align*} 
\mathcal{I}^G_{\ud}(X)&\coloneqq \{ I\in \Rel_\ud \,|\, \sum_{i\in I} x_i=0 \} \,,\\
\nss_\ud^G(X)&\coloneqq \frac{1}{2}\Card(\mathcal{I}_{\ud,X}) \,, 
\end{align*}
to be the set (resp. the number) of relations verified by $X$. We will omit the sub and supscripts $G$ and $\ud$ when it is clear from the context. We define the diagonal with respect to $\ud $ by
\begin{align*} 
Z_\ud  &\coloneqq \{ X=(x_1,\dots,x_{2g})\in G^{2g}\,|\, x_i=x_j \quad \text{for some} \quad 2d^{k-1}<i<j\leq 2d^k 
\} \\
&\subset G^{2g}\,. 
\end{align*}
We define
\[ \nss^G_{\ud,\mathrm{max}}\coloneqq \max\{ \nss^G_\ud(X)\,|\,X\in G\setminus Z_\ud \,, \quad x_1+\cdots+x_{2g}=0\}\,. \]
\begin{question}\label{Question 2}
Let $\ud\in \NN^n$ with $\deg \ud=g$ and $G$ be an abelian group (for instance $G=(E,+)$ an elliptic curve or $G=(\CC^\ast,\cdot)$ a torus). What is the value of $\nss^G_{\ud,\mathrm{max}}$?
\end{question}

We have the following intermediate question about $\ZZ$-modules.
\begin{question}
Let $\ud\in \NN^n$ with $\deg \ud=g$ and $e_1,\dots,e_{2g}\in \ZZ^{2g}$ be the canonical basis. For $I\in \Rel_\ud$ let
\[ T_I\coloneqq \sum_{i\in I} e_i \,. \]
Let $\bar{1}=\sum_{i=1}^{2g} e_i$. What is the biggest set of relations $\mathcal{I}\subset \Rel_\ud$ such that the $\ZZ$-module
\[ M_\mathcal{I} \coloneqq \bar{1}+\langle T_I \rangle_{I\in \mathcal{I}}\subset \ZZ^{2g} \]
does not contain $e_i-e_j$ for all $2d^{k-1}<i<j\leq 2d^{k}$, for all $1\leq k \leq n$.
\end{question}
The answer to the second question must be higher to the first question since if $X$ is a solution to the first question, then $\mathcal{I}_X$ must work for the second question. Some remarks are in order:
\begin{itemize}
    \item If we have an embedding of groups $G_1\hookrightarrow G_2$, any set of solutions in $G_1$ translates to solutions in $G_2$.
    \item Because we impose $\sum_{i=1}^{2g} x_i=0$, relations always come in pairs, thus $\nss^G_\ud(X)$ is an integer:
    \[ \forall I\in \mathcal{I}^G_\ud(X)\,, \quad I^o\coloneqq \left([\![1,2g]\!] \setminus I \right)\in \mathcal{I}^G_\ud(X) \,.\]
\end{itemize}

The following bound on $\nss^G_{\ud,\max}$ is far from optimal, as we will see below:
\begin{lemma}\label{Lemma: Rough Bound on X^G_d}
    Let $G$ be an abelian group and $\ud=(d_1,\dots,d_n)\in \NN^n$ with $d_n>0$. We have
    \[ \nss^G_{\ud,\max}\leq \frac{1}{2d_n} \prod_{k=1}^n \binom{2d_k}{d_k} \,. \]
\end{lemma}
\begin{proof}
    Let $g=\deg \ud$. Suppose $\mathcal{I}\subset \Rel_\ud$ is a maximal set such that 
    \[ M_\mathcal{I} \coloneqq \bar{1}+\langle T_I \rangle_{I\in \mathcal{I}}\subset \ZZ^{2g} \]
does not contain $e_i-e_j$ for all $2d^{k-1}<i<j\leq 2d^{k}$, for all $1\leq k \leq n$. Let
\[\mathcal{I}'\coloneqq \{ I\in \mathcal{I}\,|\, 2g\in I\} \,.\]
note that $\mathcal{I}=\mathcal{I}'\sqcup \mathcal{I}'^o)$ where $I^0\coloneqq [\![1,2g]\!] \setminus I$. For all $I\in \mathcal{I}'$, and for all $i\in I^o\cap [\![2g-2d_n+1,2g]\!]$\,, let 
\[ \tau_{i} I \coloneqq \{i\}\cup I\setminus \{2g\}\in \Rel_\ud \,. \]
Suppose that $\tau_i I=\tau_j J$ for some $I,J\in \mathcal{I}'$, $i,j\in [\![2g-2d_n+1,2g]\!] $. We then have
\[ T_{\tau_i I}-T_{\tau_j J}=e_i-e_j \in M_\mathcal{I} \]
which is a contradiction. For $I\in \Rel_\ud$ we have $|I\cap [\![2g-2d_n+1,2g]\!]|=d_n$ thus by the pigeonhole principle, we have
\[ |\mathcal{I}|=2|\mathcal{I}'|\leq \frac{1}{d_n} |\Rel_\ud|=\frac{1}{d_n} \prod_{k=1}^n \binom{2d_k}{d_k} \,.\]
\end{proof}
The following table summarizes the results for $g=5$. It was obtained by solving question \ref{Question 2} with an algorithm first and then working out the minimal group and solution set by hand. We have chosen to represent the solution $X$ as a tuple of sets $X_k\subset G$ of cardinal $2d_k$, since the ordering of $X_k$ is irrelevant. The last column corresponds to the rank of the $\ZZ$-module $M_{\mathcal{I}_X}$ where $M_\mathcal{I}$ was defined above. We have shown the combinations of $G$ and $X$ that maximize $\nss^G_\ud(X)$ for some given $\ud$, but we also included some combinations of $G$ and $X$ that lead to lower values of $\nss^G_\ud(X)$.
\stepcounter{equation}
\begin{table}[h]\caption{Values of $\nss^G_\ud$ for $g=5$.}\label{Table: maximal eta values for give G and d}
$
\hspace*{-35pt}{
\begin{array}{|c|c|c|c|c|}
\hline
\underline{d} & \nss(X) & G & X=(X_1,\dots,X_n) & \mathrm{rank} M \\
\hline
(1,1,1,1,1) & 5 & \ZZ_3 & \{0,1\},\{0,1\},\{0,1\},\{0,1\},\{0,2\}&6\\
\hline 
(1,1,1,2) & 6 & \ZZ_4 & \{0,1\},\{0,2\},\{0,3\} ,\ZZ_4 & 7 \\
\hline 
(1,1,3) & 7 & \ZZ_6 & \{0,1\},\{3,5\},\ZZ_6 &8 \\
\hline
(1,2,2) & 8 & \ZZ_4 & \{1,3\},\ZZ_4,\ZZ_4 & 8 \\
\hline
(1,2,2) & 7 & \ZZ_5 & \{1,3\} , \ZZ_5\setminus \{2\} , \ZZ_5\setminus \{2\} & 8 \\ 
\hline  
(1,4) & 9 & \ZZ_8 & \{0,4\},\ZZ_8 & 9 \\
\hline 
(1,4) & 8 & \ZZ_2\times \ZZ_6 & \{(0,0),(0,4)\},G\setminus \{(0,2),(0,5),(1,1),(1,2)\} & 9 \\
\hline
(2,3)& 10 & \ZZ_6 & \{0,2,3,4\},\ZZ_6& 9 \\
\hline 
(2,3) & 9 & \ZZ_7 & \{1,2,3,6\}, \ZZ_7\setminus \{5\} & 9 \\
\hline 
(2,3) & 8 & \ZZ_2\times \ZZ_4 & \{(0,0),(0,3),(1,0),(1,1)\} , G\setminus \{(1,1),(1,3)\} & 9  \\
\hline 
(5) & 11 & \ZZ_2\times \ZZ_2 \times \ZZ_3 & G \setminus \{ (1,1,1),(0,0,1) \} & 10 \\
\hline 
(5) & 10 & \ZZ_{14} & \ZZ_{14}\setminus \{ 1,2,5,13\} & 10\\
\hline 
(5) & 9 & \ZZ_2 \times \ZZ_8 & \ZZ_2\times \{1,4,5,7\}\cup\{(0,0),(0,6)\} & 10\\
\hline
\end{array}
}
$
\end{table}

We will not treat the whole table but as an example, let us explain the case $\underline{d}=(5)$ with $10$. Let $G=\ZZ_{14}$, and let
    \[ \mathcal{M}\coloneqq \begin{pmatrix} 
    1&1&1&1&1&1&1&1&1&1\\
1&1&1&1&0&0&0&0&0&1\\
1&1&0&0&1&1&0&0&0&1\\
1&1&0&0&0&0&1&1&0&1\\
1&0&1&0&1&0&1&0&0&1\\
1&0&1&0&0&1&0&1&0&1\\
1&0&0&1&1&0&0&0&1&1\\
1&0&0&0&0&1&1&0&1&1\\
0&1&1&0&1&0&0&0&1&1\\
0&0&1&1&0&1&1&0&0&1\\
0&0&0&1&1&0&1&1&0&1 
\end{pmatrix}\,, \quad \text{and} \quad X\coloneqq 
\begin{pmatrix}
    9\\4\\11\\6\\7\\10\\3\\0\\8\\12
\end{pmatrix} \in G^{10} \,.\]
The rows of $\mathcal{M}$ generate the module $M_{\mathcal{I}_X}$. We have
\[ \mathcal{M} \cdot X = \,^t (0,\dots,0) \mod{14}\,.\]
The smith normal form of $\mathcal{M}$ is 
\[ 
\begin{pmatrix} 
        1&0&0&0&0&0&1&0&0&13\\
        0&1&0&0&0&0&0&0&0&9\\
        0&0&1&0&0&0&1&0&0&28\\
        0&0&0&1&0&0&0&0&0&24\\
        0&0&0&0&1&0&1&0&0&33\\
        0&0&0&0&0&1&0&0&0&19\\
        0&0&0&0&0&0&2&0&0&3\\
        0&0&0&0&0&0&0&1&0&14\\
        0&0&0&0&0&0&0&0&1&4\\
        0&0&0&0&0&0&0&0&0&35\\
        0&0&0&0&0&0&0&0&0&0
\end{pmatrix}\,,
\]
thus the rows of the above matrix are generators of $M_{\mathcal{I}_X}$. In particular, this module is of rank $10$. The other cases are treated similarly.
\par
\renewcommand{\arraystretch}{1.5}
We can now prove the following
\begin{lemma}\label{Lemma: max eta in dimension 5}
    The following are the maximal values of $\nss^G_\ud$ away from the diagonal $Z_\ud$ for $G=E$ an elliptic curve:
\[ \begin{array}{|c|c|c|c|}
\hline 
 \ud & (5) & (1,4) &(2,3 )  \\
 \hline 
 \nss^E_{\ud,\mathrm{max}} & 11& 9 & 10 \\ 
 \hline 
 \end{array}\,.
\]
Moreover, for all $\ud\in \{(5),(1,4),(2,3)\}$ and for all  $0\leq k \leq \nss^E_{\ud,\mathrm{max}}$, there exists a configuration of distinct points $X\in E^{2g}$ with $\nss^E_{\ud}(X)=k$. \par
The following are the maximal values of $\nss^{\CC^\ast}_\ud $ for $\deg \ud = 5$:
\[
\begin{array}{|c|c|c|c|c|c|c|c|}
\hline 
    \ud &(1,1,1,1,1) & (1,1,1,2) &(1,1,3) &(1,2,2)&(1,4)&(2,3)&(5) \\
    \hline 
   \nss^{\CC^\ast}_{\ud,\max }& 5 & 6&7&8&9&10&10\\
    \hline 
\end{array}\,.
\]
Moreover, for all $\ud$ of degree $5$ and for all $0\leq k \leq \nss^{\CC^\ast}_{\ud,\max}$, there exists a configuration of distinct points $X\in (\CC^\ast)^{10}$ with $\nss^{\CC^\ast}_{\ud}(X)=k$.
\end{lemma}
\begin{proof}
    The assertions for the maximal values of $\nss$ follow directly from Table \ref{Table: maximal eta values for give G and d}: all groups in the table embed in an elliptic curve, whereas some (like $\ZZ_2\times \ZZ_2\times \ZZ_3$) do not embed in $\CC^\ast$. \par
    Fix $\ud$ and $G\in \{E,\CC^\ast\}$. For $\mathcal{I}\subset \Rel_{\ud}$, let 
    \[ S_{\mathcal{I}} \coloneqq \{ X\in G^{10}\setminus Z_\ud \,|\, \sum_{i=1}^{10} x_i=\sum_{i\in I} x_i=0 \quad \forall I\in \mathcal{I} \} \subset G^{10}\setminus Z_\ud  \,.\]
    This is a closed locus in $G^{10}\setminus Z_\ud $ of codimension $\rank M_{\mathcal{I}}$ (if non-empty). Each pair of conjugate relations reduces the dimension by at most one. Thus, if the codimension is the expected one ($\nss(X)+1$), then for all $0\leq k \leq \nss(X) $, there exists a configuration $Y\in G^{10}\setminus Z_\ud $ with $\nss(Y)=k$. For the remaining values, there are explicit examples in the table.
\end{proof}
\begin{remark}
    It would be interesting to know if this property holds for higher $g$: Namely, for a given $\ud$ and say an elliptic curve $E$, can we always find a configuration $X\in E^{2g}$ with
    \[ \nss^E_\ud(X)=k \]
    for all $0\leq k \leq \nss^E_{\ud,\max}$? Somehow this seems a bit too optimistic.
\end{remark}
\subsection{Application to bielliptic Prym varieties}
  As we have seen in Section \ref{Sec: Construction of the Families Eg,t} for a Prym $(P,\Xi)\in \EE'_{g,t}$ we can construct a tower of curves (\ref{Diagramm: E_g,t tower}) and associate a smooth elliptic curve $E$, a line bundle $\delta\in \Pic^g(E)$ and two effective reduced divisors $\Delta'\in E_{2t}$, $\Delta''\in E_{2g-2g}$ verifying
    \[ \BO_E(\Delta'+\Delta'')=\delta^{\otimes 2} \,. \]
    Recall the definition
    \[ \nss^E_\delta(\Delta',\Delta'') \coloneqq \frac{1}{2}\mathrm{Card} \{(D',D'')\in E_{t}\times E_{g-t}\,|\, D^i\leq \Delta^i\,, \BO_E(D'+D'')=\delta \} \,. \]
    Fix an origin on $E$ so we can identify $E$ with $\Pic^0(E)$ and consider $E$ as a group. Let $x\in E$ such that
    \[ g\cdot x= \delta \,. \]
    Let 
    \[ X'=\tau_{x}^\ast \Delta' \,, \quad X''=\tau_{x}^\ast  \Delta''\,, \]
    where $\tau_{x}(y)=y+x$ is the translation by $x$. Then it is clear that $X=(X',X'')\in E^{2g}\setminus Z_{(t,g-t)}$ and
    \[ \nss^E_{(t,g-t)}(X',X'')=\nss^E_\delta(\Delta',\Delta'') \,. \]
    Conversely, given $X=(X',X'')\in E^{2g}\setminus Z_{(t,g-t)}$, one can do the construction in the other way and end up with a Prym $(P,\Xi)$ with the same $\nss^E_\delta$. The case of $\SE_\ud$ is similar after replacing the smooth elliptic curve with a cycle of $\PP^1$'s and after identifying 
    \[ \Pic^{\underline{0}}(E)=\CC^\ast \,. \]
    We thus have the following immediate corollary to Lemma \ref{Lemma: max eta in dimension 5}:
\begin{corollary}\label{Cor: possible values of eta for Egt and SEud}
    Let $0\leq t \leq 2$ and $(P,\Xi)\in \EE_{5,t}'$. Consider a corresponding tower of curves as in Section \ref{Sec: Construction of the Families Eg,t}. Let $E,\delta,\Delta',\Delta''$ be the corresponding elliptic curve, line bundle and ramification divisors. Then the maximal values of $\nss^E_\delta(\Delta',\Delta'')$ (for $(P,\Xi)$ varying in $\EE'_{g,t}$) are given by Lemma \ref{Lemma: max eta in dimension 5}, namely
    \[ \begin{array}{|c|c|c|c|}
\hline 
 (g,t) & (5,0) & (5,1) &(5,2)  \\
 \hline 
 \max_{\EE_{5,t}'} \nss & 11& 9 & 10 \\ 
 \hline 
 \end{array}\,.
\]
Similarly suppose $\deg \ud=5$ and $(P,\Xi)\in \SE_\ud$, then the maximal values of $\nss^E_\delta(\Delta)$ are given by
\[
\begin{array}{|c|c|c|c|c|c|c|c|}
\hline 
    \ud &(1,1,1,1,1) & (1,1,1,2) &(1,1,3) &(1,2,2)&(1,4)&(2,3)&(5) \\
    \hline 
   \max_{\SE_\ud}  \nss& 5 & 6&7&8&9&10&10\\
    \hline 
\end{array}\,.
\]
Moreover, for $0\leq t \leq 2$ (resp. $\deg \ud=5$) and all $0\leq k \leq \max_{\EE'_{g,t}} \nss$ (resp. $0\leq k \leq \max_{\SE_\ud} \nss$), there exist a $(P,\Xi)\in \EE'_{g,t}$ (resp. in $\SE_\ud$) with
\[ \nss^E_\delta(\Delta',\Delta'')=k \,. \]
\end{corollary}
We have the following corollary to Lemma \ref{Lemma: Rough Bound on X^G_d}:
\begin{corollary}\label{Cor: Gauss degree on EEg,1 greater than Jacobians}
    Let $g\geq 5$. Then for all $(P,\Xi)\in \EE'_{g,1}$ we have
    \[\deg \GG_\Xi > \binom{2g-2}{g-1}=\deg \GG (\mathcal{J}_g) \,. \]
\end{corollary}
\begin{proof}
    By Theorem \ref{Thm: degree Gauss Map on Egt, in general} and Lemma \ref{Lemma: Rough Bound on X^G_d} we have
    \begin{align*}
        \deg \GG_\Xi - \binom{2g-2}{g-1}&=2\binom{2g-2}{g-2}-2^{g-1}-2\nss^E_{(1,g-1),\max} \\
        &\geq 2\binom{2g-2}{g-2}-2^{g-1}-\frac{2}{g-1}\binom{2g-2}{g-1}\\
        &>0
    \end{align*}
    when $g\geq 6$. For $g=5$, we have by \ref{Cor: possible values of eta for Egt and SEud}
    \[ \deg \GG_\Xi \geq 94-18=76>70=\deg \GG(\mathcal{J}_5) \,.\]
\end{proof}

\subsection{A special ppav in arbitrary dimension}\label{Sec: App: A special ppav in arbitrary dimension}
In this section, we let consider $\ud=(1^g)=(1,\dots,1)\in \mathscr{P}_g$ and construct a very ``degenerate'' Prym in $\SE_{(1^g)}$. In the case $g=4$ we recover Varley's fourfold with this construction. In general, this gives a Prym with a very low Gauss degree. This is the lowest non-zero Gauss degree that we know of, apart from hyperelliptic Jacobians. \par 
Suppose first that $g$ is even. Let $E$ be the cycle of $g$ $\PP^1$'s. It is easy to see how to maximize $\nss^E_\ud$ in this situation. Let $E_1,\dots,E_g$ be the irreducible components of $E$. Identify each component with $\PP^1=\CC \cup \{\infty\}$, such that $E_i$ meets $E_{i-1}$ at $0$ and $E_{i+1}$ at $\infty$. For $1\leq i \leq n$ let $Q^1_i$ (resp. $Q^{-1}_i$) be the point in $E_i$ corresponding to $1$ (resp. $(-1)$) under this identification. We define
\[ \Delta\coloneqq \sum_{i=1}^g Q_i^1+Q_i^{-1} \,, \quad \delta \coloneqq \BO_E(Q_1^1+\cdots+Q_g^1)\in \Pic^{\ud}(E) \,. \]
By Proposition \ref{Prop: Boundary: construction of the covering from E and a branch divisor} and Section \ref{Sec: Preliminaries boundary of EEgt} we can associate to this a Prym variety $(P,\Xi)\in \SE_{\ud}$. We then have
\begin{proposition}
    The Prym $(P,\Xi)$ constructed above minimizes the Gauss degree on $\SE_{(1^g)}$ with $g$ even. We have
    \[ \nss^E_\delta(\Delta)=2^{g-2} \,, \quad \text{and} \quad \deg \GG_\Xi = g \binom{g-1}{ g/2} - 2^{g-1} \,. \]
\end{proposition}
\begin{proof}
    It is clear that $\nss^E_\delta(\Delta)=2^{g-2}$. Indeed, if we want to find a configuration of points in $\Delta$ whose associated line bundle is $\delta$, we can choose the first $(g-1)$ points arbitrarily among $\{Q_i^1,Q_i^{-1}\}$, and then we only have one choice for the last point. The Gauss degree then follows from Theorem \ref{Thm: Degree Egt with E cycle of P1's}. Finally, it is clear that this maximizes $\nss^E_\ud$: Suppose that there exists a configuration $\Delta=Q_1+Q_1'+\cdots Q_g+Q_g'\in E_{2\ud}\setminus Z_\ud$ (recall $Z_\ud$ is the diagonal defined at the beginning of this section) with strictly more than $2^{g-1}$ relations. Then by the pigeonhole principle we would have
    \[ D+Q_i\sim \delta \sim D+Q_i'\]
    for some $D\subset \Delta $ and $i$. This is impossible since $Q_i\neq Q_i'$
\end{proof}
\begin{remark}
    When $g=4$, the above gives a ppav with Gauss degree $4$. We thus recover Varleys fourfold! Indeed, Varleys fourfold is the unique ppav in $\mathcal{A}_4$ with Gauss degree $4$ \cite{Gru17}. In a sense, the construction above could be seen as a generalization of Varleys fourfold. One could ask if this is the smallest achievable non-zero Gauss degree besides the hyperelliptic locus $\mathcal{H}_g$ (indeed, an easy computation shows that this degree is higher than the degree on $\mathcal{H}_g$ as soon as $g\geq 5$).
\end{remark}
We have the following construction when $g$ is odd: repeat the same procedure but take $Q_i^1$ (resp. $Q_i^{-1}$) to be the point above $j=e^{2i\pi/3}$ (resp. $j^2$) instead. We have
\begin{equation}\label{Equ: Add sing, number of sing of special ppav in odd dimension}
    \nss^E_\delta(\Delta)=\frac{1}{3}(2^{g-1}-1)\,.
\end{equation} 
Indeed, let $G=\ZZ_3$, $X=\{1,-1\}\subset \ZZ_3$ and for $g\geq 1$, let
\begin{align*} 
u_g&\coloneqq  \Card \{ x=(x_1,\dots,x_g)\in X^g \,|\, x_1+\cdots+x_g = 0 \} =\nss^E_\delta(\Delta) \,, \\
v_g &\coloneqq \Card \{ x\in X^g \,|\, x_1+\cdots+x_g=1 \} =\Card \{ x\in X^g\,|\, x_1+\cdots+x_g=-1 \}\,. 
\end{align*}
We have $u_g+2v_g=2^g$ and $u_{g+1}=2 v_g$. From this \ref{Equ: Add sing, number of sing of special ppav in odd dimension} follows. The corresponding Prym $(P,\Xi)$ thus has Gauss degree
\[ \deg \GG_\Xi= g \binom{g-1}{\lfloor g/2 \rfloor}- \frac{2^g-2}{3} \,. \]
We do not know whether this is the minimum on $\SE_\ud$.

\printbibliography

\end{document}